\documentclass[12pt]{article}
\usepackage{latexsym,color,amsmath,amsthm,amssymb,amscd,amsfonts}
\usepackage{scalefnt}
\usepackage[font=small,labelfont=bf]{caption}
\usepackage{float}
\usepackage{graphicx}
\usepackage{amsopn}
\usepackage{relsize}
\usepackage{mathrsfs}
\usepackage{upgreek}
\usepackage{tikz}

\newtheoremstyle{nonum}{}{}{\itshape}{}{\bfseries}{.}{ }{\thmnote{#3}}

\newtheorem{thm}{Theorem}[section]
\newtheorem*{thm*}{Theorem}

\newtheorem{cor}[thm]{Corollary}
\newtheorem{lem}[thm]{Lemma}
\newtheorem{prop}[thm]{Proposition}
\newtheorem{rem}[thm]{Remark}
\newtheorem{conj}[thm]{Conjecture}
\newtheorem{exm}[thm]{Example}
\newtheorem{definition}[thm]{Definition}
\newtheorem*{definition*}{Definition}
\newtheorem{fact}[thm]{Fact}

\newtheorem*{rems*}{Remarks}
\newtheorem{rems}[thm]{Remarks}

\theoremstyle{nonum}

\newtheorem*{definition-recall}{Definition \ref{def:klens}}

\newcommand{\R}{\mathbb R}
\newcommand{\RR}{\mathbb R}

\newcommand{\N}{\mathbb N}

\def\S{{\cal S}}

\newcommand{\iprod}[2]{\langle #1,#2 \rangle} 

\def\vol{{\rm Vol}}

\def\cvx{{\rm Cvx}}

\def\eps{{\varepsilon}}

\def\conv{{\rm conv}}
\def\cconv{{\rm conv}_c}
\def\outrad{{\rm Outrad}}
\def\diam{{\rm diam}}
\def\inrad{{\rm Inrad}}

\def\bas{{{\rm Basin}}}

\begin{document}
\title{An in-depth study of ball-bodies}
\date{}
\author{S. Artstein-Avidan, D.I. Florentin}
\maketitle
\begin{abstract}
In this paper we study the class of so called ``ball-bodies'' in $\RR^n$, given by intersections of translates of Euclidean unit balls. We study the class along with the natural duality operator defined on it, called $c$-duality. The class is naturally linked to many interesting problems in convex geometry, including bodies of constant width and the Knesser-Poulsen conjecture. We discuss old and new inequalities of isoperimetric type and of Santal\'{o} type, in this class. We study the boundary structure of bodies in the class, Carath\'eodory type theorem and curvature relations. We discuss various symmetrizations with relation to this class, and make some first steps regarding problems for bodies of constant width.   
\end{abstract}

\section{Definitions and First Observations}\label{sec:first}

In this project we study a special class of convex bodies in $\RR^n$, which we denote by ${\cal S}_n$ and call ball-bodies. We will give several equivalent definitions of this class in what follows. In the literature they are sometimes referred to as ``ball bodies'', or ``spindle-convex'' bodies, or as $\lambda$-convex for $\lambda = 1$. By homogeneity all the results easily translate to $\lambda$-convex bodies with any other  parameter $\lambda$ instead of $1$, and we fix $\lambda = 1$ for simplicity of the presentation. 
In what follows we use $\iprod{\cdot}{\cdot}$ to denote the standard inner product on $\RR^n$, we use $\|y\|_2=\sqrt{\iprod{y}{y}}$ to denote the Euclidean norm and $B(x,r) = \{ y: \|y-x\|_2 \le r\}$ to denote the closed Euclidean ball of radius $r$ centered at $x$, and $S(x,r)= \partial B(x,r)$ its boundary. We sometimes use $B_2^n$ instead of $B(0,1)$ for the unit Euclidean ball centered at the origin and use $S^{n-1}$ instead of $S(0,1)$.  

\begin{definition}\label{def:Sn}
Let $n\ge 1$. A set $K\subset \RR^n$ is called a ball-body if there exists some subset 
$A \subset \R^n$ such that 
\[ K = \bigcap_{x\in A} B(x,1). \]
The class of all ball-bodies bodies in $\RR^n$ is denoted by $\S_n$. 
\end{definition}

If $\outrad(A)>1$ then the intersection is empty, thus $\emptyset\in \S_n$. When $A=\emptyset$, the ``empty intersection'' is defined to be $\RR^n$. These two degenerate sets will sometimes be omitted when we discuss properties of sets in $\S_n$.

There are various other ways to describe the class $\S_n$, as we shall demonstrate shortly. We first mention that Definition \ref{def:Sn} corresponds to a description of ${\cal S}_n$ as an image class for an order reversing quasi involution, that is, as the image of a mapping $A\mapsto A^c$ on subsets of $\RR^n$ which reverses the partial order of inclusion and satisfies $A\subseteq A^{cc}$. Such mappings (see \cite{artstein2023zoo}) and their image class have structural properties which will play an important role in this note. 

\begin{definition}
Let $n\ge 1$. For $A\subset \RR^n$, its $c$-dual is defined to be 
	\[
	A^c = \{ y: \forall x\in A, d(x,y) \le 1\}  = \bigcap_{x\in A} B(x,1), 
	\]
and its $c$-hull, denoted $\cconv(A)$ is defined to be 
	\[ 	\cconv (A) : = A^{cc}.\]
	\end{definition}

The mapping $A\mapsto A^c$ is an order reversing quasi involution. Indeed, it reverses order since if $A_1\subseteq A_2$ then $A_2^c\subseteq A_1^c$, as we intersect more balls. The fact that $A\subseteq A^{cc}$ is also immediate, since if $x\in A$ then $x\in B(y,1)$ for any $y\in A^c$ which means that $y\in B(x,1)$ for any
$y\in A^c$ which means $x\in A^{cc}$. The class $\S_n$ is by definition the image of the $c$-duality transform. Therefore (see \cite{artstein2023zoo}) we immediately see that if $K\in {\cal S}_n$ then $K^{cc} = K$ (namely, on ${\cal S}_n$ the $c$-duality is an order reversing involution) and furthermore, for any $A\subset \RR^n$ the set $A^{cc}$ is the smallest member of ${\cal S}_n$ containing $A$, which motivates the name ``$c$-hull''. If no compact set in $\S_n$ includes $A$, which is the case when $\outrad(A) >1$, we get $\cconv (A) = \RR^n$. It is useful to note that the $c$-hull can also be understood as the intersection of all $1$-balls that contain $A$.

\begin{rem}
For every $A\subset \R^n$,
\[
A^{cc} = \bigcap_{\{x: A\subseteq B(x,1)\}} B(x,1) = \bigcap_{\{K\in \S_n : A\subseteq K\} }K . 
\]
Indeed, for the first equality
\begin{eqnarray*}
A^c = \bigcap_{x\in A} B(x,1) = 
\{ y: \forall x\in A,\,\, x\in B(y,1)\} =
\{ y: A\subseteq B(y,1)\},  
\end{eqnarray*}
therefore 	
\begin{eqnarray*}
A^{cc} = \bigcap_{y\in A^c} B(y,1) = \bigcap_{\{ y: A\subseteq B(y,1)\}} B(y,1),  
\end{eqnarray*}
as claimed. The second equality is trivial by order reversion of the $c$-duality. Indeed, $A\subseteq K$ implies $A^{c}\supseteq K^c$ and thus $A^{cc}\subseteq K^{cc} = K$ so that $A^{cc}$ is a subset of the intersection. However clearly $A^{cc}$ includes the intersection, either by our first assertion or simply since $A^{cc}\in {\cal S}_n$ and includes $A$. This completes the demonstration. 
\end{rem}

After the dual pair $\emptyset$ and $\R^n$, the next simplest members of $\S_n$ are the dual pairs $\{x\}$ and $B(x,1)$, for any $x\in \RR^n$. More generally, the $c$-dual of $B(x,r)$ is $B(x,1-r)$, for any $r\in [0,1]$. We mention that for any $A\subseteq B(x,1)$ with out-radius $1$ we have $A^{c} = \{x\}$ and $A^{cc} = B(x,1)$, see Lemma \ref{lem:inplusout}. It is also easy to check that the $c$-duality commutes with rigid motions, namely for $g(x)=x_0 + Ux$ for $U\in O(n)$ we have 
\[ (g(A))^c =
\bigcap_{x\in A} B(x_0 + Ux,1) =
x_0 + U \bigcap_{x\in A}   B(x,1) = g (A^c).
\]

We introduce some special sets in ${\cal S}_n$ which play a prominent role in this survey. A non-empty intersection of two $1$-balls is called a lens, and up to translation and rotation, it is determined by the distance between the two centers of the $1$-balls which are intersected. The dual of such a set is the $c$-hull of the centers of the two balls. It can be equivalently defined as the body of revolution of a $1$-arc connecting these two centers (with axis of revolution along the segment connecting the two centers). We call the $c$-hull of two points (namely of a ``sphere'' in $\RR$) a $1$-lens, and note that the $c$-hull of an $(n-1)$-dimensional sphere is precisely a lens, so a lens will be called an $(n-1)$-lens. More generally, we define a $k$-lens as follows.

\begin{definition}\label{def:klens}
Let $n\ge 2$, $1\le k\le n$, let $E\subset  \R^n$ be some $k$-dimensional  subspace, let $d\in [0,1]$ and let $x\in \RR^n$. The $k$-lens about $x$ of ``radius'' $d$ is defined to be $A^{cc}$ for $A = S(x,d) \cap (x+E)$, and is denoted by 
	$
	L_k(x,E,d)$.  
\end{definition}
\noindent It can be checked that  $L_{n-1}(x,   u^{\perp}, \sqrt{1-d^2}) = B(x+du, 1) \cap B(x-du, 1)$, and $L_1(x,  \RR u, d) = \cconv (x+du, x-du) $ is its dual.  More generally $k$-lenses are mapped by the $c$-duality to $(n-k)$-lenses, for any $k<n$. For more details, as well as a proof of this fact, see Appendix \ref{appendix-examples}.
 
An equivalent definition for ${\cal S}_n$ is captured in the following proposition.
\begin{prop}\label{prop:ball and spindle convexity}
Let $K\subset \RR^n$, then $K\in {\cal S}_n$ if and only if for any $x_0, x_1\in K$ we have $ \{x_0, x_1\}^{cc} \subseteq K$.
\end{prop}

In other words, and in analogy to the fact that convex bodies are characterized by the fact that along with any two points $x_0 ,x_1$, they contain the segment joining them (their convex hull $[x_0, x_1]$), bodies in $\S_n$ are  characterized by the fact that along with any two points $x_0 ,x_1$, they contain their $c$-hull. 
This simple fact was noted in many places, for example \cite{KupitzMartiniPerles,Bezdek2007}. It is interesting to note that when the intersection of translates of Euclidean balls is replaced by the intersection of translates of some other convex body, the equivalence no longer holds,  see \cite{Bezdek2007}, and one has to consider two (dual) families, one corresponding to ``ball bodies'', namely intersections of translates of the original body, and the other corresponding to ``spindle convex bodies'', namely those which include, with any two points inside them, their respective ``hull''.  While ball-bodies are always spindle convex, the opposite is not true in general. In this paper we only consider intersections of Euclidean balls, but many of the questions can be addressed in the geneal settings, and related work has been done, see  for example \cite{JahnMartiniRichter2017}, and \cite{MARYNYCH2022108086} for some applications. For the convenience of readers we include a proof of the above proposition.

\begin{proof}[Proof of Proposition \ref{prop:ball and spindle convexity}]
	One direction is immediate, if $K\in \S_n$ and $\{x_0, x_1\}\subseteq K$ then $K^c\subseteq \{x_0, x_1\}^c$ so that  
$\{x_0, x_1\}^{cc} \subseteq K^{cc} = K$. In the other direction, if for any $x_0, x_1\in K$ we have $\{x_0, x_1\}^{cc} \subseteq K$, then in particular $[x_0,x_1] \subset K$ i.e. $K$ is convex. Suppose towards a contradiction that there exists $s\in K^{cc}\setminus K$. Let $H=\left\{y: \iprod{u}{y} = d\right\}$ be a hyperplane tangent to $K$ at some $x\in K\cap H$, separating $K$ from $s$, i.e. $K\subset H^+=\left\{y: \iprod{u}{y} \ge d\right\}$, and $s\in H^-=\left\{y: \iprod{u}{y} < d\right\}$. We claim that $K\subseteq B(x+u,1)$. Indeed, if $y\in H^+\setminus B(x+u,1)$ then either $\|x-y\|_2>2$ (i.e. $\{x, y\}^{cc} \subseteq \R^n$), or the $1$-arc connecting $x$ and $y$ which lies in the affine span of $x,y,x+u$, intersects $H^{-}$. 
In both cases, $\{x, y\}^{cc}\cap H^- \neq\emptyset$. Thus $x+u\in K^c$ but $\|(x+u)-s\|_2 >1$ which contradicts $s\in K^{cc}$.
\end{proof}

In some sense, the class ${\cal S}_n$ can be thought of as an analogue to the class of convex bodies, when half-spaces (the intersections of which produce the class of convex bodies) are replaced by unit Euclidean balls.  
The following proposition makes the analogy more precise. 

\begin{prop}\label{prop:boundary-points-have-friends}
Let $K\subset \RR^n$, $K\neq \emptyset, \RR^n$. Then $K\in {\cal S}_n$ if and only if $K$ is convex and for any $x\in \partial K$ there exists $y\in K^c$ with $\|x-y\|_2 = 1$.\\
Moreover, in this case $y\in \partial K^c$, $x-y \in N_K(x)$ and $y-x\in N_{K^c}(y)$. Here, $N_T(z)$ denotes the set of unit outer normals to $T$ at $z$, namely the intersection of the outer normal cone of a convex body $T$ at a boundary point $z$ with the sphere. 
\end{prop}

\begin{proof}
	Assume $K\in \S_n\setminus \{\emptyset, \RR^n\}$ (so it is obviously convex) and let $x\in \partial K$. Since  $K = (K^c)^c$ and since there are points $x_k\not\in K$ with $x_k\to x$ there is a sequence $y_k \in K^c$ with $\|y_k - x_k\|_2 > 1$ and $\|y_k - x\|_2\le 1$ so that by the triangle inequality $\|y_k - x\|_2\to 1$ and by compactness we find some $y\in K^c$ with $\|x-y\|_2 = 1$. Since $x\in K$ we see that $y\in \partial K^c$. 

Next assume $K$ satisfies the condition in the statement of the proposition and take all $x\in \partial K$ and their corresponding $y(x) \in \partial K^c$. Intersect all the corresponding balls $B(y(x), 1)$. The result is a convex set $L\in \S_n$ which clearly contains $K$ as it contains $K^{cc}\supseteq K$. On the other hand, it cannot contain any point which is not in $K$ since then some boundary point $x\in\partial K$ would be in the interior of $L$, contradicting the fact that the ball $B(y(x), 1)$, participating in the intersection, included $x$ as a boundary point. Thus $K=L\in\S_n$ and the proof of the first part is complete.
	
Moreover, note that if $K\subset B(y,1)$ and the two bodies are touching at a boundary point $x\in \partial K\cap \partial B(y,1)$, then the normal  of  $B(y,1)$ at the point $x$, which is $x-y$, is also a normal of $K$ at $x$. Since $K = K^{cc}$, the same argument holds for $y$.   
\end{proof}

\begin{rem}
    A strengthening of Proposition \ref{prop:boundary-points-have-friends} is given in Lemma \ref{lem:normals map boundary to boundary}, where we show that   $u\in N_K(x)$, if an only if $y = x-u \in \partial K^c$. \end{rem}

It turns out that $\S_n$ is closed under Minkowski averages, and that projections and sections of elements in $\S_n$ with a lower dimensional subspace (say of dimension $k$) belong to the corresponding class ${\cal S}_k$. We next show the closed-ness with respect to sections. For projections and for Minkowski addition we defer the proof until after we have presented another useful description of $\S_n$ and the duality,  and these appear in Theorem \ref{thm:minkowski-sum-c-is-linear} and in Corollary  \ref{cor:cosed-under-projection} below.

\begin{lem}\label{lem:cosed-under-section} Let $K\in {\cal S}_n$ and let $H$ be a hyperplane which we identify with $\RR^{n-1}$. Then $K\cap H\in \S_{n-1}$.
\end{lem}

\begin{proof}
	
Assume $K = \cap_{x\in A} B(x, 1)$ and let $H$ be some affine hyperplane. Letting $d_x = d(x,H)$ we note that $B(x,1) \cap H = B_H(P_Hx, \sqrt{1-d_x^2})$ is a ball in $H$ of some radius at most $1$, and in particular  belongs to $\S_{n-1}$. Therefore, as $\S_{n-1}$ is closed under intersections,  
\[ K\cap H = \bigcap_{x\in A} B(P_Hx, \sqrt{1-d_x^2}) \in \S_{n-1}.\]
\end{proof}

\begin{rem}
	We mention that the  $c$-hull does not commute with Minkowski average, in contrast to standard convex hull. However, for $X,Y\subset \RR^n$ it holds that 
	\[
	\cconv\left(\frac{X+Y}{2}\right) \subseteq 
	\frac{\cconv(X) + \cconv(Y)}{2},
	\]
since, as we shall see in Theorem \ref{thm:minkowski-sum-c-is-linear}, the right hand side is a set in $\S_n$ which contains $\frac{X+Y}{2}$. To see that in certain cases this inclusion can be strict one may consider $X = \{\pm e_1\}$ and $Y = \{\pm e_2\}$. In this case, $\cconv(X) = \cconv (Y) = B(0,1)$, but $\frac{X+Y}{2}$ is a set of out-radius $1/\sqrt{2}$ (it consists of the vertices of the centered square of side length $1$) and its $c$-hull is thus clearly a   subset of $B(0,1/\sqrt{2})$.  
\end{rem}

Let us present yet another description of $\S_n$. Sets in $\S_n$ (excluding $\emptyset$ and $\R^n$) are precisely summands of $B(0,1)$, namely $K\in \S_n\setminus \{\emptyset, \RR^n\}$ if and only if there is some convex $T\subset \R^n$ such that $K+T=B(0,1)$. Moreover, in this case $T=-K^c$. To see this, let us first recall some definitions and results from convexity.

\begin{definition}
We say that a compact convex set $K\subset \R^n$ slides freely inside $B(0,1)$, if for every $x\in \partial B(0,1)$ there exists $y \in \R^n$ such that $x \in y + K \subseteq B(0,1)$.
\end{definition}

The following theorem regarding bodies sliding freely in $B(0,1)$ is well known (see e.g. \cite[Theorem 3.2.2]{schneider2013convex}). 

\begin{thm}\label{thm:sliding-iff-summand}
Let $K \subset \RR^n$ be a convex body. Then $K$ {\em slides freely} inside $B(0,1)$ if and only if $K$ is a summand of $B(0,1)$. Moreover, if $K+T=B(0,1)$ then
\begin{equation}\label{eq:intersections-over-your-summand}
K = \bigcap_{-x\in T} B(x,1).
\end{equation}
\end{thm}

Theorem \ref{thm:sliding-iff-summand} implies that if $K$ is a summand of $B(0,1)$ then $K\in \S_n\setminus \{\emptyset, \RR^n\}$.
Moreover, Proposition \ref{prop:boundary-points-have-friends} implies that if $K\in \S_n\setminus \{\emptyset, \RR^n\}$ then $K$ slides freely inside $B(0,1)$, which by Theorem \ref{thm:sliding-iff-summand} implies $K$ is a summand of $B(0,1)$. This observation is attributed to Maehara, who proved (see \cite[Theorem 3.2.5]{schneider2013convex}) the following.
\begin{thm} Any nonempty intersection of translates of Euclidean
unit balls in $\RR^n$
is a summand of $B(0,1)$.
\end{thm}

\begin{rem}
We remark here that, of course, one can discuss summands of some other fixed convex body $M\subset \RR^n$ in place of $B(0,1)$, namely pairs of convex bodies $K$ and $L$ such that $K+L = M$. In such a case, the summand $K$ is of the form $K = \bigcap_{x\in A} (M-x)$, however it need not be the case that every intersection of translates of $M$ is a summand of $M$, when $n\ge 3$  
(the cross-polytope $L = B_1^3$ is a counterexample). See \cite[Section 3.2]{schneider2013convex} for more details. 
\end{rem}

The following proposition follows from the above discussion, and we add yet another proof. We use $h_K:\RR^n \to \RR$ to denote, as usual, the support function of a set $K$, that is, $h_K(u)=\sup_{x\in K}\iprod{x}{u}$.

\begin{prop}\label{prop:KminusK}
Let $K\in \S_n\setminus \{\emptyset, \RR^n\}$. Then $K - K^c = B(0,1)$, or, equivalently, for all $u\in \S^{n-1}$
	\[ h_{K^c} (u) = 1- h_K(-u).\]
	In particular $K + K^c$ is a body of constant width $2$.\\
Moreover, bodies of constant width $1$ are precisely the fixed points of the $c$-duality. 
\end{prop}

\begin{rem}
   The study of bodies of constant width was one of the motivations for studying $\S_n$, and we address them   in Section \ref{sec:constantwidth}.  
\end{rem}

\begin{proof}
Let $u\in S^{n-1}$ and consider the unique point $x\in \partial K$ such that $u\in N_K(x)$. By Proposition \ref{prop:boundary-points-have-friends}, $y = x-u\in \partial K^c$, $-u \in N_{K^c}(y)$ and $\|x-y\|_2 = 1$. Therefore
\begin{eqnarray*}
h_{K-K^c}(u)  = h_{K}(u) + h_{K^c}(-u) 
= \iprod{x}{u} + \iprod{y}{-u} = \iprod{x-y}{u} = \|u\|_2^2 = 1 
\end{eqnarray*}	
For the second assertion, note that \begin{eqnarray*} 
		h_{K + K^c} (u) + h_{K+ K^c}(-u) &=& h_K(u) + h_{K^c}(u) +
		h_K(-u) + h_{K^c}(-u)\\& =&   h_K(u) + h_{-K^c}(u) +
		h_{K}(-u) + h_{-K^c}(-u)\\& = &h_{K-K^c}(u) +  h_{K-K^c}(-u) = 2.
	\end{eqnarray*}
Finally, if $K$ is a body of constant width $1$ then $K-K = B(0,1)$ namely $K$ is a summand of the ball, so $K\in \S_n\setminus \{\emptyset, \RR^n\}$, thus $K-K^c = B(0,1)$, by the first assertion we have just shown. Combining the two, we get $K = K^c$.
\end{proof}

With this characterization of the $c$-duality and the class $\S_n$, we can easily prove some useful properties. The first follows from the classical Brunn-Minkowski inequality  $\vol_n(A+B)^{1/n} \ge \vol_n(A)^{1/n} + \vol_n(B)^{1/n}$ and Proposition \ref{prop:KminusK}. (We call it Santal\'o-type since it bounds from above the volume of the $c$-dual, similarly to the way Santal\'o-'s inequality bounds the volume of the polar body, \cite{Santalo1949}, see also \cite[Section 1.5.4]{AGMBook}.)

\begin{lem}[A Santal\'o-type inequality]\label{lem:santalo}
For any $K\in \S_n \setminus \{\emptyset, \RR^n\}$ it holds that
\[
\vol(K)^{1/n} + \vol(K^c)^{1/n}
\le
\vol(B(0,1))^{1/n},
\]
with equality if and only if $K$ is some ball $B(x,r)$. In particular for any $K$ of constant width $1$ we have  $\vol(K)  \le \vol(B(0,1/2))$. 
\end{lem}
This lemma was well known 
and admits many different proofs, some of which we will encounter in this text. In fact, let us present another simple one. 
\begin{proof}[Another proof for Lemma \ref{lem:santalo}]
The mean width of a convex body $K$ is defined by $M^*(K) = \int_{S^{n-1}}h_K(u) d\sigma(u)$ (where $\sigma$ is the normalized Haar measure on the sphere),
and Urysohn's inequality states that $\vol(A)^{1/n} \le  M^*(A)\cdot\vol(B(0, 1))^{1/n}$, with equality only for Euclidean balls (see \cite[Theorem 1.5.11]{AGMBook}).
The fact that  $h_K(u) + h_{K^c}(-u) = 1$ for $u\in S^{n-1}$ implies that $M^*(K) + M^*(K^c) = 1$, and applying Urysohn to both $K$ and $K^c$, we get $\vol(K)^{1/n} + \vol(K^c)^{1/n} \le \vol(B(0,1))^{1/n}$ with equality only for balls, as required.
\end{proof}

Since the Brunn-Minkowski inequality holds also for mixed volumes $V_k$ (see  
\cite[Section 1.1.5 and Appendix B]{AGMBook}) we can also show 
\begin{prop}[A Santal\'o-type inequality]\label{lem:santalo-mixed}
    For any $K\in \S_n \setminus \{\emptyset, \RR^n\}$ and $k\in \{1,\ldots, n\}$ it holds that
\[
V_k(K)^{1/k} + V_k(K^c)^{1/k}
\le
V_k(B(0,1))^{1/k},
\]
with equality if and only if $k=1$ or $K$ is some ball $B(x,r)$. In particular for any $K$ of constant width $1$ we have  $V_k(K)  \le V_k(B(0,1/2))$. 
\end{prop}

\begin{proof}
    This is an immediate consequence of the Brunn-Minkowski inequality for mixed volumes, together with its equality cases. 
\end{proof} 

\begin{rem}
In fact, by Brunn-Minkowski we have
\[ \vol(K)^{1/n} + \vol(K^c)^{1/n} \le \vol(K+K^c)^{1/n}, \]
and similarly for $V_k$. This bound is generally better than the bound in Lemma \ref{lem:santalo} since $K+K^c$ is of constant width $2$ and thus has  volume (or $V_k$) at most that of the Euclidean ball. 
\end{rem} 

The next result we demonstrate is a linearity result for the $c$-duality with respect to Minkowski averaging, which we find quite surprising, even though it follows immediately from the summands point of view of $\S_n$.

\begin{thm}\label{thm:minkowski-sum-c-is-linear}
Let $K,T\in \S_n\setminus \{\emptyset, \RR^n\}$,  and $\lambda \in (0,1)$. Then $(1-\lambda)K + \lambda T\in \S_n$ and  \[ ((1-\lambda)K + \lambda T)^c = (1-\lambda)K^c + \lambda T^c.\]
\end{thm}

\begin{proof}
Since $K,T\in \S_n\setminus \{\emptyset, \RR^n\}$, we know by Proposition \ref{prop:KminusK} that they are summands of the ball and that $K - K^c = T- T^c = B(0,1)$. 
	Therefore 
	\[ (1-\lambda) K - (1-\lambda) K^c + \lambda T- \lambda T^c = B(0,1),\]
	which implies that $(1-\lambda)K + \lambda T$ is a summand of the ball, hence in $\S_n$, and moreover, its dual is precisely $(1-\lambda)K^c + \lambda T^c$. 
\end{proof}

\begin{rem}
    The case $\lambda = 1/2$ and $K = -T$ was observed in \cite[Lemma 15]{Bezdek2021}, with a very different proof.
\end{rem}

Theorem \ref{thm:minkowski-sum-c-is-linear} has some immediate consequences.

\begin{cor}\label{cor:minkowski-sum-c-is-sub-linear}
Let $K,T\subset \RR^n$ be non-empty sets with $\outrad(K), \outrad(T)\le 1$ and let $\lambda \in (0,1)$. Then    \[ ((1-\lambda)K + \lambda T)^c \supseteq  (1-\lambda)K^c + \lambda T^c.\]
\end{cor}
\begin{proof}
Indeed,
\[ (1-\lambda)K + \lambda T  \subseteq  (1-\lambda)K^{cc} + \lambda T^{cc}.\]
Note that $K^{cc}, T^{cc}\in \S_n\setminus \{\emptyset, \RR^n\}$ (since $K$ and $T$ are non-empty, and have out-radius at most $1$). Thus by Theorem \ref{thm:minkowski-sum-c-is-linear}, the $c$-dual of their Miknowski average is given by $\left((1-\lambda)K^{cc} + \lambda T^{cc}\right)^c =  (1-\lambda)K^{c} + \lambda T^{c}$.
Since $c$-duality reverses inclusion we get 
\[ \left((1-\lambda)K + \lambda T \right)^c \supseteq  \left((1-\lambda)K^{cc} + \lambda T^{cc}\right)^c =  (1-\lambda)K^{c} + \lambda T^{c},\]
as claimed. 
\end{proof}

Another consequence of Theorem \ref{thm:minkowski-sum-c-is-linear} deals with the Minkowski symmetral of a (convex) body $K$ with respect to a subspace $u^\perp$, defined as $M_u K = \frac{1}{2}\left(K + R_uK\right)$, where $u\in S^{n-1}$, and $R_u(x) = x- 2\iprod{x}{u}u$ is reflection with respect to $u^\perp$.
\begin{cor}\label{cor:closed-under-Mink}
Let $u\in S^{n-1}, K\in \S_n$. Then $M_u K\in\S_n$, and $M_u(K^c) = (M_u K)^c$. Moreover, $\vol(M_u K) \ge \vol(K)$, and $\vol(M_u K ^c) \ge \vol(K^c)$.
\end{cor}

\begin{proof}
Clearly $R_u K\in \S_n$, and $R_u(K^c) = (R_u K)^c$, as $c$-duality commutes with rigid motions. 
Thus, the first two claims follow from Theorem \ref{thm:minkowski-sum-c-is-linear}, and the last two claims follow from the Brunn-Minkowski inequality.
\end{proof}

Our last consequence of Theorem \ref{thm:minkowski-sum-c-is-linear} deals with orthogonal projections onto lower dimensional subspaces.
\begin{cor}\label{cor:cosed-under-projection} Let $K\in {\cal S}_n$ and let $E\subset \RR^n$ be a $k$-dimensional subspace. Then $P_E K\in \S_{k}$, where $P_E:\RR^n\to E$ is the orthogonal projection onto  $E$. Moreover, $P_E(K^c )= (P_EK)^c$ where 
on the right hand side, the $c$-duality is understood as intersections of $1$-balls in $E$. 
\end{cor}

\begin{proof}	
If $K\in \S_n\setminus \{\emptyset, \RR^n\}$ then $K-K^c=B(0,1)$, by Proposition \ref{prop:KminusK}. Since $P_E$ is a linear map, it commutes with Minkowski sum, and we get
\[
P_E K  - P_E(K^c)  = P_E(B(0,1)) = B_E(0,1).
\] 
This shows that $P_EK$ is a ball-summand, and its $c$-dual in $E$ is $P_E(K^c)$. 
\end{proof}
In the case $k=n-1$ of Corollary \ref{cor:cosed-under-projection}, since $P_E K = M_u(K) \cap E$ for $u^\perp = E$, we could have used Corollary \ref{cor:closed-under-Mink} together with Lemma \ref{lem:cosed-under-section} to get that the projection $P_E K$ is in the class. The attentive reader may have also noticed that a slight sharpening of Proposition \ref{prop:boundary-points-have-friends} would allow for a direct proof for the fact that $\S_n$ is closed under sections and projection, without the use of ball summands. Indeed, the proposition stated that for $K\in {\cal S}_n$ and $x\in \partial K$, one can always find a normal $u\in N_K(x)$ such that  $y = x-u\in K^c$ (and in such a case $-u\in N_{K^{c}}(y)$. However, a stronger fact is true: for any $u\in N_K(x)$, the point $y = x-u \in \partial K^c$ and $-u\in N_{K^c}(y)$. Since this fact will be useful for us in what follows, we prove it here as well.

\begin{lem}\label{lem:normals map boundary to boundary}
Let $K\in \S_n$, $x\in \partial K$, and $u\in N_K(x)$. Then for $y = x-u$ we have $K\subseteq B(x-u, 1)$, i.e. $y \in \partial K^c$. Moreover, $-u \in N_{K^c}(y)$. 
\end{lem}

\begin{proof}
    
Since  $K\in \S_n$ is strictly convex, its support function $h_K$ is necessarily $C^1$, and by \cite[Corollary 1.7.3]{schneider2013convex}, $\nabla h_K(u) = x$ 
and $h_K(x)=\iprod{x}{u}$. 
For all $v\in \RR^n$ we have by Proposition \ref{prop:KminusK}
$h_K(v) + h_{K^c}(-v) = |v|$,  and differentiating we get 
\begin{equation}\label{eq:nablah_K}
\nabla h_K(v) - \nabla h_{K^c}(-v) = v/|v|.\end{equation}
Plugging in $v = u$ we get that 
$y = x-u = \nabla h_{K^c} (-u)$. In particular, $y\in \partial K^c$ is the unique point in $K^c$ for which $h_{K^c}(-u)=\iprod{y}{-u}$ and $-u \in N_{K^c} (y)$ as claimed. 
\end{proof}

 \begin{proof}[Another proof for Lemma \ref{lem:normals map boundary to boundary}]
 	Using Proposition \ref{prop:KminusK}, we know that 
 	\[ h_{K^c}(-u) = 1 - h_K(u) = 1 - \iprod{x}{u} = \iprod{x-u}{-u}.\] 
 	On the other hand, the point $y = x-u$ is at distance $1$ to $x$, and all other points on the hyperplane $\iprod{\cdot}{-u} = h_{K^c}(-u)$ 
 are at distance more than $1$ from $x$. Since $K^c\subseteq B(x,1)$, and we know there is some point in $K^c$ on this hyperplane, we conclude $y\in K^c$, as claimed. This is equivalent to $K\subseteq B(y,1)$. In particular we also get that $-u\in N_{K^c}(y)$ since $h_{K^c}(-u) = \iprod{y}{-u}$. 	
 \end{proof}

\begin{rem}
From the above lemma we see that the boundary of $K^c$ is precisely  the image of the set valued map taking a point  $ x\in \partial K$ to $ x - N_K(x)$. This mapping satisfies $\|x - Tx\|_2 = 1$ (in the set-valued sense, namely $\|x - y\|_2 = 1$ for all $y\in Tx$).
\end{rem}

One can give yet another description of the mapping $K\mapsto K^c$, captured in the following lemma.

\begin{lem}\label{lem:C(K)}
Let $K\in \S_n$.
Then  
	\begin{eqnarray}\label{eq-CKK}
	  K^c = \left\{
	x\in\R^n :
	\outrad\left((K-x)\cup (x-K) \le 1
	\right)
	\right\}.
	\end{eqnarray}
\end{lem}

 In fact, considering this new description of $K^c$, one is motivated to study a similar definition when one of the copies of $K$ is replaced by a different body $T$. While originally not even clear if such an adjustment would produce a body in $\S_n$, it turns out that this gives yet another description of the Minkowski average of the duals, and so we will prove the following, which implies Lemma \ref{lem:C(K)}.

\begin{lem}\label{lem:CKT)}
	Let $K,T\in \S_n$. Then  
	\begin{eqnarray}\label{eq-CKT}
 \frac{K^c + T^c}{2}	=  \left\{
	x\in\R^n :
	\outrad\left((K-x)\cup (x-T)) \le 1
	\right)
	\right\} . 
	\end{eqnarray}
\end{lem}

\begin{proof}
Denote the right hand side by $C(K,T)$. Let $x\in\R^n$. Then $x\in C(K,T)$ if and only if there exists $z\in\R^n$ such that $z\in(K-x)^c=K^c-x$ and also $z\in(x-T)^c=x-T^c$. Therefore $x\in C(K,T)$ if and only if $K^c-x$ intersects $x-T^c$, or equivalently $2x \in K^c + T^c$. This proves \eqref{eq-CKT} and in particular \eqref{eq-CKK}.
\end{proof}

\begin{rem}
We shall see that this simple representations can be applied to study some non-trivial intersections of $1$-lenses  in Section \ref{sec:application-carambula}, Proposition \ref{prop-1-MONTH-Car-INTERSECT}.
\end{rem}

We conclude this section with another representation of the class $\S_n$, which is sometimes used as the definition of the class, and which can serve to further convince the reader that this is a central class worthy of deep study. The bodies in $\S_n$ are convex bodies which are sufficiently curved in every direction. 

\begin{thm}\label{thm:char-curv}
The class 	$\S_n$ consists of all convex bodies in $\R^n$ for which all sectional curvatures are in $[1,\infty]$. 
\end{thm}

This theorem is classical, it follows from Blaschke's Rolling Theorem and its various generalizations, see \cite{Weil1982,Firey1979} and the discussion in \cite{drach2023reverse}.  

\section{Continuity, Isometry and Uniqueness}
 
In this section we discuss some basic useful facts regarding the $c$-duality.
First, we show that under mild assumptions, the mapping $K\mapsto K^c$ is, up to obvious adjustments, the only order reversing isomorphism on $\S_n$. This theorem  is in the spirit of \cite{boroczky2008duality,artstein2010chain,alesker2009fourier,alesker2010product,MilmanSchneider,florentin2016mixed,artstein2009duality,artstein2011hidden}, where many classical operations in convex geometry and beyond are shown to be ``god given'' in the sense that, having fixed a class of objects and very few properties, they are uniquely defined. 

Secondly, we study the continuity properties of the mapping $K\mapsto K^c$. On $\S_n$ the $c$-duality mapping is an isometry, so it is obviously continuous, but since it is defined on the larger class of all subsets of $\RR^n$, we can ask what continuity properties it satisfies on this larger domain. The fact that up to rigid motions there are only two isometries on $\S_n$, the $c$-duality and the identity, is shown in \cite{FUTUREWORK}. 

In the third part of the section we show that sets in $\S_n$ can be approximated by relatively simple sets, akin to the well known properties of convex bodies which can be well approximated by polytopes. We leave questions of rates of approximation of a ball-body by simple sets to future work.

\subsection{Characterization of the $c$-duality  on $\S^n$}

In many cases in convex geometry, only a single order reversing isomorphism exists, up to trivial linear adjustments, for example standard duality  on convex bodies \cite{boroczky2008duality, slomka2011} (and some of its sub classes \cite{florentin2011,segal2013}), or the Legendre transform in $\cvx(\R^n)$ \cite{artstein2009duality}.   This same phenomenon exists for the class $\S_n$,  namely  $c$-duality is (essentially) the only order reversing involution on $\S_n$. 

Since $K\mapsto K^c$ is an order reversing bijection of $\S_n$, for every order reversing isomorphims $T:\S_n\to\S_n$, the composition of $T$ with the $c$-duality is an order preserving isomorphism, so that it suffices to characterize order preserving isomorphisms on the class $\S_n$, which is the objective of the following theorem.

\begin{thm}\label{uniq:bijection-thm}
Let $F:\S_n\to\S_n$ be an order preserving bijection. Then $F$ is induced by a rigid motion $f:\R^n\to\R^n$, that is, there exist $x_0 \in \RR^n$ and $U\in O(n)$ such that for every $K\in\S_n$ 
\[
F(K)=\left\{f(x): x\in K\right\} 
\]
where $f(x) = x_0 + Ux$. Conversely, every rigid motion$f:\R^n\to \R^n$ induces a bijection on $\S_n$.
\end{thm}

For the proof we will use two simple lemmas.

\begin{lem}\label{uniq:signletons-to-singletons}
Let $F:\S_n\to\S_n$ be an order preserving bijection. Then $F(\emptyset)=\emptyset$, $F(\R^n)=\R^n$, and there exists a bijection $f:\RR^n \to \RR^n$ such that for any $x\in \RR^n$ we have $F(\{x\})=\{f(x)\}$.
\end{lem}
\begin{proof}
For every $K\in\S_n$ one has $\emptyset\subseteq K \subseteq \R^n$, i.e. $\emptyset$ and $\R^n$ are the (unique) minimal and maximal elements of the partially ordered set $\S_n$. Thus we necessarily have $F(\emptyset)=\emptyset$ and $F(\R^n)=\R^n$. Next, note that   singletons (elements of the form $\{x\}$) are the only elements of $\S_n$ which are greater than only one element - the empty set $\emptyset$. Since this property is preserved by $F$, singletons must be mapped to singletons, as required. Denoting by $f:\RR^n\to \RR^n$ the point map for which $F(\{x\}) =\{ f(x)\}$, clearly it must be a bijection on $\RR^n$ since $F^{-1}$ is an order isomorphism as well. 
\end{proof}

\begin{lem}\label{uniq:point-map-is-orthogonal}
Let $F:\S_n\to\S_n$ be an order preserving bijection, and let $f:\RR^n \to \RR^n$ be the bijection for which $F\left(\{x\}\right) = \left\{f(x)\right\}$. Then $f$ is a rigid motion (that is, $f(x)  = f(0)+ Ux$ for some $U\in O(n)$). 
\end{lem}
\begin{proof}
For simplicity of the following argument, let $\S_n^*=\S_n\setminus\{\R^n\}$.
First we show that $\|x-y\|_2=2$ if and only if $\|f(x)-f(y)\|_2=2$. Indeed, let $x,y\in\R^n$. If $\|x-y\|_2<2$ there are infinitely many elements in $\S_n^*$ which include both $\{x\}$ and $\{y\}$. If $\|x-y\|_2>2$ there are no elements in $\S_n^*$ which include both $\{x\}$ and $\{y\}$. If $\|x-y\|_2=2$ there is exactly one element in $\S_n^*$ which includes both $\{x\}$ and $\{y\}$, namely $B\left(\frac{x+y}2,1\right)$. The property of a pair of sets $A,B\in \S_n^*$, of having a unique element in $\S_n^*$ which includes both, is preserved by $F$. Thus for every $x,y\in\R^n$, we have
\[\|x-y\|_2=2 \,\,\Longleftrightarrow\,\, \|f(x)-f(y)\|_2=2.\]
By a theorem of Beckman and Quarles \cite{BeckmanQuarles1953}  (see also \cite{Benz1987}), this implies that $f$ is an affine orthogonal map, as required.
\end{proof}

\begin{proof}[Proof of Theorem \ref{uniq:bijection-thm}]
By  Lemma \ref{uniq:point-map-is-orthogonal}, given $F$ we find its associated   rigid motion $f:\R^n\to \R^n$, such that $F(\{x\}) = \{ f(x)\}$ for all $x\in \RR^n$.  Given $K\in\S_n$, denote  $\tilde{K}=\left\{\{x\}: x\in K\right\}$ the set of singletons which are included in $K$. Since $F$ is an order preserving bijection mapping singletons to singletons, each element in  $\tilde{K}$ is mapped by $F$ to a singleton which is included in $F(K)$, namely   $\{f(x):x\in K\} \subseteq F(K)$. However, as the same reasoning can be applied to $F^{-1}$ and $f^{-1}$, we see that 
\[
F(K)=\left\{f(x): x\in K\right\},
\]
as required. 

The fact that every rigid motion $f$ induces an order isomorphism is trivial, since $K\in\S_n$ implies $f(K)\in \S_n$,  the map $K\mapsto F(K)= \{ f(x): x\in K\}$ is order preserving, and $f^{-1}$ is a rigid motion as well, so that $F$ is an order preserving isomorphism.
\end{proof}

\subsection{Continuity properties of the $c$-duality}

It will be convenient in this section to denote $B(0,1) = B$. 
We shall use the Hausdorff distance between convex bodies, defined for two compact convex sets  $K,L$ in  $\RR^n$ by
\[ d_H(K,L)  = \inf \{ \lambda\ge 0: K \subseteq L+ \lambda B \, {\rm ~and~}
\, L\subseteq K + \lambda B 
\}.\]
Equivalently, we embed the class of convex bodies into $C(S^{n-1})$ using the support map, $K \mapsto h_K$, and pull back the uniform distance, namely $d_H(K, T) = \sup_{u\in S^{n-1}} |h_K(u) - h_T(u)|:=\|h_K-h_T\|_{\infty}$, see \cite{schneider2013convex} for details.  
We mention that if the reader feels uneasy using the set $K + \lambda B$ which might not be in $\S_n$, he or she can instead write the above inclusions as 
\[ \frac{1}{1+\lambda}K \subseteq \frac{1}{1+\lambda}L + \frac{\lambda}{1+\lambda} B\,\,\, {\rm and}
\,\,\, \frac{1}{1+\lambda}L\subseteq \frac{1}{1+\lambda}K + \frac{\lambda}{1+\lambda} B, 
\]
where now if both bodies $K,L\in \S_n$ then so do the sets for which inclusion is considered. 

It turns out that $c$-duality is an isometry on the class $\S_n$.   In particular, it is continuous and $1$-Lipschitz. 

\begin{prop}
	On the class $\S_n\setminus \{\RR^n, \emptyset\}$, the mapping $K\mapsto K^c$   is an isometry with respect to the Hausdorff distance. 
\end{prop}

\begin{proof}
	Let $K,L\in \S_n\setminus \{\RR^n, \emptyset\}$. Using Proposition \ref{prop:KminusK} we see that, denoting $h_K^{-}(u) = h_K(-u)$
	\[ d_H(K,L) = \|h_K - h_L\|_\infty =  \|(1-h_K^-) - (1-h_L^-)\|_\infty  
	= \|h_{K^c} - h_{L^c}\|_\infty = d_H(K^c, L^c). \]
This completes the proof. 	
\end{proof}

The fact that $K\mapsto K^c$ is an isometry is quite exciting, especially in view of the non-existence of isometries which are not rigid motion induced in the class of all convex bodies. Indeed, it was shown by Schneider \cite{Schneider1975} that on the class of all convex bodies in $\RR^n$,  an isometry with respect to the Hausdorff metric which is surjective must be induced by an isometry of $\RR^n$, in the sense that there is a rigid motion $g:\R^n \to \RR^n$ so that $F(K)=gK $ for all $K$. Without assuming surjectivity, Gruber and Lettl \cite{GruberLettl1980}  have shown that $F(K)=gK+L$ for some rigid motion $g$ and convex $L$.
However, we see that when we reduce to the class $\S_n$, there appears a new isometry. It corresponds, of course, to the possibility of ``subtraction'' since the class now is that of summands. 

\begin{rem}
   As mentioned above,  on the class $\S_n$ there are still not ``too many'' isometries with respect to the Hausdorff distance, and we show in \cite{FUTUREWORK, FUTUREWORK2} that  the only isometries are, up to affine rigid motions, the identity and the $c$-duality. 
\end{rem}

When considering  continuity properties of the mapping $A\mapsto A^c$ on the whole space of subsets of $\RR^n$, some caution is needed. First, sets of out-radius strictly greater than $1$ are mapped to the empty set,  which is of infinite distance to any other set, and thus the continuity is only possible on the class of subsets of $\RR^n$ with out-radius at most $1$ (and discarding the empty-set as well). To show continuity, we need a few preparations.

\begin{definition}
	For a set $A\subset \RR^n$ define the function $R_A:\RR^n \to \RR^+$ by 
	\[ R_A(x) = \inf\{ R>0: A\subseteq B(x,R)\}, \]
    in particular $A^c = \{ x: R_A(x) \le 1\}$. 
\end{definition}

Note that   the function $R_A(y)$, considered for a fixed $y$ as a function of the set $A$, satisfies monotonicity with respect to $A$ of course, but also the following property
\begin{eqnarray*} R_{A+ \eps B}(y) &=& \inf\{ R>0: A+ \eps B\subseteq B(x,R)\} =  \inf\{ R>0: A \subseteq B(x,R-\eps )\}  \\
	& =&
	\inf\{ R+\eps >0: A\subseteq B(x,R)\} = R_A(y) + \eps.\end{eqnarray*}

\begin{lem}
The function $R_A:\RR^n\to \RR^+$ is $1$-Lipschitz and convex. It attains a unique minimum which equals to $\outrad(A)$. 
\end{lem}

\begin{proof}
The fact that the minimum (which by definition equals the out-radius) is attained at a unique point is a classical fact from convex geometry, following from the fact that the intersection of two balls of the same radius $R$ and a different center, has outer-radius strictly smaller than $R$.
For convexity we note that if $A\subseteq B(x,r)$ and $A\subseteq B(y,s)$ then $A\subseteq B((1-\lambda)x + \lambda y, (1-\lambda)r + \lambda s)$. For Lipschitz we note that if $R_A(x) = r$ then $A\subseteq B(x,r)\subset B(y, r+|x-y|)$ and so $R_A(y)\le r+|x-y|$ and vice versa.  
\end{proof}

The next lemma captures the following fact: When a ball of radius $1-\delta$ is intersected with a ball of radius $1$, the intersection might still have out-radius $1-\delta$, if the $1$-ball included two antipodal points on the ball of radius $\delta$. However, if the centers of the two intersected balls are far enough in terms of $\delta$, this cannot happen, and the intersection will have out-radius strictly smaller than $1-\delta$. (The proof of Lipschitz continuity above, handles in fact the easy case $\delta = 0$, in which the distance between $x$ and $y$ merely has to be positive.)

\begin{lem}\label{lem:above}
Let $\delta\in(0,1)$ and $x,y\in \RR^n$. If $\|x-y\|_2>\eta(\delta) = \sqrt{2\delta-\delta^2}$ then 
	\[ \outrad(B(x,1) \cap B(y, 1-\delta)) < 1-\delta.\]
Moreover, there exists some $p\in(x,y)$ with $B(x,1) \cap B(y, 1-\delta) \subseteq B(p,1-\delta)$.
\end{lem}

\begin{proof}
Denote $d=\|x-y\|_2$. If $d> 2-\delta$ then $B(x,1) \cap B(y, 1-\delta)=\emptyset$, so there is nothing to prove.
The case $d=\eta(\delta)=\sqrt{1- (1-\delta)^2}$ corresponds to the boundaries of the balls intersecting on a big circle of $S(y, 1-\delta)$, i.e. $S(x,1) \cap S(y, 1-\delta) = \left(y+(x-y)^\perp\right)\cap S(y, 1-\delta)$. If $d\in (\eta(\delta),2-\delta]$, it can be checked by simple Euclidean geometry that the intersection $S(x,1) \cap S(y, 1-\delta)$ is a sphere centered at 
\[ p = \left(\frac12 - \frac{\eta(\delta)^2}{2d^2}\right) x + \left(\frac12 + \frac{\eta(\delta)^2}{2d^2}\right)y,
\]
of radius 
\[ r (d,\delta)  = \sqrt{1 - \left(\frac{\eta(\delta)^2 + d^2}{2d}\right)^2} < \sqrt{1 - \eta(\delta)^2} = 1-\delta.
\]
The intersection $B(x,1) \cap B(y, 1-\delta)$ consists of two spherical caps of radii $1-\delta$ and $1$, meeting in an $(n-1)$-dimensional ball of radius $r(d,\delta)$, centered at $p$. Since $r(d,\delta)<1-\delta<1$, the ball of radius $r(d, \delta)$ centered at $p$ contains both spherical caps, thus $\outrad(B(x,1) \cap B(y, 1-\delta)) = r(d,\delta) < 1-\delta$.
\end{proof}

\begin{prop}\label{prop:etasheldetla}
Let 
$\delta\in (0,1)$ and 
    $A\subset \RR^n$ with $\outrad (A+ \delta B)\le 1$. Then 	
	\[ (A +\delta  B)^c\subseteq A^c \subseteq (A +\delta  B)^c + \eta (\delta ) B. \]
	where $\eta(\delta ) = \sqrt{2\delta -\delta ^2}$ as in Lemma \ref{lem:above}
\end{prop}

\begin{proof}
We clearly have $(A +\delta  B)^c \subseteq A^c$, since $A\subseteq A+\delta B$. For the second inclusion, since $A^c$ is convex (as the intersection of balls), it suffices to show that its boundary $\partial A^c$ is contained in $(A +\delta  B)^c + \eta (\delta ) B$. Let $x\in \partial A^c$. This means $R_A(x) = 1$. The (convex) level set $K = \{ y: R_A(y)\le 1-\delta \} =  \{ y: R_A(y)+\delta  \le 1\} = (A+\delta  B)^c$ is not empty, by the assumption $\outrad (A+ \delta B)\le 1$. Let $y\in \partial K$ be the closest point to $x$ in $K$. Then $R_A(y) = 1-\delta$ and $x-y$ is in an outer normal direction to $\partial K$ at $y$, so that $[x,y)\cap K = \emptyset$.

We claim that $\|x - y\|_2 \le \eta(\delta)$. Indeed, if we suppose towards a contradiction that $\|x - y\|_2 > \eta(\delta)$, then by Lemma \ref{lem:above} there exists some $p\in(x,y)$ such that $A\subset B(x,1) \cap B(y, 1-\delta) \subset B(p,1-\delta)$, thus $R_A(p)< 1-\delta$, i.e. $p\in K$. This is  a contradiction, since the entire interval $[x,y)$ lies outside of $K$.
Summing up, for every $x\in \partial A^c$ we found some $y\in K=(A +\delta  B)^c$ such that $\|x - y\|_2 \le \eta(\delta)$, i.e. $\partial A^c\subseteq (A +\delta  B)^c + \eta(\delta ) B$, as required.  	
\end{proof}

\begin{cor}\label{cor:cdual-is-cont}
Let $n\in \N$, and consider the class of non-empty subsets of $\RR^n$ with out-radius at most $1$. On this class, the mapping $K\mapsto K^c$ is continuous in the Hausdorff metric.  	
\end{cor}

\begin{proof}
	Fix $K\subset \RR^n$, $K\neq \emptyset$, $\outrad(K)\le 1$. We first address the case that $\outrad(K) = 1$. Then for the (unique) $x\in K$  with $R_K(x)=1$ we have and $K^c = \{ x\}$. Denoting 
    \[ r_T(x) = \sup \{ r: B(x,r)\subset T\}\]
we see that $R_T(x) = R$ implies $T\subseteq B(x,R)$ implies $T^c \supseteq B(x,1-R)$ which implies $r_{T^c}(x)\ge 1-R$, namely $B(x,1-R)\subset T^c$, which in turn implies $T\subseteq T^{cc}\subseteq B(x,R)$.     
    
     If $d_H(L, K) \le  \eps$ for some non-empty $L\subset \RR^n$ with $\outrad(L)\le 1$ then as $R_{L+\eps B}(x) = R_L (x)+\eps$
    and $K\subset L+\eps B$, we see $1=R_K(x)\le R_L(x)+\eps$ so that $R_L(x)\ge (1-\eps)$ which means $r_{L^c}(x)\le \eps$ so that $L^c\subseteq B(x,\eps) = K^c+\eps B$, and clearly in such a case $\{x\}\subset L^c+\eps B$  
     and we get the claim (in fact, with constant $1$) as needed. 
	
	The second case to consider is continuity at a set $K$ where $\outrad(K) <1$. 	Let $\eps>0$, and let 
    $\delta<\delta_0  = \min \left(\frac{1-\outrad(K)}{4}, \frac{\eps^2}{2} \right)$. 
    
    Consider some non-empty $L$ with $d_H(K,L)<\delta$. Then 
$\outrad(K), \outrad(L) \le 1-\delta$ and 
$L\subseteq K+ \delta B$ and $K \subseteq L+\delta B$. Therefore
 $L^c\supseteq (K+ \delta B)^c$ and $K^c\supseteq (L+\delta B)^c$ and so 	\[ L^c + \eta B \supseteq (K+ \delta B)^c+ \eta B, \,\,\,{\rm and}\,\,\, K^c+\eta B \supseteq (L+\delta B)^c+\eta B.\] 
	Picking $\eta = \eta(\delta)$ from   Proposition \ref{prop:etasheldetla} we get  
	\[ L^c + \eta(\delta) B \supseteq K^c, \,\,\,{\rm and}\,\,\, K^c+\eta(\delta) B \supseteq L^c,\] 
	namely $d_H(K^c, L^c) \le \eta (\delta)$. We check that 
    \[ \eta(\delta) <\eta (\eps^2/2) = \sqrt{\eps^2- \eps^4/4}<\eps,    
    \]
    and the proof is complete. 
\end{proof}

\begin{cor}\label{cor:c-hull-is-continuous}
	Let $n\in \N$, and consider the class of non-empty subsets of $\RR^n$ with out-radius at most $1$. On this class, the mapping $K\mapsto K^{cc} = \cconv(K)$ is continuous in the Hausdorff metric.      
\end{cor}

\begin{proof}
    The statement follows immediately from the fact that $K\mapsto K^c$ is continuous and maps non-empty sets with out-radius at most $1$ to sets in $\S_n$ which are non-empty and with out-radius at most $1$.
\end{proof}

\begin{rem}
We did not aim for best constants in the inequalities. It is interesting to check whether one can in fact get that $K\mapsto K^c$ is $1$-Lipschitz with respect to the Hausdorff distance, on the class of subsets of $\RR^n$ which are of out-radius at most $1$. 
\end{rem}

\subsection{Approximation}

We end this section with several  useful of theorems regarding  the denseness of some natural  subsets of $\S_n$.

\begin{prop}[Denseness of $c$-polytopes]
    Let $n\in \mathbb N$ and $K\in \S_n$. There exists  a sequence of finite sets $A_m\subset \RR^n$ such that  $d_H(\cconv(A_m), K)\to_{m\to \infty}  0$. 
\end{prop}

\begin{proof}
We use the fact that a convex body can be approximated from within by polytopes, $P_m \to K$, see \cite[Proposition A.3.5]{AGMBook}. If $P\subset K$ and $K\in \S_n$ then also $\cconv(P)\subset K$ and $d_H(\cconv(P),K)\le d_H(P,K)$
proving the proposition (where $A_m$ is the set of vertices of the polytope $P_m$).     
\end{proof}

Similarly by dualizing we get a corresponding fact for approximating a ball-bodiy from the outside by intersections of Euclidean unit balls. 

\begin{prop}[Denseness of $c$-polyhedrals]
    Let $n\in \mathbb N$ and $K\in \S_n$. There exists a sequence $(K_m)_{m\in \mathbb N}$ of finite intersections of $1$-balls such that   $d_H(K_m, K)\to_{m\to \infty} 0$. 
\end{prop}

\begin{proof}
We use the fact that any convex body can be approximated by polytopes in which it is included,  see \cite[Proposition A.3.5]{AGMBook}. 
Let $K\in \S_n$. If $K^c$ is a translate of the Euclidean unit ball, there is nothing to prove (a singleton is the intersection of two Euclidean unit balls). Otherwise, $\outrad(K^c)<1$. Take a sequence 
 $P_m \supseteq K^c$ and $d_H(P_m, K^c)\to 0$. By Corollary \ref{cor:cdual-is-cont} we have that $d_H(P_m^c, K)\to 0$ and $P_m^c$ is the intersection of a finite number of balls centered at the vertices of $P_m$, completing the proof. 
\end{proof}

\begin{rem}
It is not hard to check that $c$-polytopes cannot be self dual in dimension $n\ge 3$. Nevertheless, if one looks for a ``simple'' dense subset of self-dual (i.e., constant width) bodies
a natural set to consider is the subset of $\{( P + P^c) /2: P = A^c,\,A~ {\rm finite}\}$. This is easily shown to be a dense subset of self-dual (that is, constant width 1) bodies.  
\end{rem}

Finally, smooth bodies in $\S_n$ constitute a dense subset. To this end we employ standard approximation techniques which are explained\footnote{We would like to thank Daniel Hug for discussing approximations and for pointing us to the most relevant theorem in \cite{schneider2013convex}} clearly in \cite[Section 3.4]{schneider2013convex}. In this proof we denote the Euclidean unit ball by $B_2^n$. 

\begin{prop}[Denseness of smooth bodies]\label{thm:dense-are-the-smooth}
    Let $n\in \mathbb N$ and $K\in \S_n$. There exists a sequence $(K_m)_{m\in \mathbb N}$ with $K_m\in \S_n$ which are $C^\infty$ smooth convex bodies with $h_K\in C^\infty$, such that  $d_H(K_m,K)\to_{m\to \infty} 0$. 
\end{prop}

\begin{proof}
    Given $K\in \S_n$ we may assume without loss of generality that $K\subseteq B_2^n$. 
    We employ the approximation procedure described in \cite[Theorem 3.4.1]{schneider2013convex}, where  $\eps_m>0$ is some sequence with $\eps_m \to 0$. To this end we fix for every $m$ some $\varphi_m:[0,\infty)\to [0,\infty)$ which is $C^\infty$ smooth, has $\int \varphi_m(|z|)dz = 1$,  and is supported on $[\eps_m/2, \eps_m)$, and define the mapping
    \[ T_mf (x) = \int_{\RR^n} f(x + |x|z)\varphi_m(|z|)dz.\]
Theorem 3.4.1 in \cite{schneider2013convex} implies that $T_m(h_{K})$ is the 
    support function of a convex body, which we call $K_m'$, and moreover $h_{K_m'}$
     is $C^\infty$ on $\RR^n\setminus \{0\}$. 
     Moreover (upon identifying the map $T_m$ on support functions and on convex bodies), $d_H(K, T_mK) \le \eps_m$ for all $K\in \S_n$ since $\outrad(K)\le 1$ and using property (c) of \cite[Theorem 3.4.1]{schneider2013convex}. We see also that $T_m(K+L) = T_mK + T_mL$ by definition (this is (a) in \cite[Theorem 3.4.1]{schneider2013convex}), and  that $T_m(B_2^n)$ is a Euclidean ball since it is invariant under rigid motions by property (b) of the same theorem. Denoting $T_m(B_2^n) = \alpha_m B_2^n$, we see that $\alpha_m \in [1-\eps_m, 1+\eps_m]$ since $d_H(\alpha_mB_2^n , B_2^n)\le \eps_m$.   
In particular, if $K+L= B_2^n$ we have $K_m'+T_m(L)= \alpha_m B_2^n$ and thus $\frac{1}{\alpha_m} K_m' + \frac{1}{\alpha_m}T_m(L) = B_2^n$. We let $K_m'' = \frac{1}{\alpha_m} K_m'$, so that 
$d_H(K_m'', K_m') = |1-\frac{1}{\alpha_m}|\sup_{u\in S^{n-1}}|h_{K_m'}(u)|\le |1-\frac{1}{\alpha_m}|\le 2\eps_m$ if we assume $\eps_m<1/2$, which we may. 
We see that $K_m''$ is a summand of $B_2^n$, and its support function is $C^\infty$, since this was the case for $K_m'$. 
The last step in our construction is to ensure that the body has no singular points. This would already imply that the body is $C^\infty$; For the discussion connecting the smoothness of the support function with the smoothness of the body, see \cite[Section 2.5]{schneider2013convex}  where the $C^2_+$ case is considered, but the proof works for any degree of smoothness. See also the discussion after Theorem 3.4.1 in the same book. To this end
we let 
\[
K_m  = (1-\eps_m)K_m'' + \eps_m B_2^n, 
\]
so that 
\[ d_H(K_m'',K_m)= \sup_{u\in S^{n-1}} |h_{K_m}(u) - h_{K_m''}(u)|
= \eps_m d_H(K_m'', B_2^n) \le \eps_m.
\] 
Clearly $K_m$ is a summand of $B_2^n$,  it is $C^\infty_+$, and its support function is also $C^\infty$. We get that  
\[ 
d_H(K_m, K) \le d_H(K_m, K_m'') + d_H(K_m'', K_m')+d_H(K_m',  K)\le 4\eps_m  
\]
and the proof is complete. 
\end{proof}

\section{Iso-parametric inequalities}\label{sec:isop-ineq}

Within a fixed class of bodies, it is of geometric interest to understand the extremal behavior of certain size or shape parameters with respect to others. A classical example is the isoperimetric inequality, stating that fixing volume, surface area is minimized (among {\em all} sets for which it can be reasonably defined) for balls. A reverse  isoperimetric inequality (maximizing surface area for fixed volume) does not hold without additional assumptions since one may construct bodies, even convex ones, with arbitrarily large surface area and fixed volume, for example by taking a very thin sheet. To solve this problem, it is customary to introduce a ``position'', in which case a celebrated theorem by Ball \cite{ball1991reverseisop} gives the extremizers. However, if one considers a smaller class, for instance $\S_n$, it is already reasonable to investigate sets of maximal surface area for a fixed volume without any position assumption. 


Within the class $\S_n$, an isoperimetric-type conjecture was suggested by Borisenko. It appeared first in the Ph.D.~dissertation of Drach \cite{Drach2016}, and is first formally stated in English in \cite[Section 4.2]{ChernovDrachTatarko2019}. 
\begin{conj}\label{conj:reverse-isop-twod}
Let $n\in {\mathbb N}$ and 	$ V\in (0, \kappa_n)$. Of all sets $K\in \S_n$ with fixed volume $\vol(K) = V$, the ones maximizing surface area  
are precisely lenses of volume $V$. 	
\end{conj}

Conjecture \ref{conj:reverse-isop-twod}  was proved by Borisenko and Drach \cite{BorisenkoDrach2014} in dimension $n=2$,  and recently by Drach and Tatarko \cite{drach2023reverse} in dimension $n=3$. The case $n\ge 4$ is currently open.

One may of course compare various other parameters for bodies in $\S_n$, such as $V(K[n-j], B[j])$, $V(K[n-j], -K[j])$ and similar parameters involving the $c$-dual of $K$. Usually, when considering the comparison of two parameters, a one-sided inequality follows simply since $\S_n$ is a subset of the class of all convex bodies, where extremizers are known, while the other side in the class of all convex bodies requires a ``position'', and in $\S_n$ can be considered directly since the bodies cannot be too degenerate. 

In this section we discuss relatively simple parameters, which already have interesting properties in this class, which are, additionally to volume, the diameter of a set,  its out-radius ($\min_x R_K(x)$) and its in-radius ($\max_x r_K(x)$, in the notations of the proof of Corollary \ref{cor:cdual-is-cont}).

\subsection{In-radius, Out-radius and Diameter}
Recall that for $K$ a convex body,  $\inrad(K) = \max\{ r: \exists x, B(x,r) \subseteq K\}$ is its in-radius, 
$\outrad(K) = \min\{ R: \exists x, K \subseteq B(x, R) \}$ is its out-radius, and $\diam(K) = \max \{ \|x-y\|_2:  x,y \in K\}$ is its diameter. 
We start with a simple fact, which is that for $K\in \S_n$ both the out-radius and the in-radius have a unique point in which they are attained. This is very much not the case for in-radius in the bigger class of all convex bodies. Moreover, the points at which they are attained are connected by duality, as are their values.

\begin{lem}\label{lem:inplusout} 
	For any $K \in \S_n$ we have 
	\[ \outrad(K) + {\rm Inrad}(K^c) = 1.\]
	Moreover, there is a unique point $x$ for which $K \subseteq B(x, \outrad(K))$, which is also the unique point for which $B(x, {\rm Inrad}(K^c))\subseteq K^c$. In other words, the 
	 smallest ball containing $K$ and the largest ball contained in $K^c$ are unique, concentric, and   $c$-dual to one another.
\end{lem}
\begin{proof}
Clearly for any convex body $K$ the out-radius is attained at a unique point. Indeed, if $K\subseteq B(x, R)\cap B(y, R)$ with $x\neq y$, then $K\subseteq B(\frac{x+y}{2}, \sqrt{R^2 - \frac{|x-y|^2}{4}})$ so that the out-radius of a convex body is attained at a unique point. If $K\in \S_n$, then also the in-radius is attained at a unique point. This can be shown directly since the $c$-hull of two $r$-balls contains a ball with larger radius. However, it also follows from the following argument.

	For any $x\in\R^n$ and $R\in[0,1]$ we have
	that $K \subseteq B(x,R)$ if and only if $B(x,1-R) \subseteq K^c$, 
    from which follows that $r_0$ is minimal for $K$ if and only if $(1 - r_0)$ is maximal for $K^c$, in which case they are attained for balls centered at the same $x_0$. By uniqueness of the ball attaining the out-radius we get that for bodies in $\S_n$, the maximal inscribed ball is also unique, and moreover  $ \outrad(K) + {\rm Inrad}(K^c) = 1$, as claimed. 
\end{proof}

It is easy to check that the $c$-hull cannot increase the out-radius of a set, since the out-ball containing the set  is a body in $\S_n$  and will thus also include its $c$-hull. Therefore
	\begin{equation}\label{eq:out-and-outc}
	\outrad(\cconv(A)) = \outrad(A).
	\end{equation}
Nevertheless, when discussing diameter this is no longer true, and in contrast with the classical convex hull operation, one can find examples for which 		\[ 	{\rm diam}(\cconv(A)) >	{\rm diam}(A).  \]

\begin{exm}
Let $\varepsilon\in (0,\pi/3)$, $L\in (1,2\cos(\varepsilon)]$, and consider a thin isosceles triangle $T = {\rm conv}(x,y,z)\subset \RR^2$ with  $\|x-z\|_2 = \|y-z\|_2 = L, \|x-y\|_2 = 2\sin(2\varepsilon)$. Then $\outrad (T^{cc}) = \outrad (T) < 1$, however it is easy to see that $2\ge \diam(T^{cc}) = \sqrt{L^2-\sin^2(2\varepsilon)} + 2\sin^2(\varepsilon) > L = \diam(T)$.
In fact, the worst ratio $\diam(K^{cc})/\diam(K)$ is $\sqrt{2n/(n+1)}$, as we show  in Theorem \ref{thm:diam-diamhull}.
\end{exm}

The choice of $L>1$ in the previous example is not incidental, and we next show that when a set is of diameter at most $1$, the operation of $c$-hull does not change its  diameter. 
	We do this in two steps. 	
	First, we show that if a set has diameter less than $1$, the $c$-hull operation does not change this fact. 
	\begin{lem}\label{lem:diam-of-cconv-stay-below-1}
    Let $n\in {\mathbb N}$ and let $K\subset \RR^n$ satisfy $\diam(K) \le 1$. Then $\diam(K^{cc}) \le 1$. 
		\end{lem}
	\begin{proof}
	Indeed, $\diam (K) \le 1$ is equivalent, by definition, 
	to the condition $K\subseteq K^c$. This in turn implies, as $c$-duality reverses order, that $K^{cc}\subseteq K^c = K^{ccc}$ which means that the diameter of $K^{cc}$ is at most $1$. 
\end{proof}
For the second step  
 we need the following lemma  
	\begin{lem}\label{lem:small-dilations-have-less-1-conv}
		Let $n\in {\mathbb N}$, let   $K\subset \RR^n$ and let $t\in (0,1)$. Then 
		\[
		(t K)^{cc} \subseteq t K^{cc}.
		\]
	\end{lem}
	
\begin{proof}
	Let $x\in (tK)^{cc}$ , we should show $x/t \in K^{cc}$ namely $B(x/t, 1) \supseteq K^c$ 
	when we know $B(x, 1) \supseteq (tK)^c$. 
	We need to show that if $K\subseteq B(y,1)$ then $\|y-x/t\|_2\le 1$, under the assumption $B(x, 1) \supseteq (tK)^c$. Assume  	
	$K\subseteq B(y,1)$.   Then   $tK \subseteq B(ty, t) \subseteq B(z, 1)$ for every $z\in B(ty, 1-t)$, and by the assumption this implies that $\|z-x\|_2\le 1$ for every $z\in B(ty, 1-t)$. Clearly $B(ty,1-t)\subseteq B(x,1)$  implies $\|ty-x\|_2\le t$ meaning $\|y-x/t\|_2\le 1$ as needed. 
\end{proof}

 With this in hand, we  
 prove that the  $c$-hull operation does not increase the diameter of a convex body, if it is smaller than $1$.
	\begin{prop}
		Let $K\subset \RR^n$ with $\diam(K)\le 1$. Then 
        \[ \diam(K^{cc}) = \diam(K).
		\]
	\end{prop}
	\begin{proof}
		Let $d:=\diam(K)$ and $t = d^{-1}$.  Since $\diam(tK) = 1$, we have by Lemma \ref{lem:diam-of-cconv-stay-below-1} that 
		$\diam((tK)^{cc})= 1$, or, equivalently,
	$\diam(d(tK)^{cc})= d$. 
		By Lemma \ref{lem:small-dilations-have-less-1-conv}, we have $K^{cc}\subseteq d(tK)^{cc}$, i.e. $\diam(K^{cc})\le d$, which completes the proof.
	\end{proof}

	To demonstrate the ``worst'' behavior of diameter with respect to $c$-hull, consider the regular simplex $\Delta_n(d)$ of edge length $d = \sqrt{\frac{2(n+1)}{n}}>1$. It  has diameter $d$, but has out-radius $1$ so that there is a unique ball including it, namely $(\Delta_n(d))^{cc} = B(0,1)$. In other words,  $c$-convexifying increases the diameter by a factor close to $\sqrt{2}$. In fact, this  example is sharp, and we have 
	
	\begin{thm}\label{thm:diam-diamhull}
		Let $n\in {\mathbb N}$ and let $K\subset \RR^n$ with $\outrad(K) \le 1$.  Then    $\diam (K^{cc}) \le 
		\sqrt{\frac{2n}{n+1}}
		\diam(K)$. 	
	\end{thm}

	\begin{proof}
		We use Jung's theorem \cite{Jung1901} which states that 
		\begin{equation}\label{eq:Jung}\outrad(K) \le \sqrt{\frac{n}{2(n+1)}}\diam(K).\end{equation}  
	Combining \eqref{eq:Jung} with \eqref{eq:out-and-outc} we see that \[ \diam(K^{cc}) \le 2\outrad(K^{cc}) = 2\outrad(K) \le \sqrt{\frac{2n}{n+1}}\diam(K).\]
 The simplex with edge-length $d = \sqrt{\frac{2(n+1)}{n}}$ attains an equality. 
	\end{proof}

Using these simple observations, we can prove a Santal\'{o}-type inequality for the diameter as follows. 

\begin{thm}
		Let $n\in {\mathbb N}$. For any $K\in \S_n$ we have 
	\[2\le  \diam(K) + \diam(K^c) \le 2\sqrt{2}.\]
	Moreover, fixing $\diam(K) = d$, we have that $2-d\le \diam(K^c) \le  \sqrt{4-d^2}$. 
\end{thm}

\begin{proof}
First, since $K - K^c = B(0,1)$ and diameter is sub-additive, we see that 
	\[ 2 = \diam(K - K^c) \le \diam(K) + \diam(K^c).\]
	For the inequality in the opposite direction, let
	 $\diam(K) = d$, and let $x,y\in K$ with $\|x-y\|_2 = d$. Then $\cconv (x,y)\subseteq K$ and so $K^c \subseteq 
	(\cconv (x,y))^{c} = \{x,y\}^c$ which means $\diam(K^c) \le \diam(\{x,y\}^c)$. So all that is left for proving the inequality is to 
    consult the appendix, specifically 
    Lemma \ref{lem:k-lens-properties}, for the appropriate values of the diameters of a pair of dual $1$-lens and $(n-1)$-lens, showing  
    $\diam((\{x,y\})^{c}) = 2\sqrt{1-(\|x-y\|_2/2)^2}$, implying that 
    that $\diam(K^c) \le  
    \sqrt{4-d^2}$. Maximizing over $d$ gives the value $2\sqrt{2}$ for $d = \sqrt{2}$. 
\end{proof}
 
\subsection{Contact points of a body and its in/out-ball}

We will make use, in the sequel, of the special structure of the set of contact points of a body in $\S_n$ and its in-ball.

\begin{lem}\label{lem:contact points give definig ball}
	Let $n\ge 2$  and $r\in(0,1)$. Let $B(0,r)\subseteq K\in \S_n$ and assume $\inrad(K) = r$. 
    Then  $rv\mapsto -(1-r)v$ is a one-to-one correspondence between the contact points 
 $B(0,r)\cap \partial K$ and the contact points $S(0,1-r)\cap K^c$.
\end{lem}

\begin{proof}
Clearly (see Lemma \ref{lem:inplusout})  $K^c\subseteq B(0,1-r)$ and $\outrad(K^c) = 1-r$. 
 For the correspondence, let $rv\in B(0,r)\cap \partial K$. The   outer normal to $K$ at $rv$ is $v$, as it is also a normal to $B(0,r)$ at $rv$. 
 By Lemma \ref{uniq:signletons-to-singletons} the point $rv-v = -(1-r)v\in \partial K^c$ and is clearly in $S(0, 1-r)$ meaning it is a contact point of $K^c$ and its out-ball. For the other direction, given $(1-r)u$ a contact point of $K^c$ and $B(0,1-r)$, we have that $u = n_{B(0,1-r)}((1-r)u) \subseteq N_{K^c}((1-r)u)$ and again by Lemma \ref{uniq:signletons-to-singletons} we see that $-ru\in \partial K\cap B(0,r)$ as needed. 
\end{proof}

Contact points of the in-ball of a convex body and the body itself must be relatively ``spread'', and similarly the contact points of a convex body and its out-ball. The facts mentioned in the following two lemmas are well known, and we sketch the proofs after the statements for the convenience of the reader.

\begin{lem}\label{lem:contactopints} 
	Let $n\ge2$, let $K\subset \RR^n$ be a convex body,  and let $r>0$.  Assume $B(0,r)\subseteq K$ and denote   $C = B(0,r) \cap \partial K$. The following are equivalent:\\
 (i) The in-radius of $K$ is $r$.\\
 (ii) The set $C$ intersects every closed hemisphere of $S(0,r)$, i.e.
	\begin{equation}\label{eq:contact-condition}
		\forall u\in S^{n-1} \, \exists x\in C\,\,{\rm s.t.}\, \iprod{x}{u} \ge 0.
	\end{equation}
 (iii) There exists $1\le k\le n$ and a subset $C'\subseteq C$ of $k+1$ points such that the positive span of $C'$ is a subspace of dimension $k$ (equivalently, $\conv(C')$ is a $k$-simplex with $0$ in its relative interior). 
\end{lem}

The out-radius is similarly characterized
\begin{lem}\label{lem:contactopintsout} 
	Let $n\ge2$, let $K\subset \RR^n$ be a convex body,  and let $R>0$. Assume $ K \subset B(0,R)$ and denote   $C = S(0,R) \cap K$. The following are equivalent:\\
 (i) The out-radius of $K$ is $R$.\\
 (ii) The set $C$ intersects every closed hemisphere of $S(0,R)$, i.e.
	\begin{equation}\label{eq:contact-condition2}
		\forall u\in S^{n-1}\, \exists x\in C \,\,{\rm s.t.}\, \iprod{x}{u} \ge 0.
	\end{equation}
 (iii) There exists $1\le k\le n$ and a subset $C'\subset C$ of $k+1$ points such that the positive span of $C'$ is a subspace of dimension $k$ (equivalently, $\conv(C')$ is a $k$-simplex with $0$ in its relative interior). 
\end{lem}

\begin{proof}[Sketch of proofs of Lemma \ref{lem:contactopints} and Lemma \ref{lem:contactopintsout}] We start with Lemma \ref{lem:contactopints}, and show that (i) is equivalent to (ii). 
	Assume that $r = \inrad(K)$. Suppose towards a contradiction that there exists $u\in S^{n-1}$ with $\iprod{x}{u} < 0$ for all $x\in C$. Let $\eps_0$ denote the distance between $\partial K$ and $B(0,r) \cap \left\{ x: \iprod{x}{u} \ge 0 \right\}$, which is positive since these two closed sets do not intersect, by our assumption. Therefore the ball $B(\frac{\eps_0}{2} u, r)$ is contained in $K$. This implies that the in-radius of $K$ is attained at two different points. Consider the midpoint $z$ of these two. By convexity, the ball of $B(z,r)$ is contained in $K$, and moreover, it can intersect the boundary of $K$ only where it intersects the convex hull of the two balls. However, since by assumption all the points on the lower dimensional sphere $\{x: \iprod{x}{u}  = 0\}$ do not belong to $\partial K$, these are in its interior. By convexity this implies all the points on the translated hemisphere $\{x: \iprod{x}{u}  = \iprod{z}{u}\}$ are in the interior of $K$ as well. Therefore all of $S(z,r)$ belongs to the interior of $K$, and by compactness one may find a larger ball with center $z$ contained in $K$, contradicting that $r = \inrad(K)$.

For the other direction, assume that $C$ satisfies \eqref{eq:contact-condition}. Had there been some ball in $K$ of radius $r'>r$ then in particular $K$ would include another ball of radius $r''>r$ with a different center, $u\neq 0$. The half-space $\{x:\iprod{x}{u}\ge 0\}$, on the one hand, contains some point of $C$, and on the other hand, its intersection with $S(0,r)$ is contained in the (usual) convex hull of the two balls $B(0,r), B(u, r'')$, which is a contradiction, as the convex hull of two balls with different radii centered at different points includes in its interior a closed half-ball of the smaller ball  $B(0, r)$.  

The equivalence of (i) and (ii) in Lemma \ref{lem:contactopintsout} is proven using a similar argument; If $K\subseteq B(0,R)$ and $u\in S^{n-1}$ satisfies $C\subseteq 
\{x:\iprod{x}{u}< 0\}$  then we can find a ball $B(z, R)$ containing $K$, for $z = -\eps u$ where $\eps$ is chosen using compactness, namely letting 
$d=  d(C, \{ x: \iprod{x}{u} = 0 \})$ we take 
$0< \eps <  d(K, \{ x: \iprod{x}{u} \ge  - d/2\}\cap S(0,R) )<d$. 
This contradicts uniqueness of the out-ball of convex body. 
For the other direction, if condition \eqref{eq:contact-condition2} 
is satisfied but $R> \outrad(K)$ then we can find a smaller ball $K \subset B(z, R')$, $R'<R$, meaning $K \subset B(z, R') \cap B(0, R)$. However this means all the contact points of $K$ and $B(0,R)$ belong to $B(z,R') \cap S(0,R)$ which is contained in an open hemisphere of $S(0,R)$ contradicting condition (ii).

To show the equivalence of (ii) and (iii) in both Lemmas, 
first note that (iii) immediately implies (ii). For the opposite direction, given $C$ which satisfies (ii),  
take a minimal subset $C'$ of $C$ which still satisfies \eqref{eq:contact-condition} (it exists by Zorn's lemma on closed subsets of $C$ satisfying \eqref{eq:contact-condition} with the order of inclusion). Consider its positive span 	
	\[
	E =
	\left\{
	\sum_{i=1}^{m} \lambda_i x_i \,\big|\,
	m\in\N,\, \lambda_i > 0,\, x_i\in C'
	\right\}. 
	\]
	The set $E$ is a cone in $\RR^n$. A cone in $\RR^n$ is always of the form $K\oplus F$ for some subspace $F$ and some proper cone $K$ (see \cite[Lemma 1.4.2]{schneider2013convex}). 
    First, we claim that $F$ cannot be trivial. Indeed, if $F = \{0\}$ this means that $K = E$ is a proper cone. A proper cone in $\RR^n$, intersection with $S^{n-1}$, is contained in an open half sphere (the cone is a convex subset of $\R^n$ with the origin a boundary point, take the normal cone at $0$, it must have non-empty interior otherwise the normal cone is contained in a proper subspace and its orthogonal complement will be part of the original cone). 
	We thus see that $F\neq \{0\}$. 
	Next we claim that the subset $C'\cap F$ must positively span $F$.  
	Indeed, $F$ is positively spanned by points in $C'$. Assume $f\in F$ satisfies that $f = \sum \lambda_i x_i$ where $x_i \in C'$. If $\lambda_i \neq 0$ for some $x_i \in C'\setminus F$ then by manipulating this expressions we can get an equation of the form 
	\[ f'  = \sum \lambda_j x_j\] where $f'\in F$ and $x_j \in C'\setminus F$, $\lambda_j >0$ for all $j$. The expression  $ \sum \lambda_j x_j$ belongs to $K$, and since the sum $K\oplus F$ is a direct sum, this is a contradiction. 
	
	Finally, since $C'\cap F$ positively span $F$, in particular they satisfy condition \eqref{eq:contact-condition} for $S(0,r)\cap F$ and this is inherited by $S(0,r)$. We conclude, by minimality, that $C' = C'\cap F$, and that $K = \{0\}$, so that $E$ is indeed a subspace. 
\end{proof}

\subsection{Fixing $r(K)$ and extremizing $R(K)$ and $\diam(K)$} 

The following theorem is a comparison between the inner and outer radius of a body in $\S_n$. It is worthwhile to consider the analogous question in the classical duality theory for convex bodies. Clearly, we may find convex bodies with inner radius $r$ and outer-radius any number $R\ge r$. Since, after an appropriate translation, the inner radius is the reciprocal of the outer-radius of the polar body $K^\circ= \{ y: \sup_{x\in K}\iprod{x}{y}\le 1\}$ and vice versa, this means that $I(K) := r(K)/R(K)=\inrad(K)\cdot \min_z \inrad ((K-z)^\circ)$ 

is bounded above by $1$, and can be arbitrarily close to $0$. 
In the setting of $\S_n$ and $c$-duality, the same trivial upper bound $r(K)\le R(K)$ of course holds. However here a  lower bound also exists,  and follows from  the characterization of maximal inscribed balls in terms of their contact points with the containing body.

\begin{thm}\label{thm:santalo-thereal}
Let $K\in \S_n$ with $\inrad(K)=r\in [0,1]$. Then
\begin{equation}\label{eq:Santalo-for-omi}  
		r \le \outrad(K) \le \sqrt{2r-r^2}.
\end{equation}
Moreover, if $B(x,r)\subseteq K$ then $K\subseteq B(x, \sqrt{2r-r^2})$. 
Equality on the left hand side is attained if and only if $K=rB_2^n$.
Equality on the right hand side is attained for many bodies, for example for any 
$K$ that lies between a $1$-lens of in-radius $r$, 
and an $(n-1)$-lens  of in-radius $r$.
\end{thm}

\begin{rems}
        Inequality \eqref{eq:Santalo-for-omi}  appeared in 
        \cite{BorisenkoDrach2013}, and more explicitly in \cite{Drach2015}, see also 
        \cite[Lemma 12]{Bezdek2021}. \\
    We also mention the parameter $\outrad(K)- \inrad(K)$, measuring a ``distance'' from being a ball, which in the class $\S_n$ is thus always smaller than $\sup_{r\in (0,1)}(\sqrt{2r-r^2}-r) =\sqrt{2}-1$, which is attained at $r=1-1/\sqrt{2}$.     
\end{rems}

\begin{proof}[Proof of Theorem \ref{thm:santalo-thereal}]
The left hand side inequality is trivial for any convex body. Fix some $r\in (0,1)$ and a body $K$ with $\inrad(K) = r$. Translate $K$ so that $B(0,r) \subseteq K$. A lens $L$ of in-radius $r$ is one example of such a body, and its outer-radius is $g(r) = \sqrt{2r-r^2}$.
Assume towards a contradiction that there exists $x\in K$ with $\|x\|_2>g(r)$. The half-space $\{y: \iprod{x}{y}\ge 0\}$ contains at least one contact point of $K$ and $B(0,r)$ by Lemma \ref{lem:contactopints}. However, this contact point corresponds by Lemma \ref{lem:contact points give definig ball} to a ball including $K$ which does not include $x$ (it does not include, in this half-space, any point of Euclidean norm more than $g(r)$), a contradiction. 
\end{proof}

\begin{cor}\label{cor:inrad-lower-bound-for-CW-bodies}
If $K\subset \RR^n$ is a body of constant width $1$, then its in-radius satisfies $\inrad(K)\ge 1 -\sqrt{\frac12}\approx 0.293$. 
\end{cor}
\begin{proof}
Indeed, if $K=K^c$ has $\inrad(K)=r\in(0,1)$ then $1-r=\outrad(K)$ and also, by Theorem \ref{thm:santalo-thereal}, we have $\outrad(K)\le \sqrt{2r-r^2}$, so we get $(1-r)^2\le 2r-r^2$ which implies $1-\sqrt{\frac12}\le r$ as required.
\end{proof}

\begin{rem}
In fact, Jung's inequality \eqref{eq:Jung}  gives the tight lower bound $1-\sqrt{\frac{n}{2(n+1)}}$ for the in-radius of a body of constant width $1$. Indeed, $\diam(K) = 1$ so $\outrad(K) \le \sqrt{\frac{n}{2(n+1)}}$ and so $\inrad(K) = \inrad(K^c) = 1- \outrad(K) \ge 1-  \sqrt{\frac{n}{2(n+1)}}$. These two estimates, are, however, asymptotically the same. Moreover,  the previous argument  also  gives a simple proof for a Jung-type result since every body of diameter $1$ is a subset of a body of constant width $1$. 
\end{rem}

This elementary bound already gives a simple lower bound for the volume of a body of constant width $1$. The best lower bound is a well known open problem called the Blaschke-Lebesgue problem. We discuss this and other bounds in Section \ref{sec:constantwidth}. 

Theorem \ref{thm:santalo-thereal} has a similar but not identical, analogous fact regarding the diameter of a body in $\S_n$ with a fixed in-radius. Its proof is much simpler.  

\begin{prop}\label{prop:extremality of diam fix inradius}
	Let $K\in \S_n$ and assume $\inrad(K) = r\in [0,1]$. Then 
    $\diam(K)\le 2\sqrt{2r-r^2}$, with equality  attained for example for a $1$-lens of in-radius $r$ and for an $(n-1)$-lens of in-radius $r$. 
\end{prop}

\begin{proof}
	Let $\diam(K)=d$. Then there are two points $x,y\in K$ with $\|x-y\|_2=d$  and therefore $\cconv\{x,y\}	\subseteq K$. This convex hull is a $1$-lens, and if $d>2\sqrt{2r-r^2}$ then this $1$-lens has in-radius greater than $r$,  a contradiction. 
\end{proof}

\subsection{Extremizing volume for given $r(K)$ or $R(K)$} 

In his paper \cite{Bezdek2021}, Bezdek proves inequalities connecting volume and inner and outer radii for ball polytopes (from which these follow for all ball bodies). We quote his results in our notations. 
The following theorem is a re-writing of  \cite[Theorem 1]{Bezdek2021}. 
\begin{thm}[Bezdek] 
Let $K\in \S_n$ and let ${\rm Inrad}(K) = r$. Assume $L_{n-1}$ is an $(n-1)$-lens with ${\rm Inrad}(L_{n-1}) = r$ Then 
\[ \vol(K) \le \vol(L_{n-1}).\]
Equivalently, of all bodies $K\in \S_n$ with fixed volume, the one with minimal in-radius is the $(n-1)$-lens. 	
\end{thm}

The opposite direction is the following theorem (which is \cite[Theorem 7]{Bezdek2021}). 
\begin{thm}[Bezdek] 
Let $K\in \S_n$ with $\outrad(K) = R$, and let $L_1$ be a $1$-lens with 
$\outrad(L_1) = R$ (namely the $c$-hull of two points at distance $2R$). Then 
\[ \vol(L_1) \le \vol(K).\]
Equivalently, of all $K\in \S_n$ with fixed volume, the one with maximal out-radius is the $1$-lens. 	
\end{thm}

For both theorems, Bezdek conjectured \cite[Conjectures 5,10]{Bezdek2021} that the same is true when volume is replaced by $V_k$, the quermassintegrals of various orders. He also provides some non-sharp bounds for these quantities. 
In \cite{drach2023reverse} Drach and Tatarko prove the ``The Reverse Inradius Inequality'' which is an instance of one of Bezdek's conjectures.

\begin{thm}[Drach-Tatarko] 
Let $n\ge 2$ and let $K\in \S_n$ with $\inrad(K) = r$, and let $L$ be an $(n-1)$-lens with 
$\inrad(L_{n-1}) = r$. Then 
\[ \vol_{n-1}(\partial K) \le \vol_{n-1}(\partial L_{n-1}).\]
Equivalently, of all $K\in \S_n$ with surface area volume, the one with minimal in-radius is the $(n-1)$-lens. 	
\end{thm}

\subsection{The intersection of $K$ and $K^c$} 

By Lemma \ref{lem:inplusout}, for a body $K\in \S_n $ the sets $K$ and $K^c$ always intersect, since the center of the out-ball of a convex $K$ is always in $K$, and of course the center of the in-ball of $K^c$ belongs to $K^c$. 
	In particular $K$ and $K^c$ are always contained, together, in some ball (any ball with center in $K\cap K^c$) which means their union has non trivial $c$-hull (that is, $c$-hull different that $\RR^n$). This raises several natural problems,  such as 
comparing the volumes of 
$K, K^c$, $K\cap K^c$ and $\cconv (K, K^c)$.  

Bounding the volume of the intersection from above is immediate

\begin{prop}\label{prop:bding-vol-intersec-above}
Let $K\in \S_n$, then 
\[  \vol(K\cap   K^c) \le  \sqrt{\vol(K)\vol(K^c)} \le  2^{-n} \vol(B_2^n),\]
with equality in the first inequality if and only if $K=K^c$ and in the second inequality if and only if $K$ is a translate of $\frac12 B_2^n$. Similarly 
\[  \vol(K\cap   -K^c) \le  \sqrt{\vol(K)\vol(K^c)} \le  2^{-n} \vol(B_2^n),\] 
with equality in the first inequality if and only if $K=\frac12 B_2^n$.\end{prop}

\begin{proof}
  We use the Santal\'o-type inequality  Lemma \ref{lem:santalo} together with the arithmetic geometric means inequality.
The equality cases are trivial.  
    \end{proof}

If we wish to bound the volume of these intersections from below, a first simple observation is captured in the following lemma (note that for the intersection with $-K^c$ one must allow for a translation, since when translating $K$ to $x_0 + K$ the body $-K^c$ is translated to $-x_0 + K^c$ and these might not intersect at all.  

\begin{lem}\label{lem:ball in intersection}
Let $n\in \mathbb N$. For any body $K\in \S_n$, 
we have $K\cap K^c \neq \emptyset$. Moreover, letting $r= \inrad (K)$  and $R = \outrad(K)$ it holds that
$ \inrad (K\cap K^c) \ge \rho(r,R)$  where 
\[ \rho(r,R) =   
 \max (
\min (1-R, 1-\sqrt{1-R^2})
,
\min (r, 1-\sqrt{1-(1-r)^2}) 
 ) .
 \]
Moreover, there exists some $x_0\in \RR^n$ such that letting  $\tilde{K} = K+x_0$ 
we also have 
$ \inrad (K\cap -K^c) \ge \rho(r,R)$. 
\end{lem}
 
\begin{proof}
Without loss of generality (or, after a proper translation) we may assume, say, 
$K \subseteq B(0, R)$ and we have  $B(0, 1-R)\subseteq K^c$. By Lemma \ref{lem:contactopintsout}  the contact points $C = S(0,R)\cap K$ form a set which intersects every closed hemisphere of $S(0,R)$, so in particular $0$ belongs to $\conv(C)\subset K$. 

The set $\cconv(C)\in \S_n$, which is also a subset of $K$, has out-radius $R$ (by Lemma \ref{lem:contactopintsout} applied to $C^{cc}$) and so by Theorem \ref{thm:santalo-thereal} it holds that 
$\inrad (C^{cc}) \ge 1-\sqrt{1-R^2}$ and in fact 
$B(0, 1-\sqrt{1-R^2})\subseteq K$. We see that a ball of radius $\min (1-R, 1-\sqrt{1-R^2})$ centered at $0$ is a subset of both $K$ and $K^c$. Since the ball is centrally symmetric, it is also a subset of $-K$ and $-K^c$. 

Applying the same reasoning to the inclusions $B(0,r)\subseteq K$ and 
$K^c  \subseteq B(0,1-r)$ (which hold after a proper translation) we see that (after a possibly different translation) a ball of radius 
$\min (r, 1-\sqrt{1-(1-r)^2})$ is a subset of both $K$ and $K^c$ (and thus also of $-K$ and $-K^c$). This completes the proof of the lemma. 
\end{proof}

To bound the volume of the intersection $K\cap K^c$ in terms of the volumes of $K$ and $K^c$ we can use the well known inequality by Milman and Pajor \cite{milman1989isotropic}
\begin{prop}\label{prop:MP} Let $K,L\subset\RR^n$ be two convex bodies with the same baricenter. Then
	\[ \vol(K)\vol(L) \le \vol(K-L) \vol(K\cap L).\]
\end{prop}

Using it directly, together with the fact that $K-K^c = B_2^n$, we see that we get a useful bound for $\vol(K\cap K^c)$ only under the assumption that $K$ and $K^c$ share a barycenter. (If we assume central symmetry $K=-K$, for example, then indeed the barycenters  of $K$ and $K^c$ both lie at the origin.) Nevertheless, if we use this bound to estimate $\vol(K\cap -K^c)$, after an appropriate translation, then it is applicable for all $K$. It is perhaps not so surprising, since 
  $K = K^c$ means that $K$ is of constant width $1$, but $K = -K^c$ implies that $K = \frac12 B_2^n$ (since $K- K^c = B_2^n$ for any $K\in \S_n$). So,  finding a bound for the ``measure of non-ball-ness'' is at times easier than a bound for ``measure of non-constant-width-ness''. (In particular this includes the centrally symmetric case, since the only centrally symmetric convex bodies of constant width are balls.) We let $\kappa_n = \vol_n(B_2^n)$.

\begin{thm}
Let $n\ge 2$ and let $K\in \S_n$ such that $K$ and $K^c$ have the same barycenter. Then
\[\left(\frac{\vol(K)}{\kappa_n}\right)\left(\frac{\vol(K^c)}{\kappa_n}\right) \le   \frac{\vol(K\cap K^c)}{\kappa_n}. \]
For {\em any} $K\in \S_n$ there  exists some $x_0$ such that for $\tilde{K} = K + x_0$  it holds that 
\[\left(\frac{\vol(K)}{\kappa_n}\right)\left(\frac{\vol(K^c)}{\kappa_n}\right) \le   \frac{\vol(\tilde{K}\cap -\tilde{K}^c)}{\kappa_n}. \]
\end{thm}

\begin{proof}
The first inequality 
follows from Proposition \ref{prop:MP} and the fact that $K-K^c = B_2^n$. For the second  inequality, we translate the body $K$ such that the origin lies halfway on the interval connecting the  barycenter of $K$ and the barycenter of $K^c$, so that the barycenters of $K+x_0$ and $-K^c-x_0$ coincide. We then apply Proposition \ref{prop:MP} as above, together with the well known fact (which we have already mentioned, and which follows for example from Urysohn's inequality) that for a body of constant width $2$, volume is maximized when the body is a Euclidean ball (recall that by 
Proposition \ref{prop:KminusK} the body $K+K^c$ has constant width $2$). 
We get (having replaced $\tilde{K}$ by $K$ where the volume is unaffected) 
\[\frac{\vol(K)\vol(K^c)}{\kappa_n} \le \frac{\vol(K)\vol(K^c)}{\vol(K+K^c) } \le  \vol(\tilde{K}\cap -\tilde{K}^c). \] This completes the proof. 
\end{proof}

While the set $K\cup -K^c$ is non convex in general, it is sometimes instructive to consider it as well. In terms of volume, this is not different than bounding the intersections, since 
\begin{eqnarray*}
\vol(K\cup \pm K^c) &=& \vol(K) + \vol(K^c)-\vol(K\cap \pm K^c)
. 
\end{eqnarray*}
It was shown by Schramm  \cite{schramm1988constantwidth} that for constant width bodies, i.e. $K=K^c$, for which the center of the in-ball is assumed at the origin, there is a lower bound for $\vol(K\cup - K^c)$ and in fact  $K\cup - K^c$ contains a certain ball.

\begin{thm}[Schramm]
Let $n\ge 2$. For $K\in \S_n$  satisfying $K = K^c  \subseteq B(0,R)$ we have 
\[ 
B(0, \sqrt{1-R^2 + 1/4}- 1/2) \subseteq K\cup -K. 
\]
In particular, as $\diam(K) = 1$, we have $R\le \sqrt{\frac{n}{2(n+1)}}$ and therefore \[B(0, \sqrt{\frac34+\frac{1}{2(n+1)}}- \frac12 )\subseteq K\cup -K.\]
\end{thm}

It turns out that Schramm's proof carries over to the general case of $K\neq K^c$, and we present this result with the proof (which is completely analogous to Schramm's argument) 

 \begin{thm}\label{thm:Schramm-general}
Let $n\ge 2$, $K\in \S_n$, let $0<R<1$ and assume $K, K^c \subset B(0,R)$. Then  
\[ 
B(0, \sqrt{\frac{5}{4}-R^2}- \frac12) \subseteq K^c\cup -K. 
\]
\end{thm}

 \begin{proof}
Denote $g( t) = \sqrt{1- R^2+t^2}-t$. Then $g:\RR \to \RR^+$ is positive, decreasing,  and convex.  Fix $u\in S^{n-1}$.  
We claim that $au\in K^c$ for 
\begin{equation}\label{eq:upboundfrac}   a(u) = \sqrt{1- R^2+h_K^2(-u)}-h_K(-u). \end{equation}
Indeed, take $x\in K$ so that $\|x\|_2\le R$ and $-\iprod{x}{u}\le h_K(-u)$
both hold. Therefore 
\begin{eqnarray*} \|au-x\|_2^2 &=& a^2 - 2a\iprod{x}{u}+\|x\|_2^2 \le a^2 + 2ah_K(-u)+R^2 \\&=& 
(h_K(-u) + a)^2 +R^2 - h_K^2(-u)=  
1\end{eqnarray*}
for our choice of $a = a(u)$. This being true for any $x\in K$ implies $a(u)u\in K^c$. 
Similarly, $-b(u)u\in K^{cc} = K$ (equivalently, $b(u)u\in -K$) for 
\begin{equation}\label{eq:upboundfrac} b(u) = \sqrt{1- R^2+h_{K^c}^2(u)}-h_{K^c}(u). \end{equation}
Finally, we claim that for all $u\in S^{n-1}$ we have $\rho = \sqrt{\frac54 -R^2} - \frac12 \le \max(a(u),b(u))$. Indeed, 
\begin{eqnarray*}
&&\max (a(u), b(u))   \\
&& =  \max (\sqrt{1- R^2+h_K(-u)^2}-h_K(-u), \sqrt{1- R^2+h_{K^c}(u)^2}-h_{K^c}(u))
)\\
&& \ge \frac{1}{2}\left( 
\sqrt{1- R^2+h_K(-u)^2}+ \sqrt{1- R^2+h_{K^c}(u)^2}-h_K(-u)-h_{K^c}(u))
\right) 
\\
&& \ge \sqrt{1 - R^2+ \left(\frac{h_{K^c}(u) + h_K(-u)}{2}\right)^2} - \frac{h_{K^c} (u) + h_K(-u)}{2} = \sqrt{\frac{5}{4} - R^2  } - \frac{1}{2}
\end{eqnarray*}
 where we have used the convexity of $g$. Therefore $\rho B_2^n \subset K^c\cup -K$, as claimed. 
\end{proof}

\subsection{$c$-Mahler in the plane}


To conclude this section, we return to our starting point, namely 
the Borisenko Conjecture \ref{conj:reverse-isop-twod}. In  
 \cite{BorisenkoDrach2014} it is proved in $\RR^2$. In the notation of mixed volumes their theorem is the following. 
 
\begin{thm}\label{thm:tatarko-dim2}
Let $K\in \S_2$ and let $V=\vol_2(K) \le \pi$. Then $V(K,B)\le V(L,B)$ where $L$  is a lens of area $V$. 	
\end{thm}

As a consequence, we get a Mahler-type result in the plane, that is, fixing the area of a body, a lower bound for the area of its $c$-dual.  

\begin{cor}\label{cor:whoisMahlerR2}
	Let $K\in \S_2$ with $\vol_2(K) = V\le \pi$, and let $L$ be the lens of area $V$. Then 
	\[ \vol_2(L^c) \le \vol_2(K^c).\]
\end{cor}

\begin{proof}
	Using that $K - K^c = B$ and the linearity of mixed volumes we have
	\begin{eqnarray*}
		\vol_2(K^c) &=& V(K^c, K^c) = V(K^c-K, K^c-K) - 2V(K^c, -K) -V(K,K) \\&=& V(B,B)-2V(K^c-K, -K) + 2V(-K,-K) - V(K, K) \\
		& = & \pi - 2V(B,K) + V, 
	\end{eqnarray*}
so that minimizing $\vol_2(K^c)$ under the restriction $\vol_2(K) = V$ amounts to maximizing $V(K,B)$ under the same restriction, which by Theorem \ref{thm:tatarko-dim2} is maximized by a lens. 
\end{proof}

Without the area restriction, we can write the following consequence, which can also be considered as a Mahler-type result in the plane.  

\begin{cor}\label{cor:mahler-plane}
For any $K\in \S_2\setminus \{\emptyset,\RR^2\}$ one has 
\[ \sqrt{2\pi-4}  \le
\vol_2(K)^{1/2}+ \vol_2(K^c)^{1/2}
\le \pi^{1/2} \]
with equality on the left hand side for the self-dual lens (with diameter $\sqrt{2}$ and angle $\pi/2$) and on the right hand side for the ball of radius $\frac12$.  
\end{cor}

\begin{proof}
The right hand side inequality is simply Lemma \ref{lem:santalo}.
 For the left hand side, note that by Corollary \ref{cor:whoisMahlerR2} the only bodies to consider as minimizers are lenses. For a lens of angle $\theta$ (i.e. perimeter $2\theta$), the dual lens has angle $\pi-\theta$ and their areas are $\theta - \sin(\theta)$ and $\pi-\theta - \sin(\theta)$. The function $\sqrt{\theta - \sin(\theta)} + \sqrt{\pi -\theta - \sin(\theta)}$ 
 attains its (unique) minimal value (of $\sqrt{2\pi-4}$) at $\theta =  \pi/2$ (see Remark \ref{rem:minimization-ex} in Section \ref{sec:on-k-lenses}). 
\end{proof}

\section{Boundary Structure}\label{sec:extremal-structure}

We are interested in the boundary structure of sets $K\in \S_n$. Since they are convex bodies, we can apply results from the theory of convex bodies to these special bodies. However, on top of the fact that they possess specialized features, we also use here a different notion of ``convex hull'' (namely $c$-hull) so that some parts of the theory are developed anew.

\subsection{Extremal Points}

\begin{definition}
	Let $n\ge 2$ and $K\in \S_n$. A point $x\in K$ is called $c$-extremal for $K$ if $x\in {\conv}_c (y,z)$ for $y,z\in K$ implies $y = x$ or $z = x$. We denote the set of $c$-extremal points of $K$ by ${\rm ext}_c(K)$. 
\end{definition}

It will sometimes be easier to use the following equivalent definition 

\begin{lem}\label{lem:extremalist-in-arc-language}
   Let $n\ge 2$ and $K\in \S_n$. A point $x\in \partial K$ is  $c$-extremal for $K$ if and only if $x$   does not belong to an open $1$-arc $A\subset \partial K$. \end{lem}

\begin{proof}
One direction is obvious since if there exists an open $1$-arc $A\subset \partial K$ with $x\in A$ then one may find two points $x\neq y,z\in K$ on this arc with $x\in \cconv(y,z)$ so that $x$ is not $c$-extremal for $K$. For the other direction, assume $x$ is not $c$-extremal for $K$, and consider $x\neq y,z\in K$ with $x\in \cconv(y,z)$. Since $K\in \S_n$ we have that $\cconv(y,z)\subset K$, and $x\in   \cconv(y,z)\cap \partial K$ so that $x\in \partial \cconv(y,z)$. There is a unique $1$-arc in $\partial \cconv (y,z)$ connecting $y$ and $z$ and passing through $x$. Moreover, had this $1$-arc contained a point which is not in $\partial K$ then this would mean that $x$ is in the interior of $K$ (indeed, if $a\in K$ and $b\in {\rm int}(K)$ then  
$\cconv(a,b)\setminus \{a\}\subset {\rm int}(K)$). This means that the $1$-arc on the boundary of $\cconv (y,z)$ is also on the boundary of $K$, which completes the proof.
\end{proof}

\begin{lem}\label{lem:extremals=and=containment}
    Let $n\ge 2$ and  $K_1,K_2\in \S_n$ with $K_1 \subseteq K_2$. Then $K_1\cap {\rm ext}_c(K_2)\subseteq {\rm ext}_c(K_1)$.
\end{lem}

  \begin{proof}
     Let $x\in K_1$ be $c$-extremal for $K_2$. If $x\in \cconv(y,z)$ with $y,z\in K_1$ then in particular $y,z\in K_2$ so by assumption either $x=y$ or $x=z$. This means $x\in {\rm ext}_c(K_1)$ as well. 
  \end{proof}

 \begin{lem}\label{lem:ext-cont-pts}
     Let $n\ge 2$, $K\in \S_n$ and $R<1$. Assume $K\subseteq B(x_0,R)$. Then the contact points $S(x_0 ,R)\cap K$ are extremal. 
 \end{lem}

\begin{proof}
Let $x\in S(x_0, R)\cap K$. Since $R<1$ we have $B(x_0, R) \in \S_n$, and by Lemma \ref{lem:extremals=and=containment}
it suffices to show that $x$ is  $c$-extremal for $B(x_0,R)$. We use Lemma \ref{lem:extremalist-in-arc-language}, and the obvious fact that no open $1$-arc passing through a   point in $S(x_0,R)$ can be contained in $B(x_0,R)$ when $R<1$. 
\end{proof}

\begin{thm}\label{thm:only ball has no extremal points}
Let $n\ge 2$. The unit balls $B(x,1)$, $x\in \RR^n$  are the only  sets in $\S_n$ with no extremal points.  	
\end{thm}

\begin{proof}
	The  ball $B(x, 1)$ has no extremal points since it is the $c$-hull of any two antipodal points. 
	Any other set in $\S_n$ has out-radius $R<1$, and by Lemma \ref{lem:contactopintsout} has contact points with its out-ball, which by Lemma \ref{lem:ext-cont-pts} are $c$-extremal for $K$. 
\end{proof}

\begin{prop}
 Let $n\ge 2$. The only sets $K\in \S_n$ with just one $c$-extremal point are the singletons $\{x\}$.  
\end{prop}

\begin{proof}
Such a $K$ cannot be a ball $B(x, 1)$ since it has no $c$-extremal points. Thus $R = \outrad(K) <1$. If $R>0$, the set of contact points of $K$ with its out-ball has at least two elements by Lemma \ref{lem:contactopintsout} (3). For $R= 0$ the set is clearly a singleton $\{x\}$, and by definition the points $x$ is trivially $c$-extremal.  
\end{proof}

\begin{rem}
The fact that if $K\in \S_n$ has exactly two $c$-extremal points then it must be a $1$-lens will follow for example from our Caratheodory-type 
    Theorem \ref{thm:KM-Caralight} in Section \ref{sec:Carath}.  
\end{rem}

\begin{prop}\label{prop:not-extremal-means-smooth}
  Let $n\ge 2$,  $K\in\S_n$ and let $x\in \partial K$ have a non-trivial normal cone (i.e., $x$ is not a smooth point). Then $x$ is $c$-extremal for $K$. In other words, if $x\in \partial K\setminus {\rm ext}_c(K)$ then $x$ is a smooth boundary point.  	
 \end{prop}

 \begin{proof}
Assume that $x$ has more than one unit normal with respect to $K$, namely $u_1 \neq u_2 \in N_K(x)$. By Lemma \ref{uniq:signletons-to-singletons} we have $K\subseteq B(x-u_1, 1)\cap B(x-u_2, 1)=:L$. By Lemma 
\ref{lem:extremals=and=containment} it suffices to show that $x$ is $c$-extremal for the $(n-1)$-lens $L$, however for the $(n-1)$ lens $L$ the set $S(x-u_1, 1)\cap S(x-u_2, 1)$ (which is a sphere of lower dimension, containing $x$) is exactly the set of contact points of $L$ with its out-ball, and hence by Lemma \ref{lem:ext-cont-pts} consists of $c$-extremal points for $L$.  
\end{proof}

 \subsection{Duality and $c$-extremality}

It turns out that $c$-extremal points in a body $K\in \S_n$ correspond to extremal rays of the normal cones  for points in the boundary the $c$-dual $K^c$.

 \begin{prop}\label{prop:extremal-ray-and-extremality}
 Let $n\ge 2 $ and let $K\in \S_n$ with $\outrad(K)<1$. 
 Then \begin{equation}\label{eq:extremal-points-and-extremal-rays} {\rm ext}_c(K) = \{ y - u: y\in \partial K^c, ~u\in   (N_{K^c}(y))  {\rm~is~extremal}   \}.\end{equation}
Here $u$ is called extremal  for $N_{K^c}(y)$ if in the normal cone $u$ spans an extremal ray. 
 \end{prop}

\begin{proof}
 Assume $u\in S^{n-1}$ and $\RR^+u$ is an extremal ray of the cone $\RR^+N_{K^{c}}(y)$ where $y\in \partial K^c$. 
 Then by Lemma \ref{lem:normals map boundary to boundary} it holds that $x = y-u \in \partial K$. We consider two cases. If there is some $y'\neq y$ with $y'\in \partial K^c$ and $\|x-y'\|_2=1$, then $x$ is not a smooth point of $\partial K$ (since both $x-y$ and $x-y'$ belong to $N_K(x)$, again from Lemma \ref{lem:normals map boundary to boundary}) and by Proposition \ref{prop:not-extremal-means-smooth} we see that $x$ is $c$-extremal for $K$. If, however, $x$ is a smooth point and $y$ is the only point in $K^c$ with $\|x-y\|_2=1$ then had $x$ belonged to an open $1$-arc contained in $\partial K$,  
 this arc must be centered at $y$ (as it is part of the unit ball supporting $K$ at  $x$), which would mean the normal cone   $N_{K^c}(y)$ for $K^c$ at $y$  does not have $u$ spanning an extremal direction.  Therefore no such open arc exists, and by Lemma \ref{lem:extremalist-in-arc-language} the point $x$ is  $c$-extremal for $K$. 
 We have shown that the right hand side in \eqref{eq:extremal-points-and-extremal-rays} is included in the left hand side.

 The other inclusion works similarly and is simpler. 
 Let $x\in {\rm ext}_c(K)$ and take some $w\in N_K(x)$. By Lemma \ref{lem:normals map boundary to boundary} the point $x-w  = y$ belongs to $\partial K^c$ and $u=-w\in N_{K^c}(y)$. If $u$ is not an extremal ray for $N_{K^c}(y)$ then there are two other unit vectors $u_1, u_2 \in N_{K^c}(y)$ such that $u$ is proportional to $\frac12 (u_1 + u_2)$, in which case $y-u_1$ and $y-u_2$ both belong to $\partial K$ (using Lemma \ref{lem:normals map boundary to boundary} again) and span an open arc on $\partial K$ to which $x$ belongs, which cannot occur as $x$ was assumed to be $c$-extremal for $K$. 
\end{proof}

\begin{rem}
We can use the above arguments to see once again that if a body has no extremal points it must be a ball $B(x,1)$. Indeed, the normal cones for the boundary points of its dual have no extremal rays which can happen only if the dual is a point.     
\end{rem}

\begin{rem}
While non-extremal points are always smooth, extremal points can be smooth (such as the boundary points of the ball $B(0,1/2)$),  or not. There are however restrictions regarding a pair of ``dual'' points. 
\end{rem}

\begin{lem}\label{lem:pairings}
    Let $n\ge 2$ and $K\in \S_n$. Assume  $x\in \partial K$ and $y\in \partial K^c$ satisfy $\|x-y\|_2=1$. Then either both $x$ and $y$ are $c$-extremal (for $K$ and $K^c$ respectively) or one of them is $c$-extremal and the other one not. Moreover, if both are smooth points then both are $c$-extremal smooth points. 
\end{lem}

\begin{proof}
A pair of points $x\in \partial K$ and $y\in \partial K^c$ with $\|x-y\|_2=1$ satisfy that $y-x\in N_{K^c}(y)$ and $x-y \in N_K(x)$. If $x$ (say) is not extremal then it is smooth. Denoting the normal to $K$ at $x$ by $u$, we get that $y=x-u$ and that the $1$-arc testifying to $x$'s non-extremality  is part of the sphere $S(y,1)$. For any $z=y+w$ in this arc, $-w\in N_{K^c}(y)$, which means $y$ is not a smooth point of $K^c$ and in particular $y$ is $c$-extremal for $K^c$. This completes the proof of the first assertion. 

For the ``moreover'' part, assume $\|x-y\|_2 = 1$ and that both points are smooth. By the above proof, any $1$-arc testifying to the non-extremality of one of these points (say $x$) would imply the non-smoothness of the other (in this case, $y$) point. Since both are assumed smooth, both are $c$-extremal.     
\end{proof}

\begin{rem}
One may construct a body $K\in \S_2$ and a pair of points $x\in \partial K$, $y\in \partial K^c$ such that both are $c$-extremal (for $K$ and $K^c$ respectively), one of them is smooth (namely admits just one normal) and the other is not. We thus have the following possibilities for a pair of points $x\in \partial K$ and $y\in \partial K^c$ with $\|x-y\|_2 =1$:\\
1) Both are $c$-extremal and smooth (e.g. in a pair of $1/2$-balls)\\
2) Both are $c$-extremal and not-smooth (e.g. in a $1$-lens and and $(n-1)$-lens)\\
3) One of them is not $c$-extremal and the other is $c$-extremal and not smooth (again in a $1$-lens and and $(n-1)$-lens, a different pair)\\
4) One of them is $c$-extremal and smooth and the other is $c$-extremal and not smooth (see Example \ref{ex:Naztel-body}). 
\end{rem}

\begin{exm}\label{ex:Naztel-body}
    Following the construction in the paper \cite{arman2024small} we consider the $c$-hull of two sets, $(\RR^+)^2\cap B(0,R)$ and $(\RR^-)^2\cap B(0,r)$, with $R = 1-1/\sqrt{2}$ and $r = 1/\sqrt{2}$ so that $R+r = 1$. This body is of constant width $1$, and is thus self-$c$-dual. The dual pair of of points $x = (R,0)$ and $y = (-r, 0)$
satisfy that $x$ is not smooth whereas $y$ is smooth. We see here that $x-y$ is indeed an extremal ray for the normal cone to $K$ at $x$. \end{exm}

 \begin{figure}[t]
     \centering
     \begin{tikzpicture}[scale=2]
         \draw[->] (-1, 0) -- (1, 0) node[right] {$x$};
         \draw[->] (0, -1) -- (0, 1) node[above] {$y$};

         \def\rone{1 - 1/1.4142}  
         \def\rtwo{1/1.4142}      

         \def\cx{\rone}  
         \def\cy{\rone}  

         \draw[thick] (\rtwo,0) arc (0:90:\rtwo);

 \draw[thick] (180:\rone) arc (180:270:\rone);
 \draw[thick] 
        (0,\rtwo) 
        arc[start angle=135, end angle=180, radius=1];

  \draw[thick] 
        (0,-0.2928) 
        arc[start angle=-90, end angle=-45, radius=1];
     \end{tikzpicture}
     \caption{The set in Example \ref{ex:Naztel-body}}
    \label{fig:NAZETAL}
 \end{figure}
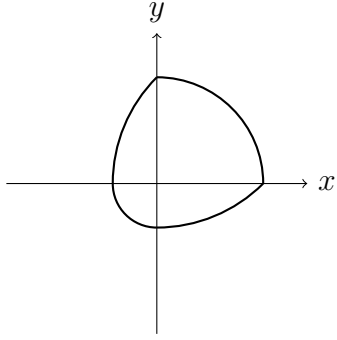

 \begin{rem}
 While Lemma \ref{lem:pairings} gives a certain intuition that smoothness of $K$ has to do with smoothness of its polar, it is instructive to note an example of a body which is (globally) smooth, but its dual is not. Indeed, to this end we can intersect two $0.99$ balls with some properly chosen centers. This is clearly not a smooth body (the points where the boundaries intersect will have non-trivial normal cones). The dual of this intersection is the $c$-hull of two $0.01$-balls with the same centers as the large balls, and this $c$-hull is easily checked to be smooth. 
 \end{rem}

\subsection{Carath\'eodory-type theorems}\label{sec:Carath}

Carath\'eodory's theorem states that a point in the convex hull of a set $A\subset \RR^n$ can be expressed as a convex combination of $(n+1)$ points in $A$. The counterparts to Carathe\'odory's theorem for $c$-hulls work out quite well. The reason is that we have the correspondence between normal cones (in the dual) and boundary points. In particular, boundary parts which are parts of spheres, must be parts of convex cones.

\begin{thm}\label{thm:KM-Caralight}
	Let $n\ge 2 $ and $K\in \S_n$ with $\outrad(K)<1$.  For any $x\in \partial K$ there exist $\{ x_i\}_{i=1}^m \subseteq {\rm ext}_c(K)$, $m\le n$ such that 
	\[ x\in \cconv (\{ x_i\}_{i=1}^m ).\]
\end{thm}

\begin{proof}
	Let $x\in \partial K$. If $x$ is not a regular point, namely $N_K(x)$ is not a singleton, then by Lemma \ref{prop:not-extremal-means-smooth} the point $x$ itself is $c$-extremal and we can take $m=1$. 	Otherwise, let $u =  n_K(x)$ and consider $y = x-u$.   Proposition \ref{prop:extremal-ray-and-extremality} implies that $y\in \partial K^c$ and $-u\in N_{K^c}(y)$ does not span an extremal ray for the normal cone $\RR^+N_{K^c}(y)$.
 Therefore $u$ is in the convex hull of the cone, and by the classical Carath\'eodory theorem we can find $m\le n$ extremal rays $(u_j)_{j=1}^m\subset  N_{K^c}(y)$ of the cone such that $u$ is in the convex hull of these rays. This also means that  on the sphere $S^{n-1}$, $u$ belongs to the $c$-hull of $(u_j)_{j=1}^m\subset  S^{n-1}$. 
Using 
Proposition \ref{prop:extremal-ray-and-extremality} again we see that $y-u_j\in {\rm ext}_c(K)$, and since $u$  is in the $c$-hull of $(u_j)_{j=1}^m$ we get that $y-u$ is in the $c$-hull of 
$(y-u_j)_{j=1}^m$, as claimed. 		
\end{proof}

\begin{rem}\label{rem:unfounded}
    Within the proof we use the following fact regarding cones and $c$-hulls, which we would like to spell out explicitly.   Given $K\in \S_n$ with, say $0\in \partial K^c$, letting  $A =  K \cap  S^{n-1}$,   consider the  convex cone $C(A)$ spanned by $\RR^+A$, intersected with the sphere. Then 
    \[ C(A)\cap S^{n-1} = K\cap S^{n-1}.\]
    Indeed, $K\cap S^{n-1} = -N_{K^c}(0)$ and this is simply Lemma \ref{lem:normals map boundary to boundary}. 
 Moreover, the image of the convex hull of two rays in $N_{K^{c}}(0)$ is a $1$-arc on $\partial K\cap S^{n-1}$, and correspondingly, the image of the convex hull of any number of rays in $N_{K^{c}}(0)$ is the intersection of $S^{n-1}$ with the $c$-hull of these point. 
\end{rem}

\begin{thm}\label{thm:cara}
	Let $n\ge 2 $ and $K\in \S_n$ with $\outrad(K)<1$.  For any $x\in  K$ there exist $\{ x_i\}_{i=1}^m \subseteq {\rm ext}_c(K)$, $m\le n+1$ such that 
	\[ x\in \cconv (\{ x_i\}_{i=1}^m ).\]
\end{thm}

\begin{proof}
	Let $x_0\in {\rm ext}_c(K)$, and let $L\in \partial \cconv(x,x_0)$ a $1$-arc (which is incidentally an extremal set of the dual lens)  connecting $x$ and $x_0$. Continue this arc as a big-circle $C$ passing through $x$ and $x_0$. It cannot be the case that the whole circle is contained in $K$ since then $K$ would be a ball.  In fact, the antipodal point to $x_0$ on this circle must be outside $K$.  Since the part connecting $x_0$ and $x$ is in the interior of $K$, there will be a first point $x'$ ``after'' $x$ which is on the boundary of $K$. Clearly $x\in \cconv (x_0, x')$. By Theorem \ref{thm:KM-Caralight} we can find $x_1, \ldots, x_m \in {\rm ext}_c(K)$ with $x'\in \cconv (\{ x_i\}_{i=1}^m )$ and therefore 
	\[ x\in \cconv (x_0, \{ x_i\}_{i=1}^m )\]
	as claimed. 	
\end{proof}

In particular, we get  a Krein-Milman type theorem for $c$-hulls. 
\begin{cor}
Let $n\ge 2 $ and  $K\in \S_n$ with $\outrad(K)<1$. Then 
\[ K = \cconv ({\rm ext}_c(K)).
\]
\end{cor}

Since taking the $c$-dual of a set is the same as taking the $c$-dual of its $c$-hull, we get the following. 
\begin{cor}
Let $n\ge 2 $ and  $K\in \S_n$ with $\outrad(K)<1$. Then 
\[ K^c = {\rm ext}_c(K)^c.
\]
\end{cor}

Finally, to have a full analogue of Carath\'eodory's theorem, we prove that when taking a $c$-hull of a set, no $c$-extremal points are added. This will allow us to prove

 \begin{thm}\label{thm:cara-hard}
	Let $n\ge 2 $ and let $A\subset \R^n$ be closed with $\outrad(A)<1$.  For any $x\in \cconv(A)$ there exist $\{ x_i\}_{i=1}^m \subset A$, $m\le n+1$ such that 
	\[ x\in \cconv (\{ x_i\}_{i=1}^m ).\]
\end{thm}

Equivalently, we can prove the following
\begin{thm}\label{thm:cara-hard-neeD}
Let $n\ge 2 $ and let $A\subset \R^n$ be closed with $\outrad(A)<1$.
Then
\[
{\rm ext}_c(\cconv(A)) \subseteq A.
\]
\end{thm}
To see the equivalence: if $\cconv(A)$ would have some $c$-extremal point which is not in $A$, this would contradict Theorem \ref{thm:cara-hard} as it cannot be given as a $c$-hull of points different from it. On the other hand, once Theorem \ref{thm:cara-hard-neeD} is proved, we can use it together with Theorem \ref{thm:cara} to obtain a proof of Theorem \ref{thm:cara-hard}.

We will make use of the following lemma.

\begin{lem}\label{lem:helpful-for-carahard}
    Let $n\ge 2 $ and $A\subset \R^n$ closed with $\outrad(A)<1$. Let $y\in A^c$. Then 
\[ \cconv(A\cap S(y,1)) \cap S(y,1) = \cconv(A)\cap S(y,1). \] 
\end{lem}

 \begin{proof}[Proof of Lemma \ref{lem:helpful-for-carahard}] 
 We split $A  =  A_g \cup A_b$ where $A_g = A\cap S(y,1)$ and $A_b  = A\setminus S(y,1)$. Since $A$ is closed, if $y\not \in \partial A^c$ then $A_g = \emptyset$ and $S(y,1)\cap \cconv (A) = \emptyset$ as well (since there is some smaller ball $A\subset B(y,R)$ with $R<1$), so in this case the conclusion of the lemma holds.  
	
We may this assume $A_g\neq \emptyset$. 
Assume towards a contradiction that there exists a point $x\in \cconv(A)\cap S(y,1)$ which is not in $\cconv (A_g)$. Since $\cconv (A_g)$ is the intersection of all $1$-balls including $A_g$, this means there exists $z\in \RR^n$ such that $B(z,1)\supset A_g$ and $d(x,z)>1$. 
	We claim that for a small enough $\eps>0$ and some $\theta\perp (x-y)$, the ball $B(y+\eps \theta, 1)$ includes $A$ and does not include $x$. This would contradict $x\in \cconv (A)$. 
	
	To find this $\theta$ we note that $A_g \subset B(y,1)\cap B(z,1)$ which implies that for any $\psi$ in a cone of directions surrounding $z-y$, there exists an $\eps>0$ with 	
	\[ A_g \subset B(y,1)\cap B(z,1)\subset B(y+ \eps \psi,1).\]
	Since $x\in S(y,1)$ and $x\not\in B(z,1)$ we know that $(x-y)^\perp$ must intersect  this cone of directions (in its interior) and this is how we pick out $\theta$. 
	
	By our choice of $\theta$, and $\eps$ small enough, the ball we just defined includes all of $A_g$ and does not include $x$. We claim that by taking small enough $\eps$ we can also guarantee that it does   include $A_b$. Indeed, denote $\delta>0$ the minimal distance of a point in $A_b$ and the cap $S(y,1)\setminus B(y+\eps_0\theta)$. The fact that it is not zero follows from the fact that the only accumulation point of $A_b$ which do not belong to $A_b$ are in $A_g$, which has a positive distance to this ball (when $\eps_0$ is chosen small enough, since $A_g \subset B(y,1)\cap B(z,1)$).
	
	Clearly if $\eps<\min (\delta, \eps_0)$ then $A_b\subset B(y+\eps\theta)$
	as well. This gives a contradiction to the fact that $x\in \cconv(A)$, proving the lemma. 	
\end{proof}

\begin{proof}[Proof of Theorem \ref{thm:cara-hard-neeD}]
	Denote $K = \cconv (A)$, and by \eqref{eq:out-and-outc} we have $\outrad(K) <1$. Clearly $K^c = A^c$. Let $x\in \partial K$. 
 There is some $y\in A^c$ with $x\in S(y,1)\cap K$ (namely any $x-u$ where $u\in N_K(x)$).  By Lemma \ref{lem:helpful-for-carahard} we see that 
 \[ x \in \cconv(A)\cap S(y,1) =  \cconv(A\cap S(y,1))\cap S(y,1). \]
 This means that the normal cone to $A^c$ at $y$ is the cone-convex-hull of the rays $\RR^+ (a-y)$ for $a\in A\cap S(y,1)$. We get that $\RR^+(x-y)$ belongs to this cone, and   by the classical Carath\'eodory's theorem it is a combination of $m\le n$ of these extremal rays. In particular, it can only be an extremal ray if $x - y = a-y$ for some $a\in A$, namely $x\in A$. Since $c$-extremal points for $K$ correspond to extremal rays of the cones (by Proposition \ref{prop:extremal-ray-and-extremality}) we see that $x$ can be $c$-extremal for $K$ only if   it belongs to $A$. 
\end{proof}

In the classical theory for convex hulls, one can convexify a set in ``stages'', the first iteration is the  set  containing all segments connecting two points in the original set, the second iteration contains all segments connecting two points in the first iteration set, etc. It is easy to check, simply rearranging coefficients, that for a set in $\RR^n$, after approximately $\log (n)$ iterations, one achieves the convex hull of the original set. 
While in the world of $c$-hulls we cannot work with coefficients as easily, the same phenomenon holds. 

\begin{thm}\label{thm:iterative-c-hull}
 Let $n\ge 2 $ and let $A\subset \R^n$ be closed with $\outrad(A)<1$. 
 Let $A_0 = A$ and define for $j=1, 2, \ldots$ the sets
\[ A_{j+1} = \cup \{\cconv(x,y): x,y\in A_{j} \}. \]
Then for all $j$ we have $A_j \subseteq \cconv (A)$, and for 
     $2^j >n$ we have that $A_j = \cconv(A)$. 
\end{thm}

To prove it, we first note that for convex cones, a similar fact holds by the usual Carath\'eodory argument (splitting the hull in two at every step).

 \begin{lem}\label{lem:cara-cones}
Let $K\subset S(0,1)\subset \RR^n$ be a spherically convex set (namely $K$ is contained in an open half-space and $\RR^+K$ is a proper cone in $\RR^n$). Consider the extremal rays $(R_\alpha)_{\alpha \in I}$ of $\RR^+ K$ where $R_{\alpha} = u_\alpha \RR^+$ for some $u_\alpha \in S^{n-1}$. Let 
\[ K_0 = \cup_{\alpha \in I} R_\alpha\]
and 
\[ K_{j+1} = \cup \{ \conv (R, R'): R, R' \in K_j {\rm ~~rays} \}. \]
Then for $j$ with $2^j \ge n$ we have $K_{j+1} = \conv_{\alpha \in I
 } R_\alpha = K$.       
 \end{lem}

\begin{proof}[Proof of Theorem \ref{thm:iterative-c-hull}]
Let $K : = \cconv A $ and let $x\in \partial K$. Then $x = y-u$ for $y\in \partial A^c$ and some $u\in N_{A^c}(y)$. The cone $N_{A^c}(y)$ is proper (since $A^c$ is not flat, as it is not a point). The extremal rays of $N_{A^c}(y)$ are, by Proposition \ref{prop:extremal-ray-and-extremality}, of the form $y-z$ for $z\in {\rm ext}_c(K)$. By Theorem \ref{thm:cara-hard-neeD} this means $z\in A$. We see that $u = y-x$ is in the convex hull (in the cone sense) of points $y-z$. Using Lemma \ref{lem:cara-cones} we see that $y-x \in K_j$ whenever $2^j \ge n$. 
As we have seen above (see Remark \ref{rem:unfounded}), 
convex combinations in cones amount to $c$-hulls in the corresponding sphere, we see that $u$ belongs to the $j^{th}$ element in the iterative $c$-hull of ${\rm ext} (N_{K^c}(y))\subset S(y,1)$. Therefore $x$ belongs to the  $j^{th}$ element in the iterative $c$-hull of $y - {\rm ext} (N_{K^c}(y))\subset A \subset \partial K$, as needed.

So far we have included only $x$ in the boundary of $K$. With one more iteration we can make sure also points which are in the interior are obtained; indeed, any interior point is in the usual convex hull of two boundary points (and thus also in their $c$-hull). One may even force one of these boundary points to be any specified point, for example a given point in $A$. 
\end{proof}

\subsection{Curvature at a pair of dual points}

When $K$ and $K^c$ are both smooth bodies (namely having a unique normal at every point), there is a one to one correspondence between $x\in \partial K$ and $y\in \partial K^c$ given by $x\mapsto x-n_K(x)=y$ (and $y-n_{K^c}(y)=x$). There are no $1$-arcs on $\partial K$ or on $\partial K^c$. Nevertheless, the curvature at a point $x\in \partial K$ can be $1$. 

Indeed, taking any $C^2_+$ body $K'$ (namely with continuous $\nabla^2 h_{K'}$ on $\RR^n\setminus \{0\}$ and non-zero curvature) we can find the minimal value of its curvature and rescale the body to be in $\S_n$, which will necessarily produce a body $K = aK'\in \S_n$ which is also smooth. If we postulate that $K$ cannot have $1$-arcs on its boundary (which is the same as asking for $K'$ not to have a circular $r$-arc on its boundary with minimal curvature) then the $c$-dual of $K$ (which is automatically in $\S_n$ of course) will be smooth as well.    

It will follow from the discussion below that for $C^2$ bodies $K,K^c\in \S_n$, given a point $y(x)\in \partial K^c$ corresponding to a point $x\in \partial K$ of curvature $1$, the  curvature of $K^c$ at $y(x)$ will be infinite (i.e., a smooth point of curvature $+\infty$). Before we formulate this in a more precise fashion, let us discuss a specific example, namely the dual of an ellipse in $\S_2$, where this phenomenon occurs. 

\begin{exm}
    Consider the ellipse $E\subset \RR^2$ given by 
    \[ \{ (x,y): \frac{x^2}{a^2} + \frac{y^2}{b^2} \le 1\}. \]
The curvature at the point $(x,y)$ is
\[ \kappa ={\frac {1}{a^{2}b^{2}}}\left({\frac {x^{2}}{a^{4}}}+{\frac {y^{2}}{b^{4}}}\right)^{-{\frac {3}{2}}} \]
so (for $b<a$) the radius of curvature is between $b^2/a$ and $a^2/b$. We set the maximal radius of curvature to be $1$ i.e. $a^2=b$ and $b^2 < \sqrt{b}$ meaning $b < 1$.
So, the ellipse is
\[
E = 
\left\{ (x,y): \frac{x^2}{b} + \frac{y^2}{b^2} \le 1\right\}.
\]

\begin{center}\begin{tikzpicture}
\draw[thick, blue] (0,0) ellipse (3.5cm and 2.5cm);
    \draw[densely dashed] (-3.5,0) -- (3.5,0) node[anchor=north west] {$a=\frac1{\sqrt{2}}$};
    \draw[densely dashed] (0,-2.5) -- (0,2.5) node[anchor=south east] {$b=\frac12$};
\end{tikzpicture}
\end{center}
By Theorem \ref{thm:char-curv},  $E\in\S_2$. The ellipse $E$ is a convex body with all boundary points   smooth and $c$-extremal. The radius of curvature of $E$ at $(0,1/2)$ equals $1$, and this  implies, as we shall see below, that the radius of curvature of $E^c$ at $(0,-1/2)$ equals $0$, meaning the curvature of $E^c$ diverges at its smooth boundary point $(0,-1/2)$.
\begin{figure}[h]
    \centering
    \includegraphics[scale=0.35]{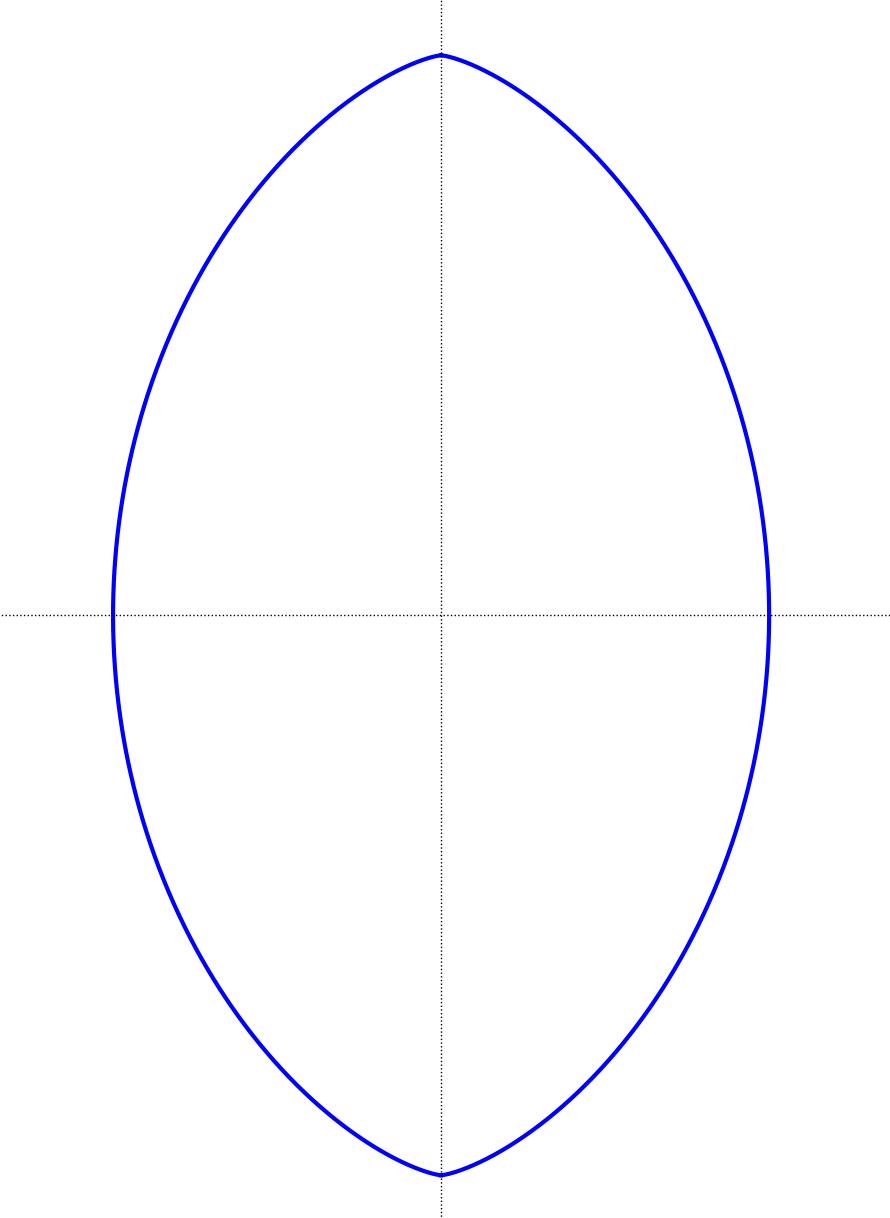}
    \caption{The $c$-dual of $E$, for $b=1/2$}
    \label{fig:ellipse-dual}
\end{figure}
Let us show it in this example explicitly. Parameterize the bottom half of $\partial E^c$ by taking a boundary point in the upper half of $\partial E$ and moving one unit along the (unique) inner normal. It turns out that the ``south pole'' $y=(0,b-1)$ of $E^c$ is a smooth point with $0$ radius of curvature, as the parametrization $(u(t),v(t))$ behaves near $y$ like
\[
v(u)\approx
-(1-b)
+ \frac32
\left(\frac{b}{2(1-b)}\right)^\frac13 \cdot |u|^\frac43.
\]
\end{exm}

The general phenomenon is that in $\RR^2$ the radii of curvature at (smooth) dual points sum to $1$, and similarly pairs of principal curvatures will sum to $1$ in higher dimensions.

More precisely we will see that for smooth points $x\in \partial K$ and $y\in \partial K^c$ with $y=x-n_K(x)$, the set of principal radii of curvature (ordered in an  increasing fashion)  satisfy $r_i^K(x)+r_{n-i}^{K^c}(y)=1$. This follows below from the relation 
$h_K(u) + h_{K^c}(-u) = 1$, and should be compared to the relation in Lemma \ref{lem:pairings}, where  $x$ lay in the interior of a $1$-arc, and then $y$ had to be a non-smoothness point (and in particular $r_1^{K^c}(y)=0$). We prove the following theorem (for an elaborate discussion on radii of curvatures see \cite[Section 2.5]{schneider2013convex}). 

\begin{thm}\label{thm:chakerian-dual-curvatures}
     Let $K\in \S_n$ and assume $u\in S^{n-1}$ is such that $h_K$ is twice continuously differentiable in a neighborhood of $u$. Then, letting $x = \nabla h_K(u)\in \partial K$ and $y = x-u\in \partial K^c$, and letting $0\le r_1\le \cdots \le r_{n-1}\le 1$ and $0\le s_1\le \cdots \le s_{n-1}\le 1$ be the principal radii of curvature of $K$ at $x$ and of $K^c$ at $y$, respectively, we have 
     \[ r_i + s_{n-i} = 1, \quad i=1, \ldots, n-1.\]
 \end{thm}

This theorem generalizes a result by Bonnesen and Fenchel pertaining to a self dual body (i.e. a body of constant width). See the discussion in Bonnesen and Fenchel \cite[Chapter 15, Section 63 in page 163]{BonnFen} and also Chakerian \cite{Chakerian}.

\begin{proof}[Proof of Theorem \ref{thm:chakerian-dual-curvatures}]

Recall  that as $K\in \S_n$ is strictly convex, we have that  $h_K\in C^1(\RR^n \setminus \{0\})$, and we saw in \eqref{eq:nablah_K} that  
$\nabla h_K(u) - \nabla h_{K^c}(-u) = u/|u|$, which can be interpreted as a relation between a pair of  dual points by letting $\nabla h_K(u) = x\in \partial K$ and $\nabla h_{K^c}(-u) = y\in \partial K^c$. 

Assume that   $h_K \in C^2(\RR^n \setminus \{0\})$. By Proposition \ref{prop:KminusK} this means that also $h_{K^c}\in C^2(\RR^n \setminus \{0\})$, and we can differentiate \eqref{eq:nablah_K} again to get
\begin{equation}\label{eq:hesses} \nabla^2 h_K(u) + \nabla^2 h_{K^c}(-u) =   \frac{1}{|u|}\left(I - \frac{u}{|u|}\otimes \frac{u}{|u|}\right)\end{equation}  
(where on the right hand side we wrote the differential of $u\mapsto u/|u|$).  
    The eigenvectors of $\nabla^2h_K$ at $u$ are closely connected to the principal radii of curvature of $K$ at $x = \nabla h_K(u)$. Indeed, by \cite[Corollary 2.5.2]{schneider2013convex} the eigenvectors of $\nabla^2h_K$ at $u\in S^{n-1}$ are the vector $u$ itself with eigenvalue $0$ and the eigenvectors of the reverse Weingarten map (see \cite[Section 2.5]{schneider2013convex}) with corresponding eigenvalues $r_1 ,\dots , r_{n-1}$, which are the principal radii of curvature of $K$ at $\nabla h_K(u)$. For $u\in S^{n-1}$ equation \eqref{eq:hesses}
 is  $\nabla^2 h_K(u) + \nabla^2 h_{K^c}(-u) =   \left(I - u\otimes u\right)$ and so the two matrices have the same eigenvectors, and the ones orthogonal to $u$ have eigenvalues summing to $1$.   This completes the proof. 
\end{proof} 

\begin{rem}
One may apply this type of argument also when 
$\nabla h_K(u)$ is not a smooth point for $K$, 
since (see again \cite[Section 2.5]{schneider2013convex}) while the principal curvatures are functions on the boundary of $K$, the principal radii of curvature are considered as functions of the outer unit normal vector, in other words, as functions on the spherical image. 
We then have to compare these radii at a pair of points with normals $u,-u$, but the radii of curvature at the non-smooth point should be properly understood and depend not only on the boundary point but also on the normal considered. 
\end{rem}

\subsection{Faces of other dimensions}

It is natural to extend the definition of a ``face'' to the setting of ball-bodies. A face of a convex set can be defined as the intersection of the set with some supporting hyperplane.

\begin{definition}
Let $n\ge 2$, $K\in \S_n$.
For $y\in K^c$, the set $\,S_{K,y} = S(y,1)\cap K \subset \partial K$ is called an exposed $c$-face of $K$ (opposite to $y$).
\end{definition}

The proof of the following lemma is immediate from Lemma \ref{lem:normals map boundary to boundary}. 
\begin{lem}
Let $n\ge 2$, $K\in \S_n$. A set $S\subset \partial K$ is  an exposed $c$-face of $K$ if and only if for some $y\in \partial K^c$ we have $S = y-N_{K^c}(y)$. In particular, an exposed $c$-face is a closed spherically convex subset of a sphere. 
\end{lem}

As an example, two faces of the lens intersect at a sphere of lower dimension and smaller radius. But all the points in this sphere are extremal, and there is no lower dimensional face involved except $0$-dimensional. This is the general case as the following proposition implies. 

\begin{prop}
    Let $n\ge 2$ and $K\in \S_n$. Then letting ${\cal F}$ denote all the exposed $c$-faces of $K$ we have that $\partial K = \cup_{F\in \cal F} F$ and if $F_1,F_2\in \cal F$ with $F_1\neq F_2$ then any $x\in F_1\cap F_2$ is a $c$-extremal point for $K$. 
\end{prop}

\begin{proof}
Since any $y\in \partial K$ satisfies $y\in x-N_{K^c}(x) \in {\cal F}$ for $x=y-u$ and $u\in N_K(y)$, we get the inclusion      $\partial K = \cup_{F\in \cal F} F$. Assume $F_1\neq F_2\in {\cal F}$ and $y\in F_1\cap F_2$. In this case we have two different points $x_1,x_2\in K^c$ and $y\in x_i - N_{K^c}(x_i)$ for $i=1,2$ which means $x_i-y\in N_K(y)$ so that $y$ is not a smooth point and in particular (by Proposition \ref{prop:not-extremal-means-smooth}, say) is $c$-extremal for $K$. 
\end{proof}

As in classical convexity (with the example of a ``stadium'' in $\RR^2$),  an extreme point need not be an exposed face. 

\begin{exm}
    There exists $K\in \S_2$ and $x\in \partial K$ which is $c$-extremal, however   $\{x\}$ is not an exposed $c$-face of $K$. Indeed, in Example \ref{ex:Naztel-body} the smooth boundary points $(-(1-1/\sqrt{2}),0)$ and $(0,-(1-1/\sqrt{2}))$ are not $c$-exposed but are nevertheless $c$-extremal. 
\end{exm}

Nevertheless, an extremal point $x$ which has a full dimensional normal cone is always an exposed $c$-face, since for any $u$ in the interior of the cone $N_K(x)$, $y=x-u\in K^c$ is a smooth point which is not extremal and hence with $N_{K^c} (y) = \{-u\}$
and $S(y,1)\cap K = \{x\}$.

It turns out that Minkowski averaging cannot produce large $c$-exposed faces, if these were not $c$-exposed in the bodies one is averaging. More precisely we show the following lemma, which we used in \cite{FUTUREWORK} to show that if the Minkowski average of two sets in $\S_n$ is an $(n-1)$-lens, they must be translates of the same lens, which we needed in order to characterize isometries on $\S_n$.

\begin{lem}\label{lem:average-cap}
	Let $n\ge 2$ let $K_0, K_1\in\S_n$, $\lambda\in (0,1)$, and set $M =  (1-\lambda)K_0+\lambda K_1$. Assume that $B_2^n$ is a supporting ball of $M$, and that $A\subset \partial M$ is a spherically convex subset of the sphere $\partial B_2^n$. Then there exists $x_0,y_0 \in \RR^n$ with $A+x_0\subseteq \partial K_0$, $A+y_0 \subseteq \partial K_1$, $(1-\lambda)x_0+\lambda y_0  = 0$.
\end{lem}

\begin{proof}
    Since $A\subseteq \partial M \cap \partial B_2^n$, and $M \subseteq B_2^n$ by assumption, every $u \in A$ is also the unique point of $M$ that has $u$ itself as a normal. By the definition of Minkowski averages, every $u\in A$ can be written uniquely as $u =  (1-\lambda)x_0(u) + \lambda x_1(u)$ where $x_0(u)$
    is the unique point on $K_0$ where $u$ is the normal, and $x_1(u)$
    is the unique point on $K_1$ with $u$ as the normal.
	
	Since $K_0, K_1\in \S_n$, we know that $x_0(u)-u \in \partial K_0^c$ and $x_1(u) - u \in \partial K_1^c$. 
	Moreover, as $(1-\lambda)x_0(u) + \lambda x_1(u) =  u$, we get that $(1-\lambda)(x_0(u)-u) = -\lambda (x_1(u)-u)$.
	Using the fact that  $c$-duality commutes with averages (Theorem \ref{thm:minkowski-sum-c-is-linear}) we get $M^c =  (1-\lambda)K_0^c +\lambda  K_1^c$, and we see that for all $u,v\in A$ it holds that $ (1-\lambda )(x_0(u) - u) + \lambda (x_1(v)-v)\in M^c$, which can equivalently be written as $(1-\lambda )(x_0(u) - u -x_0 (v) + v) 
    \in M^c$. 
	
	Assume by way of contradiction that for some $u,v\in A$ we have that $x_0(u)-u\neq x_0(v)-v$. Then all four sums of $(1-\lambda)(x_0(u)-u)$ or $(1-\lambda)(x_0(v)-v)$ with $\lambda(x_1(u)-u) = (1-\lambda) (u-x_0(u))$ or $\lambda(x_1(v)-v)= (1-\lambda)(v-x_0(v))$ lie in $M^c$, namely 
	\begin{equation}\label{eq:onpartialMc}
	  0,  (1-\lambda) (x_0(u)-x_0(v)+v - u), (1-\lambda) (x_0(v)-x_0(u)+u - v) \in M^c .
	\end{equation}
    Since $B_2^n$ is a supporting ball of $M$, we have that $0 \in \partial M^c$, 
    which means there cannot be two points $z,-z$ both in $M^c$ (as $M^c$ is strictly convex). So, equation \eqref{eq:onpartialMc} is in fact a contradiction.

	We conclude that the points $x_0(u)$ are all of the form $x_0 +u$ for some fixed $x_0$. This means $(x_1(u)-u) = \frac{(1-\lambda)}{\lambda} x_0 =:y_0$. We get that $x_0 + A \subseteq \partial K_0$ and   $y_0 + A \subset \partial K_1$, where $(1-\lambda)x_0+\lambda y_0  = 0$, as claimed. 
\end{proof}

While it is not very simple to understand the $c$-extremal points in a Minkowski-average pf two bodies, when one of them is a ball this is possible, as the following proposition states.

\begin{prop}
Let $n\ge 2$, let $K \in \S_n$ 
and  $\lambda \in (0,1)$. Then setting $K_\lambda  = (1-\lambda) K +  \lambda B_2^n$ we have
\[ {\rm ext}_c(K_\lambda) =
\left((1-\lambda) {\rm ext}_c (K) +  \lambda B_2^n\right)
\cap \partial K_\lambda.
\]
\end{prop}

\begin{proof}
The body $K_\lambda$ is smooth. 
A point $x\in \partial K_\lambda$
is not $c$-extremal for $K_\lambda$ if and only if there is some open $1$-arc $A$ on $\partial K_\lambda$ (centered at $x - u$ for $u = n_{K_\lambda}(x)$) which includes $x$. 

Given some point $x\in \partial K_\lambda$ we may write it, uniquely, as $x = (1-\lambda) y + \lambda z$ for $y\in \partial K$ and $z\in S^{n-1}$, both with the same normal as $x$ (that is, $z = u$ and $y$ with $u\in N_K(y)$).

Assume $x$ is not $c$-extremal for $K_\lambda$, then  
Lemma \ref{lem:average-cap} applied to the $1$-arc $A$ passing through $x$ on $\partial K_\lambda$, implies that there exist $x_0,y_0 \in \RR^n$ such that $A + x_0 \subset \partial K $ and $A+  y_0 \subset S^{n-1}$ with $(1-\lambda)x_0 + \lambda y_0 = 0$. Since $A$ is centered at $x - u$ and $A+y_0$ is centered at $0$ we must have $x-u+y_0=0$, which means $y_0 = (1-\lambda)(u-y)$, and  
$x_0 = \lambda(y-u)$. 
Therefore, the arc $A+ x_0 \subset \partial K$  passes through $x + x_0 = 
y$ and $y$ is not $c$-extremal for $K$.

For the other side, assume that $y$ was not extremal for $K$. Then there is some open $1$-arc $A+ y-u\subset \partial K$ which includes $y$, centered at $y-u$. The translate of this arc (by $u-y$, which gives $A$) clearly belongs to $S^{n-1}$ and therefore $(1-\lambda)(A+y-u) + \lambda A\subset K_\lambda$, but the $1$-arc on the left hand side is simply $A + (1-\lambda)(y-u) = A + x-u$ so we get an open $1$-arc in $K_\lambda$ which includes $x$, which means $x$ is not $c$-extremal of $K_\lambda$, completing the proof. 

\end{proof}

\section{Steiner Symmetrizations and Shadow Systems}

Having established in Corollary \ref{cor:closed-under-Mink} that the class $\S_n$ is closed under Minkowski symmetrizations, it is natural to consider other forms of symmetrizations, the most well known and classical one being the Steiner symmetrization (see \cite{AGMBook}). We shall see that $\S_n$ is not closed under Steiner symmetrizations for $n\ge 3$, whereas in the plane it is. We will make use of ``moving shadows'' or ``linear parameter systems'' which are a generalization of Steiner symmetrization. 

\subsection{Linear parameter systems}\label{sec:lps}

For a set $A\subset \RR^n$, a vector $v$ and a function $\alpha:A\to \RR$ let 
\[ A_t =  \{x+ t \alpha(x)v: x\in A\},\,\, K_t =  {\rm conv} \{x+ t \alpha(x)v: x\in A\},\]
and 
\[ L_t =  \cconv \{x+ t \alpha(x)v: x\in A\}.\]

In classical convexity theory, the set $K_t$ is called a linear parameter system, and these were investigated in depth by Rogers and Shephard \cite{shephard1964shadow, rogers1958extremal}, where for example is was shown that $\vol(K_t)$ is a convex function in $t$ (as are the other quermassintegrals of $K_t$). The following proposition can be seen as an  analogue to the fact proved by Campi and Gronci in \cite{CampiGronchi2006} for the polar of a set, which is that  $1/\vol(K_t^\circ)$ is convex as a function of $t$.

\begin{prop}
Let $n\in {\mathbb N}$,  $A\subset \RR^n$,  $v\in \RR^n$ and  $\alpha:A\to \RR$. For $ t_\lambda  = (1-\lambda)t_0 + \lambda t_1$ it holds that 
	\[ L_{t_\lambda} \subseteq (1-\lambda)L_0 + \lambda L_1\qquad {\rm and} \qquad (1-\lambda)  L_{t_0} ^c + \lambda  L_{t_1} ^c \subseteq L_{t_\lambda}^c.\]
	In particular, $\vol(L_t^c)^{1/n}$ is  concave in $t$, as are the quermassintegrals $V_k(L_t^c)^{1/k}$.
\end{prop}

\begin{proof}
Since by definition $A_{t_\lambda} \subseteq (1-\lambda)A_0 + \lambda A_1$, it holds that 
    $A_{t_\lambda} \subseteq (1-\lambda)L_0 + \lambda L_1$ and the right hand side belongs to $\S_n$ by Theorem \ref{thm:minkowski-sum-c-is-linear}. Therefore we also have that  $L_{t_\lambda} \subseteq (1-\lambda)L_0 + \lambda L_1$. Applying the $c$-duality to both sides, and using that the $c$-duality commutes with averaging, by Theorem \ref{thm:minkowski-sum-c-is-linear}, we see that   $L_{t_\lambda}^c \supseteq (1-\lambda)L_0^c + \lambda L_1^c$, as claimed.   The corresponding concavity follows from Brunn-Minkowski inequality. 
\end{proof}

It is natural to ask  whether in analogy to the classical theorem of Rogers and Shephard, the function $\vol(L_t)$ is convex. This question is intimatel tied with the convexity of Steiner symmetrization (more on this below). As we shall demonstrate shortly (in Theorem \ref{thm:RSconvexityofvolR2} and Section \ref{sec:stein-not-in-class}), the answer is that this is true in dimension $n\le 2$ and false in higher dimensions. However, when considering a system of just two points, the answer is yes in any dimension, as the following proposition states.

\begin{prop}\label{prop-vol-of-carambula-is-convex}
Let $n\ge 2$ and fix $y_0 \in \RR^n$. The function $x\mapsto \vol \left(\cconv(x,y_0)\right)$ is convex. 
\end{prop}

 \begin{proof}
  Letting $d = \|x-y_0\|_2/2$, the volume of the $1$-lens $\cconv(x,y_0)$ is given by 
\begin{equation}\label{eq:vol-of-car} F_n(d) =\int_0^d
\left(\sqrt{1-t^2} - \sqrt{1-d^2}    \right)^{n-1} dt.  \end{equation}
We need only check that $F_n((1-\lambda)d_0 + \lambda d_1) \le (1-\lambda) F_n(d_0 ) + \lambda F_n(d_1)$. To this end differentiate
\begin{eqnarray*} F_n'(d) &=& \int_0^d (n-1) \left(\sqrt{1-t^2} - \sqrt{1-d^2}    \right)^{n-2} \cdot( d (1-d^2)^{-1/2} )dt + 0
	\\& =&(n-1)d (1-d^2)^{-1/2}  F_{n-1}(d).
\end{eqnarray*}
Therefore, using that
\[ 
\left(   d (1-d^2)^{-1/2} \right)'  = (1-d^2)^{-1/2}+ d^2 (1-d^2)^{-3/2} =  (1-d^2)^{-3/2}
\]  
\begin{eqnarray*} F_n''(d) &=& (n-1)  \left(   d (1-d^2)^{-1/2}  F_{n-1} \right)'\\& =& (n-1)(1-d^2)^{-3/2}F_{n-1}(d) + (n-1)d(1-d^2)^{-1/2}F_{n-1}' (d)\\& =& (n-1)(1-d^2)^{-3/2}F_{n-1} (d)+ (n-1)d(1-d^2)^{-1/2}(n-2) d(1-d^2)^{-1/2}F_{n-2}(d)\\& = & 
	(n-1)(1-d^2)^{-3/2}F_{n-1} (d)+ (n-1)(n-2)d^2(1-d^2)^{-1}F_{n-2}(d).
\end{eqnarray*}
Since all the expressions are positive, we get a positive second derivative, which means $F_n$ is convex, as required. 
 \end{proof}

From Proposition \ref{prop-vol-of-carambula-is-convex} we immediately get that for the case of a $c$-shadow system of two points, the volume is a convex function. More precisely,
\begin{cor}\label{cor-2-pt-chadow-system-has-convex-vol}
Let $n\ge 2$, $v\in S^{n-1}$, $x_0,y_0\in \R^n$, $\alpha,\beta\in\R$, and for each $t\in\R$ let 
\[ L_t =  \cconv \{x_0+ t \alpha v, y_0+ t \beta v\}.\]
Then the function $f(t):=\vol_n(L_t)$ is convex.
\end{cor}
\begin{proof}
We denote the diameter of the $1$-lens $L_t$ at time $t$ by 
\[d(t):=\|x(t)-y(t)\|_2,\]
where $x(t)=x_0 + t\alpha v$ and $y(t)=y_0 + t\beta v$. The function $F_n:\R\to\R$ (from Proposition \ref{prop-vol-of-carambula-is-convex}) is convex and increasing, and the function $d:\R\to\R$ is convex, thus the composition $f=F_n\circ d$ is convex.
\end{proof}

Next we show that in dimension $n=2$, the volume of the $c$-hull of a linear parameter system is indeed convex. 
To prove convexity of the volume in $\RR^2$, one can use induction on the number of points, combined with the following technical lemma regarding ``locally convex enlargement'' of a convex function.
\begin{lem}\label{lem:patching-convexity}
Let $g:\R\to (-\infty,\infty]$ be a convex function, and let $f:A\to (-\infty,\infty]$, where $A\subset \R$ is an open set, i.e. $A$ is a countable union of pairwise disjoint intervals $\{(a_n,b_n)\}_{n\in \N}$. Assume that for every such interval $(a_n,b_n)$, the restriction $f\big|_{(a_n,b_n)}$ is a convex function that agrees with $g$ at the endpoints of the interval i.e. 
$\lim_{x\to a_n^+}f(x)=g(a_n)$, and $\lim_{x\to b_n^-}f(x)=g(b_n)$. Then the function
\[
h(x)=
\begin{cases}	\max\{f(x),g(x)\}&  x\in A  \\
	        g(x)& x\notin A
\end{cases}
\]
is convex in $\R$.
\end{lem}

\begin{proof}
Deonte $A_n=\cup_{i=1}^n (a_n,b_n)$ and let 
\[
h_n(x)=
\begin{cases}	\max\{f(x),g(x)\}&  x\in A_n  \\
	        g(x)& x\notin A_n
\end{cases}
\]
Clearly $h$ is the point-wise limit of $(h_n)_{n=1}^\infty$, thus it suffices to show that $h_n$ is convex. Since convexity is a local property, and it holds both 
on the open set $A_n$ and on open subsets of its complement, we need only check it at $\partial A_n$. Since $\partial A_n  = \{a_n,b_n\}_{n\in \N}$ we may without loss of generality check convexity around the point $x = a_1$, say. In such a case, the right derivative for $h$ at $x$ exists and is equal to the right derivative of $\max(f,g)$ (which is a convex function on $(a_1, b_1)$) at $a_1$. For some $\varepsilon>0$ we have $(a_1-\varepsilon, a_1)\subset \mathbb{R}\setminus A_n$ and $(a_1, a_1+\varepsilon)\subset A_n$, thus $h_n=g$ on $(a_1-\varepsilon, a_1)$ and we get:
\begin{eqnarray*}
h_n'(a_1^-) &=& g'(a_1^-)\le g'(a_1^+) =
\lim_{\delta\to 0^+} \frac {g(a_1 + \delta) - g(a_1)}{\delta} \\
&\le &\lim_{\delta\to 0^+} \frac {h_n(a_1 + \delta) - g(a_1)}{\delta} =
\lim_{\delta\to 0^+} \frac {h_n(a_1 + \delta) - h_n(a_1)}{\delta} = h_n'(a_1^+),
\end{eqnarray*}
completing the proof.
\end{proof}

\begin{thm}\label{thm:RSconvexityofvolR2}
Let   $v\in S^{1}$,  and let $\{ x_i\}_{i=1}^m \subset \RR^2$ and $\{ \alpha_i\}_{i=1}^m \subset \RR$. For each $t\in \RR$ let 
\[ L_t =  \cconv \{x_i+ t \alpha_i v : i=1, \ldots m\}.\]
 Then the function $F(t) = \vol(L_t)$ is convex.    
\end{thm}
\begin{proof}
We prove by induction on the number of points $m$, where the base case $m=2$ was handled in Corollary \ref{cor-2-pt-chadow-system-has-convex-vol}.
We define the set $A\subset \R$ to be the set of all $t\in\R$ such that the set of extremal points of $L_t$ is $\{ x_i\}_{i=1}^m$. For every strict subset $I\subsetneq \{1,\dots,m\}$ we define $f_I:\R\to\R^+$ by $f_I(t)=\vol_n(\cconv\{x_i+t\alpha_iv\,:\, i\in I\})$, and $g=\sup_I \{f_I\}$. 
By the induction hypothesis, $f_I$ are convex on $\R$, and thus also $g$ is convex. 

By Theorem \ref{thm:cara-hard-neeD} we know that ${\rm ext}_c(L_t) \subseteq \{x_i + t\alpha_i v\}_{i=1}^m$. The set $A$, consisting of $t\in \RR$ for which there is equality in this inclusion, is open and may be empty. If it is empty then by the induction hypothesis we are done. Assume $A$ is nonempty and consider an interval $(a,b)\subset A$. 
Define the function $f$  on this interval by $\vol (L_t)$. Since on this interval ${\rm ext}_c(L_t) = \{x_i + t\alpha_i v\}_{i=1}^m$, the set $L_t$ consists of a polygon $\conv \{x_i+t\alpha_iv\}_{i=1}^m$
and $m$ halves of $1$-lenses between neighboring vertices. Each of these sets has volume which is convex in $t$, and therefore the union has volume which is convex in $t$.  In other words, on the interval $(a,b)$, the function $f$ is convex. We   thus satisfy the conditions of Lemma \ref{lem:patching-convexity} and conclude that $\vol (L_t)$ is a convex function of $t$ on $\RR$. 
 \end{proof}

Since $c$-polytopes are dense in ball-bodies, we can conclude convexity of the volume for general $c$-linear parameter systems, in dimension $n = 2$. The fact that  $\vol(L_t)$ may fail to be convex when working in dimension $n\ge 3$ will follow from Section \ref{sec:stein-not-in-class}, see Remark \ref{rem:vol-not-conv-in-R3}.

\begin{cor}\label{cor:linparsysR2}
Let $A\subset \RR^2$ be a bounded set,  let $v\in S^1$ and let  $\alpha : A\to \RR$. For each $t\in \RR$ let $ L_t =  \cconv \{x+ t \alpha(x)v: x\in A\}$. Then the function $F(t)  = \vol (L_t )$ is convex.  
\end{cor}

\begin{proof}
    Recall $A_t = \{ x+ t\alpha(x)v:x\in A\}$. First note that the function $R_t = \outrad(A_t)$ is convex. Indeed, 
    for $t_{\lambda} = (1-\lambda)t_0 +\lambda t_1$, the inclusion
    $A_{t_\lambda} \subseteq (1-\lambda)A_{t_0} + \lambda A_{t_1}$ clearly holds by definition, and since the out-radius of a set is convex with respect to Minkowski addition, we see that $R_t$ is convex. Therefore, there is an interval $[t_{\min}, t_{\max}]$ where $(A_t)^{cc} \neq \RR^n$, and $L_{t_{\min}}, L_{t_{\max}}$ are Euclidean balls. (There is only one case where $t_{\min} = -\infty$ and $t_{\max} = +\infty$, namely when $\alpha$ is a constant function, as in all other cases there are two points moving in different velocities, meaning for large enough $|t|$, the out-radius of $A_t$ is more than $1$.)
We may thus restrict to times $t\in [t_{\min}, t_{\max}]$, as out of this interval $F(t) = +\infty$. 

Let $t_0, t_1\in [t_{\min}, t_{\max}]$ and $\lambda \in (0,1)$, and let 
$t_{\lambda} = (1-\lambda)t_0 +\lambda t_1$. Let $D_0, D_1, D_\lambda$ be countable subsets of $A$, such that $(D_{i})_t \subseteq A_{t_i}$ is dense for $i = 0,1,\lambda$, and consider $D_0\cup D_1\cup D_\lambda \equiv D \subseteq A$. Let $(A_m)_{m=1}^\infty$ be an increasing sequence, such that $A_m\subseteq D$ consists of $m$ points, and $\cup_{m=1}^\infty A_m = D$. 
By construction, $(A_m)_t \to A_t$ for $t = t_0, t_1, t_\lambda$ as $m\to \infty$, where this limit is in the Hausdorff sense. 

By Corollary \ref{cor:c-hull-is-continuous}, and using that $\outrad((A_m)_t)\le \outrad(A_t)\le 1$, we see that $ \cconv((A_m)_t) \to \cconv (A_t)$ for $t = t_0, t_1, t_\lambda$. We use Theorem \ref{thm:RSconvexityofvolR2} which implies
\[ \vol(\cconv (A_m)_{t_\lambda}) \le (1-\lambda)\vol(\cconv (A_m)_{t_0}) + \lambda \vol(\cconv (A_m)_{t_1}).
\]
Taking the limit $m\to\infty$, and using continuity of volume on $\S_2$ with respect to the Hausdorff distance, we get the desired inequality.
\end{proof}

\subsection{Steiner symmetrization}\label{sec:Steiner}

Recall the definition of the Steiner symmetrization $S_u(K)$ of a convex body $K$ with respect to the hyperplane $u^\perp$. Denoting $K\cap\left(x+\R u\right)=[x+a(x)u,x+b(x)u]$ for any $x$ in the projection $P_{u^\perp}(K)$ of $K$ to $u^\perp$, the Steiner symmetral of $K$ with respect to $u$ is defined to be
\begin{equation*}
S_u(K)=\left\{(x,y)\in u^\perp \times \RR\,:\, x\in P_{u^\perp}(K),\,|y|\le\frac{|b(x)-a(x)|}{2} \right\}.
\end{equation*}

It is well known that the Steiner symmetral 
 $S_u(K)$ can be realized as the time $t=1$ set $K_1(K)$ (In the notations of Section \ref{sec:lps}) for a linear parameter system
  in direction $u$ assigning velocity $-(a(x)+b(x))/2$ to the points $(x,y)\in K$ where $y\in [a(x), b(x)]$. Moreover, the time $2$ set for this system will be $K_2(K) = R_u(K)$, the reflection of $K$ with respect to $u^\perp$. 

In particular, since  $\vol(R_u(K)) = \vol(K)$, and since $\vol(S_u(K)) = \vol(K)$ by Fubini's theorem, we see (using the result of Rogers and Shephard about convexity of $\vol(K_t(K))$) that in this linear parameter system the volume of the sets $K_t(K)$, which are all convex, is constant for $t\in [0,2]$.

\begin{rem}\label{rem:vol-not-conv-in-R3}
  We will see in Section \ref{sec:steinerintheplane} and Section \ref{sec:stein-not-in-class} that Steiner symmetrization preserves the class $\S_n$ only when $n\le 2$. In particular, this implies that $\vol(L_t)$ cannot always be convex, since in the case where $S_u(K)\not\in \S_n$, \[\vol(L_1(K)) = \vol(\cconv(S_u(K)))> \vol(S_u(K)) = \vol(L_0(K)) = \vol(L_2(K)).\]  
\end{rem}

Nevertheless, one may combine Steiner symmetrization with  $c$-hulls to prove useful volume inequalities.

\begin{thm}\label{thm:steiner-increases-dual-volume}
	Let $K\subset \RR^n$ be convex and let $u\in S^{n-1}$. Then 
	\[ S_u (K^c) \subseteq \cconv(S_u (K^c))\subseteq  (S_u K )^c.\]
	In particular, $\vol(K) \vol(K^c) \le 	\vol(S_u K) \vol((S_u K)^c)$. 
\end{thm}

\begin{proof}
Recall the notion of Minkowski symmetrization (see Corollary \ref{cor:closed-under-Mink} and the definition preceding it), defined for a set $K\subset \RR^n$ and $u\in S^{n-1}$ by $M_u(K) = \frac12 (K + R_u(K))$. It is well known and easy to check that $S_u (K) \subseteq M_u (K)$ for any convex $K$. Using Corollary \ref{cor:minkowski-sum-c-is-sub-linear} we see that 
\[
M_u (K^c) = 
\frac12 (K^c + R_u(K^c)) = 
\frac12 (K^c + (R_u(K))^c) \subseteq 
(\frac12 (K + R_u(K)))^c =
(M_u (K))^c
\]
(here we use that $(U A )^c = U (A^c)$ for any isometry $U$). 
Joining these two facts, and the fact that $c$-duality reverses inclusion we get 
\[
S_u (K^c) \subseteq
M_u (K^c) \subseteq
(M_u (K))^c \subseteq
(S_u(K))^c. 
\]
Since the right hand side belongs to $\S_n$, inclusion remains also after taking a $c$-hull, which completes the proof. 
\end{proof}

We can once again deduce a Santal\'{o}-type inequality, although it is not stronger than the previous ones we have obtained. 

\begin{cor}\label{cor:santalofixedvol}
Let $A\subset \RR^n$ with out-radius at most $1$, and let $B(0,r)$ be the Euclidean unit ball with the same volume as $\conv (A)$. Then  
\begin{equation}\label{eq:santalo-fixed-vol}
 \vol(A)\vol(A^c) \le \vol(B(0,r)) \vol(B(0,1-r)) = (r(1-r))^n \kappa_n^2.
\end{equation} 	 
\end{cor}

\begin{proof}
One may find a sequence of Steiner symmetrizations of $\conv(A)$ which converges to a ball of the same volume (see e.g. \cite[Theorem 1.1.16]{AGMBook}). Using Theorem \ref{thm:steiner-increases-dual-volume}, the volume product is increasing along the sequence, which completes the proof.
\end{proof}

\begin{rem}
Before continuing with Steiner symmetrization, we mention yet another symmetrization that was use in the literature, also for the class $\S_n$. 
In \cite{Bezdek2018} 
Bezdek proves a  fact similar to Corollary \ref{cor:santalofixedvol} using a symmetrization called ``two-point symmetrization''. To describe it, denote for an affine hyperplane $H$ define the operation of reflection with respect to $H$ by $R_H$.  The two-point symmetral of $K$ with respect to $H$ is  
\[ \tau_H (K) = (K\cap \sigma_H(K)) \cup ((K\cup \sigma_H(K))\cap H^+). \]
It is easy to check that $K$ and $\tau_H(K)$ have the same volume (but convexity of course need not be preserved). 
\begin{thm}[Bezdek]
	If $K\subset \RR^n, n>1$  then 
	\[ {\rm conv}_c \tau_H(K^c) \subset (\tau_H (K))^c \] 
	In particular, among all compact sets of a given volume, the ball has the largest (in volume) $c$-dual. 
\end{thm}
 Bezdek uses this theorem to prove a special case of the Knesser-Poulsen conjecture (See  \cite{bezdek2008kneser}, as well as our discussion in Section \ref{sec:KnesPoul}).
\end{rem}

\subsection{Steiner Symmetrization in the plane}\label{sec:steinerintheplane}

Consider the linear parameter systems $L_t(K)$ and $K_t(K)$ associated with the Steiner symmetrization $S_u(K)$ of $K$, as explained at the beginning of Section \ref{sec:Steiner}. 
Note that by Corollary \ref{cor:linparsysR2}, in $\RR^2$ the function $\vol(L_t(K))$ is convex, and on the other hand $\vol(L_t(K)) \ge \vol(K_t(K)) = \vol(K)$. This implies that in $\RR^2$ the function $\vol(L_t(K))$ must be constant on the interval $[0,2]$, and in particular that the bodies $L_1(K)$ and $K_1(K) = S_u(K)$ are the same, so that $S_u(K)\in \S_2$. This proves the following theorem.

\begin{thm}\label{thm:InR2Steiner-preserves-class}
	Let $K\in \S_2$ and let $u\in S^{1}$.
	The Steiner symmetral $S_u(K)$ of $K$ in direction $u$   also belongs to $\S_2$. 	
\end{thm}

It is instructive to see a direct proof for Theorem \ref{thm:InR2Steiner-preserves-class}, and we provide one under the assumption that the body $K$ consists of the area between two graphs of twice differentiable functions. The general case follows by approximation, using Theorem \ref{thm:dense-are-the-smooth}. 
(Nevertheless, the   above explanation constitutes a full alternative proof.) 

\begin{proof}[Second proof of Theorem \ref{thm:InR2Steiner-preserves-class}]
	Sets in $\S_2\setminus \{\emptyset, \RR^2\}$ which are not points  are characterized as closed convex sets for which the  generalized curvature 
	at all  points  is at least $1$, as follows from Theorem \ref{thm:char-curv}. Assume $u = e_2$ and that $K$ is smooth, with boundary given by the graphs of two twice continuously differentiable concave functions $f$ and $-g$ with some support set $[a,b]$. Since there are no segments on the boundary of a set in $\S_2$, we have that $f(a)=-g(a)$ and $f(b) = -g(b)$. 
 
 For $x\in (a,b)$ 
 the curvatures at the points $(x,f(x))$ and $(x,-g(x))$ are given by $\kappa_f(x) = \frac{f''(x)}{(1+(f'(x))^2)^{3/2}}$ and $\kappa_g(x) = \frac{g''(x)}{(1+(g'(x))^2)^{3/2}}$.  The Steiner symmetral of $K$ has boundary given by the functions $h,-h$ on $[a,b]$ with $2h(x) = f(x) + g(x)$, and the curvature at points $(x,\pm h(x))$ satisfies
	\begin{eqnarray*}\label{eq:Steiner-is-in-S_2}
		\kappa_h(x) &=& \frac{h''(x)}{(1+(h'(x))^2)^{3/2}}
		=\frac{\frac12(f''(x)+g''(x))}{\left(1+\left(\frac{f'(x)+g'(x)}{2}\right)^2\right)^{3/2}}\ge \\
		&\ge& \frac{\frac12(f''(x)+g''(x))}
		{
			\frac12\left(1+f'(x)^2\right)^{3/2} +
			\frac12\left(1+g'(x)^2\right)^{3/2}
		}\ge
		\frac{\frac12(f''(x)+g''(x))}
		{\frac12f''(x)+\frac12g''(x)} = 1.
	\end{eqnarray*}
	The first inequality holds since the function $t\mapsto (1+t^2)^{3/2}$ is (strictly) convex, and the second inequality holds since $\kappa_f,\kappa_g\ge1$. 
    
To complete the proof in the case of a smooth $K$, we need to also consider the points $x=a$ and $x=b$. Start with the latter. By smoothness, the normal to $K$ at $(b,f(b))$ is in direction $e_1$ and   by assumption, the ball $B((b-1, f(b)), 1)$ contains $K$. Therefore $S_u K \subset S_u (B((b-1, f(b)), 1)) = B((b-1, 0), 1)$. This is a $1$-ball supporting $S_uK$ at $(b,h(b)) = (b,0)$. The same argument works for $x = a$ of course. Since Steiner symmetrization is continuous on bodies with no-empty interior (see e.g. \cite[Proposition A.5.1.]{AGMBook}), and since (by Theorem \ref{thm:dense-are-the-smooth}) smooth bodies are dense in $\S_n$, the proof is complete. 
\end{proof}

\subsection{A counterexample in dimension $3$}\label{sec:stein-not-in-class}

The previous section makes it natural to believe the Steiner symmetrization will preserve the class $\S_n$ is any dimension, since it respects inclusion by balls. However, as we shall see in this section, already in dimension $3$ the symmetral of a set in $\S_3$ might have some sectional curvature exceeding $1$. 
As a first step to establish 
whether $S_u(K)\in\S_n$ for any $K\in \S_n$, one easily sees that this would be equivalent to showing that $S_u(L)$ is in $\S_n$ for every lens $L$. Indeed,  fixing a fiber in direction $u$, namely $(x+ \RR u)\cap K$, we can find a lens $L$ which supports $K$ in both endpoints of the fiber (by simply intersecting the two supporting unit balls), and as $S_u$ preserves inclusion, $S_u(K) \subseteq S_u(L)$, and they both have the same intersection with $x + \RR u$. If $S_u(L)$ belongs to $\S_n$, this  gives the curvature conditions in these endpoints also for $S_u(K)$.  

From the opposite perspective, this means that if there exists $K\in \S_n$ with $S_u(K)\not\in\S_n$, we can already find a counterexample using a lens. In this section we do precisely this, and given a direction $u\in S^3$ we find a lens in $L\subset \RR^3$ whose Steiner symmetral $S_u(L)$ is not in the class $\S_3$.

For simplicity of the computation, we set $u=e_3$, and $L = B(c_0,1) \cap B(-c_0,1)$ where $c_0=(x_0,y_0,z_0)$. We write $B(c_0,1)=\{ (x,y,z):
-f_d(x,y) \le z \le f_u(x,y)
\}$ and likewise $B(-c_0,1)=\{ (x,y,z):
-g_d(x,y) \le z \le g_u(x,y)
\}$, where the functions $f_d,f_u:B\left((x_0,y_0),1 \right)\to\R$ and $g_d,g_u:B\left(-(x_0,y_0),1 \right)\to\R$ are given by
\begin{eqnarray*}
    f_u(x,y) &=& z_0 + \sqrt{1 - (x-x_0)^2 - (y-y_0)^2}\\    
    -f_d(x,y) &=& z_0 - \sqrt{1 - (x-x_0)^2 - (y-y_0)^2}\\
    g_u(x,y) &=& -z_0 + \sqrt{1 - (x+x_0)^2 - (y+y_0)^2}\\
    -g_d(x,y) &=& -z_0 - \sqrt{1 - (x+x_0)^2 - (y+y_0)^2}.
\end{eqnarray*} 
This means $L = \{ (x,y,z): \max(-f_d(x,y), -g_d(x,y))
\le z \le
\min(f_u(x,y), g_u(x,y))
\}$.
We shall make sure to pick $c_0, (x,y)$ such that $f_u(x,y)\le g_u(x,y) $ and $-g_d(x,y)\ge -f_d(x,y)$ so that when considering the fiber $\left((x,y,0)+\RR e_3 \right)\cap L$, we will be dealing with the interval $[(x,y,-g_d(x,y)), (x,y,f_u(x,y))]$. 

Our choice of parameters is
\[ x_0 = -0.2807,\,\,y_0 =0.2457 ,\,\,z-0 = 0.4,\,\, x = 0.4142,\,\, y = 0.7268. \,\, \]
for which we see 
\[ f_u(x,y) 
\simeq 0.134 \le   g_u(x,y)
\simeq 0.59 \]  
\[ -f_d(x,y) 
\simeq -0.934   \le 
  -g_d(x,y)
  \simeq 0.209.  \]

We can compute easily 
 \begin{eqnarray*}
  \nabla   f_u(x,y) &=& \frac{-((x-x_0), (y-y_0) )}{
  \sqrt{1 - (x-x_0)^2 - (y-y_0)^2}}\simeq  -(1.2995,0.8996)\\
   \nabla  g_d(x,y) &=& \frac{-((x+x_0), (y+y_0) )}{
  \sqrt{1 - (x+x_0)^2 - (y+y_0)^2}}\simeq - (0.6997,5.0948)
\end{eqnarray*}

The fiber replacing $[(x,y,-g_d(x,y)), (x,y,f_u(x,y))]$ in the Steiner symmetral will be $[(x,y,-h(x,y)), (x,y,h(x,y))]$ where $h =\frac{ f_u + g_d}{2}$. By the formula for sectional curvature (see \cite{schneider2013convex}) curvature, we see that the sectional curvature of $\varphi$ (which can be either $g_d$, $f_u$ or $h$) in direction $e_1$ is given by 
\[ 
\kappa_\varphi(x,e_1)  = \frac{ (\nabla^2 \varphi(x))_{1,1}}{\sqrt{1+\|\nabla \varphi(x)\|_2^2}}\left(\frac{1}{1 +\iprod{e_1}{\nabla \varphi(x)}^2}\right). 
\]
Since both $g_d$ and $f_u$ correspond to surfaces of a translated sphere, this expression for both of them will equal to $1$. At the same time, since $h = (f_u+g_d)/2$, also $\nabla h  = (\nabla f_u+\nabla g_d)/2$, as well as $\nabla^2 h  = (\nabla^2 f_u+\nabla^2 g_d)/2$. 

As a function of two variables, 
$\psi(s,t)= \sqrt{1+t^2+s^2}(1+t^2)$ is not convex, which is why when comparing (here $w = (x,y)\in \RR^2$)
\begin{eqnarray*}\label{eq:Steiner-wrong1}
	\kappa_h(w,e_1) &=&  \frac{ (\nabla^2 h(w))_{1,1}}{\sqrt{1+\|\nabla h(w)\|_2^2}}\left(\frac{1}{1 +\iprod{e_1}{\nabla h(w)}^2}\right) \\
	&=&
	\frac{  \frac12 \left((\nabla^2 f_u(w))_{1,1}+  (\nabla^2 g_d(w))_{1,1}\right) }{\sqrt{1+\|\frac{\nabla f_u(w) + \nabla g_d(w)}{2} \|_2^2}\left(1 + \iprod{e_1}{\frac{\nabla f_u(w) + \nabla g_d(w)}{2}}^2\right)}, 
\end{eqnarray*}
and 
\begin{eqnarray*}\label{eq:Steiner-wrong2} 
	1 & = &\frac{ \frac12 \left(\nabla^2 f_u(w))_{1,1} +  (\nabla^2 g_d(w))_{1,1}\right) }
	{ \frac12 \left( (\nabla^2 f_u(w))_{1,1}+ (\nabla^2 g_d(w))_{1,1} \right)} \\ 
    &=&  \frac{	\frac12\left(  (\nabla^2 f_u(w))_{1,1}+ (\nabla^2 g_d(w))_{1,1})\right) }
	{\frac{1}{2} \left(\sqrt{1+\|\nabla f_u(w)\|_2^2}(1 +\iprod{e_1}{\nabla f_u(w)}^2)+ \sqrt{1+\|\nabla g_d(w)\|_2^2}(1 +\iprod{e_1}{\nabla g_d(w)}^2)
		)\right)}, 
\end{eqnarray*}
we can make sure that $\kappa_h(w,e_1)<1$ by forcing 
an  inequality stating that the denominator of the former is in fact larger than the denominator of the latter. In other words, we need to make sure that the parameters were chosen so that 
\[ \psi (\frac12\left(\nabla f_u (w) + \nabla g_d(w) \right))> \frac12 \left(\psi(\nabla f_u(w))+\psi(\nabla g_d(w))\right).\]  

Since these expressions are explicit in our example, let us check
\[ \psi (\frac12\left(\nabla f_u (w) + \nabla g_d(w) \right)) 
= \psi(0.9996,2.9972) = 6.313
\]
and 
\[ \frac12 \left(\psi((1.2995,0.8996))+\psi((0.6997,5.0948))\right) \simeq  
\frac12 ( 4.251  + 7.658 ) = 5.9545.  \]  
So, indeed, this example works as desired.

\section{Application and open problems}

The class $\S_n$ is connected with an array of interesting open problems and conjectures. Some of these we have already touched upon in the text, such as Borisenko's conjecture as well as Mahler-type problems regarding $c$-duality, and various maximization and minimization problems of parameters of convex bodies within this class. In this section we aim to touch upon several other key directions in which $\S_n$ and the $c$-duality play a key role. These serve mainly as motivation to further study this class and the associated structures. 

\subsection{Measure transport}

Transportation of measure is a very active research area in close proximity to convexity theory, and was in fact part of our original motivation to study the class $\S_n$. We recall the basic setting so as to illustrate this. 
 
In the theory of measure transport, one is given a symmetric cost function on $X\times X$ for a measure space $X$, and two probability measures $\mu,\nu\in {\cal P}(X)$. The underlying question, going back to Monge \cite{Monge1781}, asks whether there exists a transport map, namely a function $T:X\to X$ satisfying $\mu(T^{-1}A)=\nu(A)$ for all measurable $A\subseteq X$, which is also optimal with respect to some cost function $c:X\times X\to \RR$, namely minimizing, over all such $T$, the cost $\int c(x,Tx)d\nu(x)$. 
The relaxation due to Kantorovich \cite{kantorovich1942transfer,kantorovich1948monge} has to do with transport {plans}, namely probability measures $\gamma\in { \cal P}(X\times X)$ with marginals $\mu$ and $\nu$, which we denote $\gamma\in\Pi(\mu,\nu)$. A transport plan always exists (e.g. $\mu\times \nu$) and the main questions regard optimal plans,
where the cost of a plan $\gamma$ is naturally given by
\[
C(\gamma) = \int c(x,y)d\gamma,
\]
and the cost of transporting $\mu$ to $\nu$ is defined by 
$ C(\mu, \nu) = \inf\left\{ C(\gamma) : \gamma \in \Pi(\mu, \nu)\right\}$.

For an overview of transportation theory see \cite{Villani2009}. For the  quadratic cost $c(x,y)=\|x-y\|_2^2$ in $\R^n$, it is well known that transport maps exist and have a special form, given in the well known and much used Brenier-McCann theorem \cite{brenier1991polar, mccann1995monotone}, see e.g. \cite[Section 1.3.2]{AGMBook}. 
 
Non-traditional cost function are costs which can attain infinite value, or, in other words, where certain pairings $x\mapsto y$ are not allowed. 
Such costs include some costs which are by now well used and studied, for example the polar cost on $\RR^n$ (see \cite{artstein2017differential}) and costs coming from geometric refractors on the sphere of the form $\log(\kappa \cos d(x,y) -1)$ see \cite{GutierrezHuang2009} and more generally the book \cite{Gutierrez2023}. Oliker \cite{Oliker2007} 
 showed that Alexandrov's problem of prescribing integral Gauss curvature of closed convex surfaces can be seen as a transport problem for the non-traditional cost $c(x,y)= \log \cos d(x,y)$.  

For  non-traditional  costs, even the existence of a finite-cost plan is not guaranteed, and necessary and sufficient conditions for a pair of measures $(\mu, \nu)$ to satisfy $C(\mu, nu)<\infty$ are usually developed separately for each cost depending on its structure. In \cite{artstein2023zoo} natural necessary conditions, and slightly stronger sufficient conditions, for the existence of a finite cost plan in the case of a non-traditional cost were given. The (easily verifiable) necessary condition is
\begin{equation}\label{eq:condition-exists-transport}
    \mu(A)+\nu(A^c)\le 1\quad \forall A\subset \R^n,
\end{equation}
where here $A^c=\{y\in\R^n:c(x,y)=\infty\}$ is a ``duality'' associated with the cost function $c$. Thus duality-type mappings are intimately connected with measure transportation with respect to non-traditional costs. In the special case where the cost function on $\RR^n$ is given by 
\begin{equation}\label{eq:costFunc}
c(x,y) = F(\|x-y\|_2) \quad {\rm with}\quad F|_{B(0,1)} = +\infty, \, F|_{\RR^n\setminus B(0,1)} < \infty, \end{equation}
this duality is precisely the $c$-duality of this note, and it is easy to see that   condition \eqref{eq:condition-exists-transport} is equivalent to the same condition restricted to $A$ in the class $\S_n$. 

In fact, the connection between optimal transport and $c$-duality runs much deeper, via Brenier-McCann type theorems or, more generally, the 
  Kantorovich Duality Theorem  \cite{kantorovich1942transfer,kantorovich1948monge}. In the case of the quadratic cost and some of its close relatives, this leads to very central geometric and functional inequalities such as Brunn-Minkowski, Pr\'ekopa-Leindler, and Brascamp-Lieb type inequalities (see \cite{AGMBook, AGA2}), as well as  concentration inequalities. Finding functional inequalities pertaining to costs of the form \eqref{eq:costFunc} will be pursued in future works. It should be emphasized, however, that specifying to the class $\S_n$, still allows for picking various functions $F:(1, \infty)\to \RR$ in \eqref{eq:costFunc}, which  affects the structure of the optimal plans (when it exists).

\subsection{Kneser-Poulsen type inequalities}\label{sec:KnesPoul}

In \cite{Gromov1987}, Gromov proves the following conjecture for $N\le n+1$ (and attributes it to Archimedes).
\begin{conj}
Let $n,N\in \N$. If $\left(x_i\right)_{i=1}^N,\left(y_i\right)_{i=1}^N\subset \R^n$ satisfy $\|x_i-x_j\|_2\le \|y_i-y_j\|_2$, then
\begin{equation}\label{eq:Knes-Poul-radii=1}
\vol \left( \cap_{i=1}^N B(y_i, 1)  \right)\le
\vol \left( \cap_{i=1}^N B(x_i, 1)  \right)
\end{equation}
\end{conj}
The case of $N>n+1$ is open in general, and is part of a family of similar conjectures   posed independently by Poulsen \cite{poulsen1954problem} and Kneser \cite{kneser1955minkowski}.

\begin{conj}\label{Kneser-P-intersection}
If $\left(x_i\right)_{i=1}^N,\left(y_i\right)_{i=1}^N\subset \RR^n$ satisfy 
$\|x_i - x_j\|_2\le \|y_i-y_j\|_2$, 
and $(r_i)_{i=1}^N \subset \RR^+$ then
\begin{align}
   \vol\left(
\cap_{i=1}^N B(x_i,r_i)
\right)
&\ge
\vol\left(
\cap_{i=1}^N B(y_i,r_i)
\right), \qquad {\rm \it and} \label{eq:kneserpolusenfirst} \\ 
   \vol\left(
\cup_{i=1}^N B(x_i,r_i)
\right)
&\le 
\vol\left(
\cup_{i=1}^N B(y_i,r_i)
\right). \nonumber        
\end{align}
\end{conj}
Conjecture \ref{Kneser-P-intersection} has been verified in various particular cases, for details see e.g.~\cite{Bezdek-Connelly,Klee-Wagon,Four-problems}. Recently Aishwarya and Li \cite{Gautam} gave  an information-theoretic counterpart  to the Kneser-Poulsen conjecture, with extensive use of measure transport techniques. 
In the notations of this paper, inequality \eqref{eq:Knes-Poul-radii=1} can be stated as 
\begin{equation}\label{eq:KP-with-duakity-notation}
\vol((\left\{y_i\right\}_{i=1}^N)^c) \le \vol((\left\{x_i\right\}_{i=1}^N)^c).
\end{equation}
Letting $K = (\left\{x_i\right\}_{i=1}^N)^{cc}, T = (\left\{y_i\right\}_{i=1}^N)^{cc}\in \S_n$, the inclusion \eqref{eq:KP-with-duakity-notation} can be further written as $\vol(T^c) \le \vol (K^c)$. In other words, the Kneser-Poulsen problem asks about the $c$-dual volumes of two ``$c$-polytopes'', with some contractive relation between their ``vertices''. This point of view gives some new insights, for example the following
\begin{fact}
Let $n,M\in \N$ and let $\left(x_i\right)_{i=1}^N, \left( y_j\right)_{j=1}^N\subset \R^n$ such that
$x_i\in T = (\left\{y_j\right\}_{j=1}^N)^{cc}$ for all $i\in\{1,\ldots,N\}$. Then we have 
$
\vol \left( \cap_{y\in T(K)} B(y, 1)  \right)\le
\vol \left( \cap_{x\in K} B(x, 1)  \right)$.  (This is since $T^c \subseteq K^c = (\left\{x_i\right\}_{i=1}^N)^{c}$.)
\end{fact}

This perspective inspires formulating the following variant of Conjecture \ref{Kneser-P-intersection}:
\begin{conj}\label{conj:KPweaker}
 Let $n\ge 2$ and let $T:\R^n\to \R^n$ be a contraction, and let $K\in\S_n$. Then
$
\vol((TK)^c) \le \vol(K^c).$   
\end{conj} 
Note that without the assumption $K\in \S_n$, Conjecture 
\ref{conj:KPweaker} is in fact a reformulation of the original conjecture for intersections of $1$-balls, since one may extend a contraction on the finite number of points to a contraction on $\RR^n$. However, as stated it is weaker.

\subsection{Bodies of Constant Width}\label{sec:constantwidth}

Bodies of constant width are studied extensively, see e.g.~the  survey \cite{MartiniMontejanoOliveros2019} and references therein. As explained in Proposition \ref{prop:KminusK}, these are precisely the fixed points of the $c$-duality on $\S_n$, namely bodies for which $K^c=K$. Moreover, for any $K\in \S_n$ the set $\frac12 (K+K^c)$ is a set of constant width $1$. 
In the special case where $\diam(K) \le 1$, meaning $K\subseteq K^c$, the set $\frac12 (K+K^c)$ can be seen as an explicit Euclidean diametric completion of $K$. In fact,  for any non-empty $K\subset \RR^n$ with $\diam(K) \le 1$, it holds that $K\subseteq K^c$ and so $K\subseteq \frac12 (K^{cc}+K^c)$ where the latter is of constant width $1$, once again a Euclidean  diametric completion. 

Let us examine the opposite direction. If a convex body $K$ satisfies $w_K(u)\ge 1$ for all $u\in S^{n-1}$, then 
$K^c\subseteq K$.
Indeed, since $K$ and $K^c$ are two convex bodies which intersect, it suffices to show that the interior of $K^c$ does not intersect $\partial K$. Let $x\in \partial K$. If $u$ is a normal to $K$ at $x$, then since $w_K(u)\ge 1$, there is some $y\in K$ with $\|x-y\|_2\ge 1$. Therefore $B(y,1)$ can include $x$, if at all, only on its boundary. So we see indeed that $x\not\in \partial K^c$ and conclude $K^c\subseteq K$. 
As $K^c\subseteq K$ we also have $\frac12(K+K^c)\subseteq K$, however the left hand side need not be a set of constant width $1$, since $K$ was not assumed to be in $\S_n$. 

In fact, it is quite easy to construct a convex $K\subset \RR^n$ with $w_K(u) \ge 1$ for every $u$, which does not include a body of constant width $1$ (in particular, such ``diametric shaving'' does not exist). One such example is $K=B_2^n\cap (\RR^+)^n$, which can easily be seen to have width at least $1$ in every direction. Indeed, $h_K(u) = \|u_+\|_2$ where $u_+$ is the vector with $i^{th}$ coordinate $\max(u_i, 0)$, and so 
\[ w_K(u) = h_K(u) + h_K(-u) = \|u_+\|_2+ \|u_{-}\|_2\ge \sqrt{\|u_+\|_2^2+ \|u_{-}\|_2^2} = 1.
\]
The fact that $K$ does not contain a body of constant width $1$ is also easy to check. Indeed, let $C\subset K$ be a body of constant width. Then $e_i\in C$ for all $1\le i\le n$, since $w_K(e_i)=w_C(e_i) = 1$, and $e_i$ is the unique supporting point of $K$ in direction $e_i$. But alas, $e_1, e_2 \in C$ implies $w_u(C) \ge \sqrt{2}$ for $u = \frac{e_1-e_2}{\|e_1-e_2\|_2}$, thus $K$ does not contain any body of constant width $1$. 

It is worthwhile to mention (and is directly related to the topic of the next subsection) that $\vol(K) = \vol (\frac12 B_2^n)$, and that Nazarov \cite{Nazarov-blog} showed that the convex hull of $(1-\varepsilon) K$ and $-\varepsilon K$, 
is a convex body with width at least $1$ in every direction, which, for a suitable chosen $\eps$, has volume exponentially (in $n$) smaller than $\vol (\frac12 B_2^n)$.

 \subsubsection{The Blaschke-Lebesgue problem}

Perhaps the most famous basic question  regarding bodies of constant width, which remains open to this day,    is the Blaschke-Lebesgue problem of finding, in $\RR^n$, the bodies of least volume among bodies of constant width $1$ (by Urysohn's inequality, $\frac12 B_2^n$ has maximal volume). In dimension $n=2$ the minimizer is known to be the Reuleaux triangle (see \cite{lebesgue1914isoperimetres,blaschke1915breite}, we define it in Lemma \ref{lem:basinR}  below) and the problem is open for $n\ge 3$, where the best known lower bound is of the form $(\sqrt{3}-1)^n \vol(\frac12 B_2^n)$ due to Schramm \cite{schramm1988constantwidth}. For $n=3$ it is conjectured that the Meissner bodies are the unique minimizers (see e.g. \cite[Section 8.3.3 and Section 14.2]{MartiniMontejanoOliveros2019}). 
In \cite{schramm1988constantwidth}, Schramm asked whether there exist bodies of constant width $1$ with volume exponentially smaller than $\vol(\frac12 B_2^n)$, and such an example was recently found in \cite{arman2024small}. Finding optimal asymptotic behavior of the volume of the minimzer(s) remains an open problem. It is worth mentioning that the body of constant width given by $\frac12(\Delta^{cc} + \Delta^c)$, for a simplex $\Delta$ of edge-length $1$, is not a minimizer for the  Blaschke-Lebesgue problem (except in dimension $n=2$, in which this construction gives the Reuleaux triangle) and in fact some standard estimates for its volume can be made, showing that its volume is {\em not} exponentially smaller than $\vol(\frac12 B_2^n)$.

The Blaschke-Lebesgue problem is equivalent to finding a sharp constant replacing $\binom{2n}{n}$ in the  Rogers--Shephard inequality \cite{RS} 
$	\vol(K-K) \le \binom{2n}{n}\vol(K)$, for $K$ of constant width $1$,
since in this case the left hand side equals $\vol(B_2^n)$. In dimension $n=3$ the problem have several equivalent formulations, one of which is to find the bodies maximizing the mixed volume $V(K,K,-K)$, which is called the (first) Godbersen coefficient.
Indeed, since for bodies of constant width $1$ we have 
\begin{eqnarray*}
   \vol(B_2^n)& =&  \vol(K-K) = 2\vol(K) + 6V(K,K,-K),\\
   \vol(\partial K) & = & 3\vol(B_2^n, K, K)  =  3\vol(K) + 3\vol(K,K,-K),\\&=& 3\vol(K) + (\vol(B_2^n)- 2\vol(K))/2 = 2\vol(K)+ \vol(B_2^n)/2,
\end{eqnarray*}
the Blaschke-Lebesgue problem is also equivalent to minimizing surface area among all bodies of constant width.

Based on Section \ref{sec:isop-ineq} we can give  an elementary  lower bound for the Blaschke-Lebesgue constant,  which is worse than the one from \cite{schramm1988constantwidth}.
In Corollary \ref{cor:inrad-lower-bound-for-CW-bodies}   we showed constant width bodies satisfy $\inrad(K)\ge 1 -\sqrt{\frac12}\approx 0.293$, showing
\begin{cor}\label{cor:Schramm-type-bound}
If $K$ is a body of constant width $1$, then 
\[
0.586\approx
2-\sqrt{2} \le
\left( \frac{\vol(K)}{\vol(\frac12 B_2^n)}\right)^{\frac1n}.
\]
\end{cor}
\noindent Recall that the currently best known lower bound is $\sqrt{3}-1\approx 0.732$, and was given by Schramm \cite{schramm1988constantwidth}, the argument of which (up to making some simple volume estimate) we have reproduced in the proof of Theorem \ref{thm:Schramm-general}. It is far from the upper bound given by Arman et.~al.~in \cite{arman2024small}.

 We mention also that there are several results in the literature which show that minimizers for the Blaschke-Lebesgue problem must have some special form. For example, sufficiently smooth points must have maximal curvature equal to $1$  (see \cite{Chakerian}) and should be ``tubular'' (see \cite{ShiohamaTakagi1970} for the definition). 

\subsubsection{Illuminating bodies of constant width}

The Boltyanskii-Hadwiger illumination conjecture, dating back to 1957 (see \cite{Hadwiger1957, Hadwiger1960, Boltyanski1960}, asks if the illumination number of every convex body $K\subset \RR^n$ is bounded by $2^n$ with equality only for parallelopipeds. (There is an equivalent formulation using covering numbers instead, called the Levi-Hadwiger conjecture.) The illumination number of a body $K$ is the minimal number $N$ of exterior points $p_1,\dots,p_N$ such that every boundary point of $K$ is {\em illuminated} by at least one of the points, where $x$ is ``illuminated'' by $p$ if the line through $p$ and $x$ intersects the interior of $K$, at a point not in $[x,p]$. Schramm proved the illumination conjecture for bodies of constant width in dimension $16$ and above \cite{Schramm1988illum}. Bezdek conjectured that for $K\in \S_n,\, I(K)\le (2-\varepsilon)^n$ for some positive $\varepsilon$, see \cite{bezdek2008kneser, Bezdek2012} (again, the best known bound is currently due to Schramm \cite{Schramm1988illum}).

\subsubsection{Basins of bodies of constant width}

It is of interest to understand which bodies are averaged to a give body of constant width, since by Brunn-Minkowski's inequality if $T_1+ T_2 = 2K$ this gives a lower bound on the volume of $K$, and for example if all three are different bodies  of constant width then $K$ cannot be a minimizer in the Blaschke-Lebesgue problem. This motivates the following definition. 
For a convex body $K\subset \RR^n$ of constant width $1$ define its ``basin''
\[ \bas(K)  = \{ T\in \S_n: T+T^c = 2K
\}.\]
\begin{lem}
	Let $n\ge 2$. Then $\bas(\frac12 B_2^n) = \{ T\in \S_n: T = -T\}$. 	
\end{lem}
\begin{proof}
	Note that  by Proposition \ref{prop:KminusK} for $T\in \S_n$
	\[ T -T^c = B_2^n.\]
	This means that $T + T^c = B_2^n$ is and only if $T^c = -T^c$ which happens for a set $T\in \S_n$ if and only if $T= - T$. 
\end{proof}

In fact, much more holds true. 
\begin{lem}\label{lem:basinR}
	Let $n\ge 2$ and let $K\subset \RR^n$ be a body of constant width $1$. Then 
	\[ \bas(K) = \{ T\in \S_n: h_K - h_T{\rm ~~is ~even }\}.\]
\end{lem}
\begin{proof}
	For any $T\in \S_n$ we decompose $h_T|_{S^{n-1}} = \frac12 + f + g$ where $f$ is even and $g$ is odd (on $S^{n-1}$).
	We use Proposition \ref{prop:KminusK} stating $T-T^c = B_2^n$ so that $h_{-T^c} = 1- h_T$ which means 
	\[ h_{T^c}(u) = h_{-T^c} (-u) = 1 - (\frac12 + f(-u) + g(-u)) = \frac12 -f(u) + g(u).\] 
	Therefore
	\[h_{\frac{T+T^c}2} = \frac12 (h_T + h_{T^c}) = \frac12 +g. \]
Since $g$ is odd, we see once again that $\frac{T+T^c}{2}$ is a body of constant width. For $T$ to saisfy $\frac{T+T^c}{2} = K$, namely $\frac12 +g = h_K$, we see that the odd part of $g$ should equal to the odd part of $h_K$, which happens if and only if $h_T - h_K = f$ is even, as claimed.

\end{proof}

\begin{cor}
	Let $R\subset \RR^2$ denote Reuleaux triangle, given by the $c$-hull of the points $(1/\sqrt{3},0), (-1/2\sqrt{3}, 1/2), (-1/2\sqrt{3}, -1/2)$ (or equivalently by the $c$-dual of this triplet).   Then $\bas (R) = R$. 
\end{cor}

\begin{proof}
	Assume $ R = (K 	+ K^c) /2$, then by  Lemma \ref{lem:average-cap} the boundaries of both $K$ and $K^c$ include translates of the 
     1-arcs on the boundary of $R$. Since $K^c$ possesses a 1-arc, this implies that $K$ possesses a vertex with normals which are opposite to the normals of the 1-arc. But in this way we trace all the normals in $S^1$, implying that  these 1-arcs meet at vertices, and both $K$ and $K^c$ are translates of $R$. Since $(R+y)^c = R+y$, the proof is complete. 	
\end{proof}

 \begin{rem}
     The exact same proof applies for other self-dual $c$-polytopes in $\RR^2$. 
 \end{rem}

\subsection{An application to the intersection of $1$-lenses}\label{sec:application-carambula}

For reasons which are out of the scope of this paper, the authors were led to examine the question of the intersection of two $1$-lenses in $\RR^n$, and how it is affected by a specific perturbation. To describe the setting, assume one is given 
two translates of a $1$-lens which intersect. Without loss of generality, one may always choose the origin so that the $1$-lenses considered are given by $\pm\cconv([u_0,u_1])$ for some $u_0, u_1 \in \RR^n$.

We then consider the following perturbation of the vertices. For a vector $z\in\R^n$, we shift the vertices $\pm u_1$ by $z$ and shift the vertices $\pm u_0$ by $-z$.
As far as we could see, it  seems non-trivial to show, straightforwardly, that in such a case, considering the ``skewed'' picture of $L_1 = \cconv([u_0-z,u_1+z])$ and $L_2 = - \cconv([u_0+z,u_1-z])$, these also must intersect. We show that this is the case, by applying  of Lemma \ref{lem:CKT)}.

\begin{prop}\label{prop-1-MONTH-Car-INTERSECT}
Let $u_0,u_1,z\in\R^n$ such that $0\in \cconv[u_0,u_1]$. Then
	\[
	\cconv[u_1 + z, u_0 - z] \cap \cconv[- u_1 + z, -u_0 - z]
	\neq \emptyset. 
	\]
\end{prop}
\begin{proof}
Note that if $\|u_1-u_0\|_2>2 $  then by convexity at least one of the vectors $u_1-u_0 \pm 2z$ is longer than $2$ and the $c$-hull of one of the two segments is all of $\RR^n$. We may thus assume $\|u_1-u_0\|_2\le 2 $.

	Let $K = B(u_1 + z) \cap B(u_0 - z) =  [u_1 + z, u_0 - z]^c$, and $T = B(u_1 - z) \cap B(u_0 + z) =  [u_0+z,u_1-z]^c$. Let $\Delta u = u_1 - u_0$.
	
	Note that $K - u_1 = B(z) \cap B(-\Delta u - z)$ and $u_1 - T = B(z) \cap B(\Delta u - z)$, so that in the notations of the proof of Lemma \ref{lem:CKT}, $u_1\in C(K,T)$, meaning $u_1 \in \frac12 (K^c+T^c)$. Similarly $K - u_0 = B(z + \Delta u) \cap B(-z)$ and $u_0 - T = B(z-\Delta u) \cap B(-z)$ and thus $u_0\in \frac12 (K^c+T^c)$. Using convexity of $\frac12 (K^c+T^c)$ and that $0\in \cconv[u_0, u_1]$ we get that 
 $0\in  \cconv[u_0, u_1] \subseteq C(K,T)$, which means in particular that there exists a unit ball containing both $K$ and $-T$, namely $K^c$ intersects $-T^c$. This in turn means precisely that $\cconv[u_0+z,u_1-z]$ intersects $\cconv[- u_1 + z, -u_0 - z]$.
\end{proof}

\begin{rem}
    It remains unclear, however, if as a function of $z\in \RR^n$, the intersection satisfies some convexity property (for example with respect to volume). 
\end{rem}

\section{Appendix - some special $c$-class bodies}\label{appendix-examples}

\subsection{On $k$-lenses}\label{sec:on-k-lenses}

We recall the definition of a $k$-lens (Definition \ref{def:klens} from Section \ref{sec:first}). We use $G_{n,k}$ to denote the Grassmannian, namely $k$-dimensional subspaces of $\RR^n$.

\begin{definition*} 
	Let $n\ge 2$, $1\le k\le n$, let $E\in G_{n,k}$, let $d\in [0,1]$ and let $x\in \RR^n$. The $k$-lens about $x$ of ``radius'' $d$ is defined to be $A^{cc}$ for $A = S(x,d) \cap (x+E)$, and is denoted by 
	$
	L_k(x,E,d)$.  
\end{definition*}

\begin{lem}\label{lem:k-lens-properties}
Let $n\ge 2$, $1\le k \le n$, let $E\in G_{n,k}$, let $d\in [0,1]$ and let $x\in \RR^n$. The $k$-lens $L_k(x,E,d)$ has out-radius $d$ and in-radius $1-\sqrt{1-d^2}$. In particular, these do not depend on $k$ but only on $d$.
If $E_k\in G_{n,k}$ and $E_{k+1}\in G_{n,k+1}$ satisfy $E_k\subset E_{k+1}$ then $L_k(x,E_k,d)\subset L_{k+1}(x, E_{k+1},d)$.
\end{lem}

\begin{proof}
We may clearly assume $x = 0$ and $E = \R^k\times\{0\}\subset \R^n$.  Denote $S(0,1)\cap E= S^E$. 
First, by \eqref{eq:out-and-outc} we know $\outrad(dS^E)^{cc} = \outrad(dS^E) = d$.  The  in-radius can be computed directly but we omit this calculation as it will follow from the next lemma. Monotonicity of the $c$-hull completes the proof. 
\end{proof}

As mentioned in Section \ref{sec:first}, $k$-lenses are $c$-dual to $(n-k)$-lenses, and we show this now, together with some other representations for a $k$-lens. 

\begin{lem}\label{lem:k-lens-dual-to-n-k-lens}
Let $n\ge 2$, $1\le k \le n$, $E\in G_{n,k}$, $d\in [0,1]$ and  $x_0 \in \RR^n$.
The $c$-dual of the $k$-lens $
L_k(x_0, E,d)$ is the $(n-k)$-lens $L_{n-k}(x_0, E^\perp,\sqrt{1-d^2})$. Moreover, we have the following formula
	\begin{equation}\label{eq:carambula-k}
		L_k(x_0, E,d) = x_0 + \{ x: \|x\|_2^2 + 2\sqrt{1-d^2} \|P_{E^{\perp}}x\|_2   \le d^2 \}.
	\end{equation}
\end{lem}

\begin{proof}
First we compute the $c$-dual of a disk $dS^E$, for $E = \R^k\times\{0\}\subset \R^n$ and $x_0 = 0$. The $c$-dual consists of all points $(x,y)\in E \times E^\perp $ such that $B((x,y),1) \supseteq dS^E$, i.e. 
\[
\|x-dw\|_2^2 + \|y\|_2^2 \le 1 \quad \forall w\in S^{E}. 
\]
Since for a fixed $x$, the maximal value of $\|x-dw\|_2$ over $w\in S^E$ is attained when   $w= -x/\|x\|_2$, the condition $(x,y)\in (dS^E)^c$ is equivalent to $(\|x\|_2+d)^2 + \|y\|_2^2 \le 1$, namely
\begin{equation}\label{eq:dual-to-k-lens}
\left(L_k(0,E,d)\right)^c
=(dS^E)^c =
	\left\{ (x,y)\in E \times E^\perp\,\,:\,\,  (\|x\|_2+d)^2 + \|y\|_2^2 \le 1
	\right\}. 
\end{equation}
	 
Next we compute the convex hull of the disk $RS^E$.  
To this end we need to intersect all of the balls $B(z,1)$ where $z \in (RS^E)^c$, namely where $z = (x,y)\in E\times E^\perp$ and 
$(\|x\|_2+R)^2+\|y\|_2^2 \le 1 $. 
We  see that 
\begin{eqnarray*}
(RS^E)^{cc} = \{ (u,v)\in E\times E^\perp :&&  (\|u\|_2 + a)^2    + (\|v\|_2+b)^2 \le 1\\ && \forall (a,b) \in (\R^+)^2 ~{\rm s.t.}~ (a+R)^2 + b^2 \le 1   \}.
\end{eqnarray*}
We claim that it is enough to require the condition on the right hand side for the pair $(a,b) = (0, \sqrt{1-R^2})$. Indeed, assume $(u,v)\in E\times E^\perp$ satisfies 
\[ 
\|u\|_2^2 +( \|v\|_2 + \sqrt{1-R^2})^2\le 1 
\]
and let $(a,b)\in (\RR^+)^2$
satisfy $(a+R)^2 + b^2 \le 1$.  Adding these two equations we see that 
\[ (a+\|u\|_2)^2 + (b+\|v\|_2)^2 +2a(R - \|u\|_2) + 2\|v\|_2 (\sqrt{1-R^2} -b)\le 1. \]
Since clearly $\|u\|_2\le R$ and $b\le \sqrt{1-R^2}$, we get that 
\[ (a+\|u\|_2)^2 + (b+\|v\|_2)^2  \le 1,  \]
as required. 
We have thus shown that 
\begin{eqnarray*}
(RS^E)^{cc} &=& \{ (u,v)\in E\times E^\perp :  \|u\|_2^2    + (\|v\|_2+\sqrt{1-R^2})^2 \le 1  \}.
\end{eqnarray*}
Comparing this equation with  \eqref{eq:dual-to-k-lens} we see that 
$\left(L_k(0,E,d)\right)^c=
L_{n-k}(0, E^\perp,\sqrt{1-d^2})$, proving the duality of a $k$-lens and an $(n-k)$-lens claimed in Lemma \ref{lem:k-lens-dual-to-n-k-lens}.

To prove \eqref{eq:carambula-k} we need only notice that 
\begin{eqnarray*}
(RS^E)^{cc} &=& \{ (u,v)\in E\times E^\perp :  \|u\|_2^2    + (\|v\|_2+\sqrt{1-R^2})^2 \le 1  \}\\
& = & \{ (u,v)\in E\times E^\perp :  \|u\|_2^2    + \|v\|_2^2+2\sqrt{1-R^2}\|v\|_2 + 1 - R^2  \le 1  \}\\
& = & \{ (u,v)\in E\times E^\perp :  \|(u,v)\|_2^2  +2\sqrt{1-R^2}\|v\|_2 \le  R^2 \}. 
\end{eqnarray*}
This completes the proof. 	
\end{proof}

\begin{proof}[Completing the proof of Lemma \ref{lem:k-lens-properties}]
The in-radius $r$ of the $k$-lens $
L_k(x_0, E,d)$ satisfies by Lemma \ref{lem:inplusout} that $1-r$ is the out-radius of its dual $L_{n-k}(x_0, E^\perp,1-\sqrt{1-d^2})$, so that $r = 1-\sqrt{1-d^2}$. 
\end{proof}

The volume of a $k$-lens $L_k(x_0, E,d)$ can be computed, and we next show that it is a convex function in $d$, the special case $k=1$ of which was mentioned above in Proposition \ref{prop-vol-of-carambula-is-convex}. 

\begin{lem}\label{lem:}
Let $n\ge 2$, $1\le k\le n-1$, $E\in G_{n,k}$, $d\in [0,1]$ and  $x_0 \in \RR^n$. Then denoting $f(n, k,d) = \vol(L_k(x_0, E,d))$ we have
\[ 
f(n, k, d) = k \kappa_k \kappa_{n-k} \int_0^d \left(\sqrt{1-s^2} - \sqrt{1-d^2} \right)^{n-k}s^{n-k}ds. 
\]
The function $f(n,k,\cdot)$ is a convex on $[0,1]$.  
\end{lem}

\begin{proof}
We may of course assume $x_0 = 0$ and $E = \RR^k \subset \RR^k \times \RR^{n-k} = \RR^n$. Since $L(0, \R^k, d)$ is a body of revolution, we can compute its volume using Fubini's theorem, using, say, the representation \eqref{eq:dual-to-k-lens} to get  
 \[ L_k(0,\RR^k,d) = 
	\left\{ (x,y)\in \RR^{k} \times \RR^{n-k}\,\,:\,\,  \|x\|_2^2+ (\|y\|_2+\sqrt{1-d^2})^2   \le 1
	\right\}. 
\]
Therefore, using polar integration, 
\begin{eqnarray*}
\vol (L(0, \R^k, d)) & = & \int_{B_k(0,d)} \vol(\{ y: \|y\|_2 + \sqrt{1-d^2}\le \sqrt{1-\|x\|_2^2}\})dx\\
& = & \int_{B_k(0,d)} \kappa_{n-k}  ( \sqrt{1-\|x\|_2^2} - \sqrt{1-d^2})^{n-k}dx\\
& = & k\kappa_k \kappa_{n-k} \int_0^d  ( \sqrt{1-s^2} - \sqrt{1-d^2})^{n-k}ds.
\end{eqnarray*} 
Since the functions are bounded and monotone, we can differentiate under the integral sign and we see that 
\begin{eqnarray*}
\frac{1}{k \kappa_k \kappa_{n-k}}\frac{\partial f}{\partial d}(n, k, d) &=&\frac{\partial  }{\partial d}\int_0^d \left(\sqrt{1-s^2} - \sqrt{1-d^2} \right)^{n-k}s^{n-k}ds\\
& = & (n-k)\frac{d}{\sqrt{1-d^2}} \int_0^d \left(\sqrt{1-s^2} - \sqrt{1-d^2} \right)^{n-k-1}s^{n-k}ds \\
&=& (n-k)\frac{d}{\sqrt{1-d^2}}  \frac{1}{(k-1) \kappa_{k-1} \kappa_{n-k}}f(n-1, k-1, d).
\end{eqnarray*}
Using this recursively we get 
\begin{eqnarray*}
&& \frac{1}{k \kappa_k}\frac{\partial^2 f}{\partial d^2}(n, k, d) = 
 (n-k)\frac{\partial }{\partial d }\left(\frac{d}{\sqrt{1-d^2}}  \frac{1}{(k-1) \kappa_{k-1}}f(n-1, k-1, d)\right)\\  
 && =
  (n-k)\left(
     \frac{({1-d^2})^{-3/2}}{(k-1) \kappa_{k-1}}f(k-1, d) 
  + \frac{d^2}{{1-d^2}} 
 \frac{(n-k-1)}{(k-2)\kappa_{k-2}}f(n-2,k-2,d)
  \right)
\end{eqnarray*}
   As all expressions are non-negative, the second derivative is non-negative and the function is convex.  
\end{proof}

Since $1$-lenses are the analogues of segments in classical convexity, it is useful to point out some of their basic properties. In particular, for a point $x$, to be included in such a ``$c$-convex segment'' is simply a question of the angle $x$ generates with the vertices of the segment.

\begin{lem}\label{lem-theta0-in-car}
	Let $n \ge 2$ and $x, x_0,x_1\in\R^n$ with $\|x_1 - x_0\|_2\le2$ and $x\not\in[x_0,x_1]$. Then $x\in [x_0,x_1]^{cc} $ if and only if
	\begin{equation}\label{eq-theta-0}
		\theta_0\le \theta :=\sphericalangle x_0xx_1  
	\end{equation}
	where $\theta_0 \in [\pi/2, \pi]$ is the angle satisfying 
	$  \sin(\theta_0) = \|x_1 -x_0\|_2/2$. 
\end{lem}

\begin{proof}
	We consider the $2$-dimensional affine space containing $x,x_0,x_1$. The intersection of $ [x_0,x_1]^{cc} $ with this subspace is an intersection of two disks, denoted $C$. The boundary of $C$ consists of two circular $1$-arcs meeting at $x_0$ and at $x_1$. The angle $\sphericalangle x_0yx_1$ is constant for $y\in \partial C$, since it is the angle opposite a chord in a circle of radius $1$. Clearly the angle is greater for points in the interior of $C$ and smaller outside of $C$. So we are left with computing this critical angle $\theta_0$, and we may choose $y$ to be the midpoint of one of the two $1$-arcs, assume it is part of a circle centered at the origin $O$. Note that $\|O-x_1\|_2=\|O-y\|_2=\|O-x_0\|_2=1$ so that the angles of the quadruple $O,x_0,y,x_1$ are $\theta_0/2, \theta_0, \theta_0/2, 2\pi-2\theta_0$. In particular, $\|x_1 - x_0\|_2/2 = \sin(\pi \theta_0) = \sin(\theta_0)$ which completes the proof.
\end{proof}

An equivalent description for the boundary of $ [x_0,x_1]^{cc} $ can be given in terms of the distance of a point to each of its vertices, and to the line connecting them.
\begin{thm}\label{thm-apolonius}
Let $x_0,x_1,y\in \R^n$ with $\|x_1-x_0\|_2 \le 2$. Denote  $a=\|y-x_0\|_2,\,b=\|y-x_1\|_2$, and let $h$ denote the distance between $y$ and the segment $[x_0 x_1]$. Then $z\in\partial  [x_0,x_1]^{cc} $ if and only if
$
2h=ab$.
In particular, if $y\in [x_0,x_1]^{cc} $ then $2h\le ab$.
\end{thm}
\begin{proof}
	The area of the triangle $\triangle(y,x_0,x_1)$ is given by $h\|x_1-x_0\|_2/2$ and also $ab\sin(\theta_0)/2$ and by the previous Lemma this completes the proof. 
\end{proof}

\subsection{Simplex-induced sets}

Another natural family of bodies to consider in $\S_n$ are those related to the standard simplex 
$\Delta_n\subset \RR^n$, the simplex of side-length $1$. While $\Delta_n$ itself does not belong to $\S_n$, we can associate with it three bodies in the class: $\Delta_n^{cc}, \Delta_n^c$ and  
$\frac12(\Delta_n^{cc}+\Delta_n^c)$ which is of constant width $1$. Incidentally, in dimension $n=2$ these three bodies coincide and are the Reuleaux triangle described in Corollary \ref{lem:basinR} (up to rotation). 
Since $\diam(\Delta_n) = 1$, we have as inclusion 
\[ \Delta_n^{cc}\subseteq \frac12(\Delta_n^{cc}+\Delta_n^c)\subseteq 
\Delta_n^c,  \]
and it is not hard to check that the boundaries of these three bodies intersect in $(n+1)$ regions (which are parts of caps $v_i + S^{n-1}$ where $v_i$ are the vertices of $\Delta_n$, and include the vertices themselves which belong to all of them. 

The diameter of $\Delta_n$, $\Delta_n^{cc}$ and also of $\frac12(\Delta_n^{cc}+\Delta_n^c)$
is equal to $1$ and each has  out-radius which is equal to 
\[ \outrad(\Delta_n) = \sqrt{\frac{n}{{2}(n+1)}}.\] 
This out-radius is also shared with $\Delta_n^c$, however its diameter is greater than $1$. 
The in-radii of the three bodies $\Delta_n^{cc}, \Delta_n^c$ and  
$\frac12(\Delta_n^{cc}+\Delta_n^c)$ are also equal (using the relation in Lemma \ref{lem:inplusout}) and we have  
\[ \inrad(\Delta_n^{cc}) = \inrad(\Delta_n^{c}) = 1 - \frac{n}{\sqrt{2(n+1)}}, \]
whereas the in-radius of $\Delta_n$ itself is much smaller, of course, 
\[ \inrad(\Delta_n) = \frac{1}{\sqrt{2n(n+1)}}.\]

These bodies are natural candidates to be extremizers of some isoperimetric-type inequalities in $\S_n$, since $\Delta_n$ itself is an extremizer, or a conjectures extremizer, for many comparison problems in geometry. Nevertheless, as we have already mentioned, the body $\frac12(\Delta_n^{cc}+\Delta_n^c)$, which is of constant width $1$, cannot be the minimizer in the Blaschke-Lebesgue problem, for example. Indeed, in dimension $3$ it is easy to check that this body is the Minkowski average of the two (essentially different) Meissner bodies, which are themselves of equal volume, so by the Brunn-Minkowski (together with its equality case), this body has volume strictly larger than that of the Meissner bodies.

We end this section with an unrelated remark about a simple minimization calculus problem we used in Corollary \ref{cor:mahler-plane}.   
\begin{rem}\label{rem:minimization-ex}
  We include, for completeness,  the proof that $g(x) = \sqrt{x-\sin(x)} + \sqrt{\pi-x-\sin(x)}$ has a unique minimum for $x\in (0,\pi)$ at $x=\pi/2$, which  is a calculus exercise. Let $f(x) = \sqrt{x-\sin(x)}$ (and it is clearly increasing) so that $g(x) = f(x) + f(\pi-x)$. Note that $f^2(x) + f^2(\pi-x) = \pi-2\sin(x)$, and differentiating both sides we see
    \[ 2f(x)f'(x) - 2f(\pi-x)f'(\pi-x)= 2\cos(x), \quad {\rm i.e.}\quad f'(\pi-x) = \frac{f(x)f'(x) - \cos(x)}{f(\pi-x)}.
    \]
We claim $g'(x) = 0$ only when $x = \pi/2$. Indeed, 
\begin{eqnarray*}
    g'(x) & = & f'(x) - f'(\pi - x) = f'(x) - \frac{f(x)f'(x) - \cos(x)}{f(\pi-x)}\\
    & = & f'(x) \left( 1 - \frac{f(x)}{f(\pi-x)} \right) +\frac{\cos(x)}{f(\pi-x)}.
\end{eqnarray*}
For $x = \pi/2$ this is clearly $0$, for $x<\pi/2$ both expressions are positive (since the volume of a lens is monotone in the angle) and for $x>\pi/2$
 both expressions are negative, which shows that $\pi/2$ is the unique point where $g'(x) = 0$, therefore the minimum must be obtained there. \end{rem}

 \noindent 
	 {\bf Acknowledgements:}   Support for this work was   provided by the ERC under the European Union’s Horizon 2020 research and innovation programme (grant agreement no. 770127), by ISF grant Number 784/20, and by the Binational Science Foundation (grant no. 2020329).

\bibliographystyle{plain}

{\small
\noindent S. Artstein-Avidan, 
\vskip 2pt
\noindent School of Mathematical Sciences, Tel Aviv University, Ramat
Aviv, Tel Aviv, 69978, Israel.\vskip 2pt
\noindent Email: shiri@tauex.tau.ac.il.
\vskip 2pt
\noindent D.I. Florentin, 
\vskip 2pt
\noindent Department of Mathematics, Bar-Ilan University, Ramat Gan,  52900, Israel.   \vskip 2pt
\noindent Email: dan.florentin@biu.ac.il.
}

\end{document}